\documentclass{article}

\pdfoutput=1

\usepackage[colorlinks,citecolor=blue,linkcolor=blue,urlcolor=blue]{hyperref}
\usepackage{amsmath,amssymb,amsthm,dsfont,enumerate,float,mathrsfs,stmaryrd,upref,xcolor}
\usepackage[inline]{asymptote}
\usepackage[ngerman,main=american]{babel}\addto\extrasamerican{\languageshorthands{ngerman}\useshorthands{"}}
\usepackage[tmargin=2.5cm,lmargin=2.5cm,rmargin=2.5cm,bmargin=2.5cm]{geometry}

\newcommand\1{{\mathds{1}}}
\renewcommand\ae{{a.\@e.\@}}
\newcommand\B{\mathrm{B}}

\newcommand\Bo{\mathcal{B}}
\newcommand\BV{\mathrm{BV}}
\newcommand\BVs{\mathcal{BV}}
\newcommand\C{\mathrm{C}}
\renewcommand\c{\mathrm{c}}
\newcommand\cd{\boldsymbol{\cdot}}
\newcommand*{\coleq}{\mathrel{\vcenter{\baselineskip0.5ex\lineskiplimit0pt\hbox{\normalsize.}\hbox{\normalsize.}}}=}
\newcommand*{\colequiv}{\mathrel{\vcenter{\baselineskip0.5ex\lineskiplimit0pt\hbox{\normalsize.}\hbox{\normalsize.}}}\equiv}
\newcommand*{\coliff}{\mathrel{\vcenter{\baselineskip0.5ex\lineskiplimit0pt\hbox{\normalsize.}\hbox{\normalsize.}}}\iff}
\newcommand\Cp{\mathrm{Cap}}
\newcommand\cpt{\mathrm{cpt}}
\newcommand\D{\mathrm{D}}
\DeclareMathOperator\inn{int}

\renewcommand\d{\mathrm{d}}
\DeclareMathOperator\dcup{\dot{\cup}\,}
\newcommand\dist{\mathrm{dist}}
\DeclareMathOperator\Div{div}

\newcommand\dx{{\,\d x}}
\newcommand\ecke{\mathop{\hbox{\vrule height 7pt width .3pt depth 0pt \vrule height .3pt width 5pt depth 0pt}}\nolimits}
\newcommand\eps{\varepsilon}
\newcommand\ka{\textup{(}}
\newcommand\kz{\textup{)}}
\renewcommand\L{\mathrm{L}}
\newcommand\loc{\mathrm{loc}}
\newcommand\M{\mathcal{M}}
\newcommand\mutrace{\mathrm{Tr}_E(\sigma)}
\renewcommand\H{{\mathcal{H}}}
\newcommand\Ln{{\mathcal{L}^n}}
\newcommand\N{{\mathds{N}}}
\newcommand\nortrace{\sigma{\cd}\nu_E}
\newcommand\p{\varphi}
\renewcommand\P{\mathrm{P}}
\renewcommand\p{\varphi}
\newcommand\Pmu{\mathscr{P}_{\mu_+,\mu_-}}
\newcommand\qq{\qquad}
\newcommand\R{{\mathds{R}}}

\renewcommand\r{\varrho}

\DeclareMathOperator\spt{spt}
\newcommand\st{s.\@t.\@}
\renewcommand\t{\vartheta}
\newcommand\ts{\textstyle}
\newcommand\W{\mathrm{W}}

\renewcommand\paragraph[1]{\noindent\textbf{#1.}}

\newtheorem{thm}{Theorem}[section]
\newtheorem{cor}[thm]{Corollary}
\newtheorem{defi}[thm]{Definition}
\newtheorem{exmp}[thm]{Example}
\newtheorem{lem}[thm]{Lemma}
\newtheorem{prop}[thm]{Proposition}
\newtheorem{rem}[thm]{Remark}

\numberwithin{equation}{section}

\begin{document}

\title{Isoperimetric conditions, lower semicontinuity, and existence results for
  perimeter functionals with measure data}

\author{
  Thomas Schmidt\footnote{Fachbereich Mathematik,
    Universit\"at Hamburg, Bundesstr. 55, 20146 Hamburg, Germany.\newline
    Email address: \href{mailto:thomas.schmidt@math.uni-hamburg.de}
      {\tt thomas.schmidt@math.uni-hamburg.de}.
    URL: \href{http://www.math.uni-hamburg.de/home/schmidt/}
      {\tt http:/\!/www.math.uni-hamburg.de/home/schmidt/}.
    }
  }

\date{October 25, 2024}
  
\maketitle

\begin{abstract}
  We establish lower semicontinuity results for perimeter functionals with
  measure data on $\R^n$ and deduce the existence of minimizers to these
  functionals with Dirichlet boundary conditions, obstacles, or
  volume-constraints. In other words, we lay foundations of a perimeter-based
  variational approach to mean curvature measures on $\R^n$ capable of proving
  existence in various prescribed-mean-curvature problems with measure data. As
  crucial and essentially optimal assumption on the measure data we identify a
  new condition, called small-volume isoperimetric condition, which sharply
  captures cancellation effects and comes with surprisingly many properties and
  reformulations in itself. In particular, we show that the small-volume
  isoperimetric condition is satisfied for a wide class of $(n{-}1)$-dimensional
  measures, which are thus admissible in our theory. Our analysis includes
  infinite measures and semicontinuity results on very general domains.
\end{abstract}

\tableofcontents

\section{Introduction}

\paragraph{Prescribed mean curvature hypersurfaces and Massari's functional}
This paper contributes to the theory of (generalized) hypersurfaces of
prescribed mean curvature in $\R^n$, approached from a parametric
calculus-of-variations side. Given a function $H\in\L^1(\Omega)$ on
an open set $\Omega\subset\R^n$, this amounts to the study of functionals of the
type
\begin{equation}\label{eq:P-H}
  \mathscr{P}_H[A;\Omega]
  \coleq\P(A,\Omega)-\int_{A\cap\Omega}H\dx
  \qq\qq\text{on measurable sets }A\subset\R^n\,,
\end{equation}
where the perimeter $\P(A,\Omega)$ of $A$ in $\Omega$ gives, in sufficiently
regular cases, the $(n{-}1)$-dimensional Hausdorff measure of
$\Omega\cap\partial A$. In order to obtain prescribed mean curvature
hypersurfaces one seeks to minimize the functional
$\mathscr{P}_H[\,\cdot\,;\Omega]$ among sets $A$ of finite perimeter in
$\Omega$, which are usually required to satisfy boundary conditions at
$\partial\Omega$ and possibly further constraints. If a minimizer $A$ with
sufficiently smooth boundary $\Omega\cap\partial A$ exists, at least in cases
with constraints \emph{only} at $\partial\Omega$, it should solve the parametric
prescribed mean curvature equation
\begin{equation}\label{eq:pmc}
  \Div\nu_A=H
  \qq\text{on }\Omega\cap\partial A\,,
\end{equation}
where $\nu_A$ denotes the outward unit normal to $A$ at points of
$\Omega\cap\partial A$ and the divergence can be taken either as the tangential
divergence of $\nu_A$ along $\partial A$ or equivalently as the standard
divergence of any smooth continuation of $\nu_A$ to $\Omega$ as a (sub"~)""unit
vector field. The equation \eqref{eq:pmc}, if valid in a suitably strong sense,
does express that the mean curvature of $\partial A$ is indeed the prescribed
function $\frac{-1}{n-1}H$ --- or more precisely that, for every
$x\in\Omega\cap\partial A$, the number $\frac{-1}{n{-}1}H(x)$ is the mean
curvature at $x$ of the hypersurface $\Omega\cap\partial A$ oriented by $\nu_A$.

A major step in the program described has been achieved by Massari
\cite{Massari74,Massari75} who introduced the approach via the functional
$\mathscr{P}_H[\,\cdot\,;\Omega]$ and extended De Giorgi's pioneering work
\cite{DeGiorgi6061} from the minimal surface case $H\equiv0$ to general
prescribed functions $H$. In fact, the papers \cite{Massari74,Massari75}
establish an existence result for minimizers of
$\mathscr{P}_H[\,\cdot\,;\Omega]$ in case $H\in\L^1(\Omega)$ and also a
minimal-surface-type\footnote{By minimal-surface-type partial regularity we mean
regularity up to an exceptional set of Hausdorff dimension at most $n{-}8$.}
partial $\C^{1,\alpha}$ regularity result under the optimal assumption
that $H\in\L^p_{(\loc)}(\Omega)$ holds for some $p>n$. If $H$ is additionally
continuous, it follows in a standard way (e.\@g.\@ by locally computing the
non-parametric first variation) that minimizers $A$ of
$\mathscr{P}_H[\,\cdot\,;\Omega]$ satisfy \eqref{eq:pmc} on the regular portions
of $\Omega\cap\partial A$ and that $\frac{-1}{n{-}1}H$ is the mean curvature of
$\Omega\cap\partial A$. For discontinuous $H$, in contrast, the geometric
significance of $H$ is far less clear, and in general it seems to be a widely
open problem if and in which precise sense one can still restrict $H$ to
$\Omega\cap\partial A$ and make any sense of equation \eqref{eq:pmc}.

\bigskip

\paragraph{Lower semicontinuity for a Massari-type functional with measures}
In the present paper, though we take the geometric situation as a background
motivation and in fact have some hope for a connection with the open problem
just mentioned, we deal with the minimization of prescribed-mean-curvature
functionals mostly in its own right. In fact, we replace the prescribed function
$H\in\L^1(\Omega)$ with prescribed non-negative Radon measures $\mu_+$ and
$\mu_-$ concentrated on $\Omega$ and possibly of dimension lower than $n$, and
correspondingly we replace Massari's functional \eqref{eq:P-H} with its natural
generalization
\begin{equation}\label{eq:Pmu-intro}
  \Pmu[A;\Omega]
  \coleq\P(A,\Omega)+\mu_+(A^1)-\mu_-(A^+)
  \qq\qq\text{on measurable sets }A\subset\R^n\,,
\end{equation}
where $A^+$ denotes the measure-theoretic closure and $A^1$ the
measure-theoretic interior of $A$ (see Section \ref{sec:prelim} for the
definitions). Our central results on the functional $\Pmu[\,\cdot\,;\Omega]$ are
semicontinuity results, which apply under sharp hypotheses on the measures
$\mu_\pm$ and are suitable to prove the existence of minimizers of
$\Pmu[\,\cdot\,;\Omega]$ in several cases with standard boundary conditions or
constraints. In fact, our semicontinuity statements take slightly different
forms in the full-space case $\Omega=\R^n$ (see Section \ref{sec:Rn}), in
versions adapted to Dirichlet problems on domains $\Omega\subset\R^n$ (see
Section \ref{sec:Dir}), and generally on domains $\Omega\subset\R^n$ (see
Section \ref{sec:dom}). For the purposes of this introduction, we restrict
the detailed discussion to the full-space case and the functional
\[
  \Pmu\coleq\Pmu[\,\cdot\,;\R^n]\,,
\]
for which we introduce the crucial hypothesis on $\mu_\pm$ and state a
prototypical case of our results as follows:

\begin{defi}[small-volume isoperimetric condition]
  We say that a non-negative Radon measure $\mu$ on $\R^n$ satisfies the
  small-volume isoperimetric condition \textup{(}briefly\textup{:} the
  small-volume IC\textup{)} in $\R^n$ with constant\/ $1$ if, for every
  $\eps>0$, there exists some $\delta>0$ such that
  \begin{equation}\label{eq:mu<P+eps-intro}
    \mu(A^+)\le\P(A,\R^n)+\eps
    \qq\text{for all measurable }A\subset\R^n
    \text{ with }|A|<\delta\,.
  \end{equation}
\end{defi}

\begin{thm}[lower semicontinuity on full space; prototypical case]\label{thm:lsc-global-intro}
  Consider non-negative Radon measures $\mu_+$ and $\mu_-$ on $\R^n$ which both
  satisfy the small-volume IC in $\R^n$ with constant\/ $1$. Then the full-space
  functional $\Pmu$ introduced above is finite and lower semicontinuous with
  respect to convergence in measure on
  $\BVs(\R^n)\coleq\{A\subset\R^n\,:\,A\text{ measurable}\,,\,|A|{+}\P(A,\R^n)<\infty\}$.
\end{thm}

We emphasize that, for this and similar semicontinuity results, we necessarily
need to use some closed representative of $A$ in the $\mu_-$-volume term of
\eqref{eq:Pmu-intro}, since measurable sets $A$ are considered in an
$\Ln$-\ae{} sense, and other choices of representative would not ensure lower
semicontinuity of $\Pmu$ along basic strictly decreasing sequences $A_k\searrow
A_\infty$ with $\P(A_k,\R^n)\to\P(A_\infty,\R^n)$, as soon as $\mu_-$ assigns
mass to the boundary of $A_\infty$. Indeed, the usage of $A^+$ as a precise
$\H^{n-1}$-\ae{} defined representative of $A$ is perfectly suited for our
purposes and is inspired by related developments in the theory of one-sided
obstacle problems; cf.\@
\cite{CarDalLeaPas88,SchSch15,BoeDuzSch16,SchSch16,SchSch18,Treinov21}. In
a very similar way, the choice of $A^1$ in the $\mu_+$-volume term allows to
cope with basic increasing sequences $A_k\nearrow A_\infty$.

\bigskip

\paragraph{Lower semicontinuity also on general domains}
Our semicontinuity results for functionals of type \eqref{eq:Pmu-intro} on
general domains $\Omega\subset\R^n$ rely on closely related (small-volume) ICs,
which partially can be understood as relative ICs adapted to the domain at hand.
However, at this introductory stage we will only briefly touch upon some aspects
of the results, while postponing the discussion of the adapted ICs entirely to
the later sections. We mention that basically all results on general domains
will be deduced from the ones on full space by extension/restriction to/from all
of $\R^n$. For cases with a generalized Dirichlet boundary condition on a
bounded domain $\Omega$, this deduction is essentially standard. However, as a technical addition, when
working out the details, we also include a careful treatment of (strongly)
unbounded domains $\Omega$ and infinite measures $\mu_\pm$; see Section
\ref{sec:Dir} for the details. Furthermore, in the final Section \ref{sec:dom},
we also obtain two semicontinuity results on general domains independent of any
boundary condition. The first result is somewhat different from the usual
semicontinuity on open sets and yields lower semicontinuity of a functional
$\Pmu[\,\cdot\,;\Omega^1]$ \emph{on the measure-theoretic interior} $\Omega^1$
of a set $\Omega$ of locally finite perimeter in $\R^n$. This type of
semicontinuity on $\Omega^1$ does not seem to be standard even in case of the
relative perimeter $\mathscr{P}_{0,0}[\,\cdot\,;\Omega^1]=\P(\,\cdot\,,\Omega^1)$
alone, but in the perimeter case is in fact not entirely new and can also be
deduced from a recent result of Lahti \cite{Lahti20}. Anyway, our theory allows
for a new and very natural proof by incorporating the perimeter measure
$\P(\Omega,\,\cdot\,)$ (and potentially even some other measures on the reduced
boundary $\partial^\ast\Omega$) into the measures $\mu_\pm$ of the full-space
functional $\Pmu$. As a complement, the second result gives lower semicontinuity
of $\Pmu[\,\cdot\,;\Omega]$ also on an arbitrary open set $\Omega\subset\R^n$
and thus can dispense with any regularity of $\Omega$ at the price of requiring
openness even in the standard topological sense. Finally, we will also further
underpin the results with several examples of admissible domains and measures
and with a detailed discussion of the relevant (relative) ICs and their
optimality.

\bigskip

\paragraph{The small-volume IC as decisive assumption for semicontinuity}
For now, we return to the full space-setting of Theorem \ref{thm:lsc-global-intro}
and discuss its crucial assumption, the small-volume IC, in some more detail. We
first highlight that this condition is not only sufficient for the lower
semicontinuity conclusion, but in itself expresses lower semicontinuity of the
functional $\mathscr{P}_{0,\mu}$ \emph{at the empty set} and thus in most cases
is also necessary for lower semicontinuity. Indeed, if $\mu=\mu_-$ violates the
small-volume IC in $\R^n$ with constant $1$, for some $\eps>0$ there exists a
sequence of counterexamples in form of measurable sets $A_k\subset\R^n$ with
$\lim_{k\to\infty}|A_k|=0$ and $\mu(A_k^+)>\P(A_k,\R^n)+\eps$. This, however,
means that $A_k$ converge in measure to the empty set $\emptyset$ with
$\limsup_{k\to\infty}\mathscr{P}_{0,\mu}[A_k]\le{-}\eps<0=\mathscr{P}_{0,\mu}[\emptyset]$,
and lower semicontinuity fails as well. Therefore, the small-volume IC with
constant $1$ is in fact the optimal assumption on $\mu_-$ in Theorem
\ref{thm:lsc-global-intro}. Moreover, if $\mu=\mu_+$ is supported in a ball $B$
and $A_k$ are as before, then $B\setminus A_k$ converge in measure to $B$, and
one finds
$\limsup_{k\to\infty}\mathscr{P}_{\mu,0}[B\setminus A_k]\le\mathscr{P}_{\mu,0}[B]-\eps$.
Therefore, at least in case of bounded support, the small-volume IC with
constant $1$ is the optimal assumption on $\mu_+$ as well.

In the proof of Theorem \ref{thm:lsc-global-intro}, the small-volume IC is
decisive in coping with cases in which (the singular part of) $\mu=\mu_-$
has mass on an $(n{-}1)$-dimensional surface $S$ and, for a decreasing sequence
$A_k\searrow A_\infty$, the sets $A_k$ include thinner and thinner neighborhoods of
$S$, while $A_\infty^+$ does not intersect $S$ anymore; see Figure \ref{fig:tentacle}
below for an illustration in case $n=2$. In such situations, with
${-}\mu(A_\infty^+)>\liminf_{k\to\infty}[{-}\mu(A_k^+)]$ the $\mu$-volume term
in $\mathscr{P}_{0,\mu}$ is \emph{not} lower semicontinuous, but it holds the
\emph{strict} inequality $\P(A_\infty,\R^n)<\liminf_{k\to\infty}\P(A_k,\R^n)$.
Under the small-volume IC from \eqref{eq:mu<P+eps-intro} we will show that it is
possible to quantitatively relate these opposite effects, to compensate for the increase of
the $\mu$-volume with the decrease of the perimeter and thus to admit a certain
cancellation effect while still preserving lower semicontinuity of the
functional $\mathscr{P}_{0,\mu}$. The functional $\mathscr{P}_{\mu,0}$ with
$\mu=\mu_+$ can be handled in a dual manner (where the decisive sequences are
the increasing ones), and the results can be combined in order to reach
functionals of the general type $\Pmu$.

\begin{figure}[h]\centering
  \includegraphics{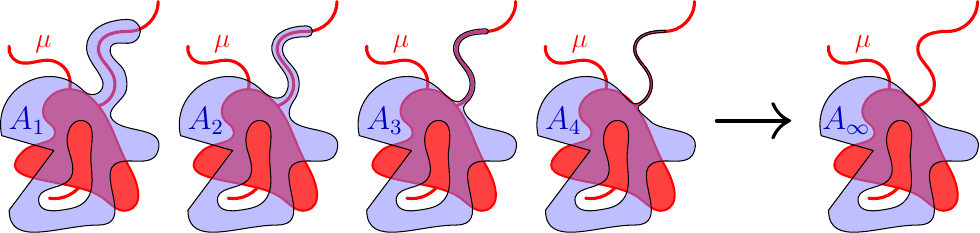}
  \vspace{-2.7ex}
  \caption{An illustration of the decisive cancellation effect in $\R^2$: A sequence
    $(A_k)_{k\in\N}$ forms thinner and thinner tentacles around a $1$d portion of $\spt\mu$,
    but in the limit $A_\infty^+$ does not cover this portion anymore.\label{fig:tentacle}}
\end{figure}

Beside the decisive effect just described, the small-volume IC also has a role
in preventing a breakdown of lower semicontinuity at infinity, which in general
can occur already in the function case $\mu_\pm=H_\pm\Ln$. Indeed, for each
$H\in\L^1(\R^n)$, continuity of the $H$-volume term and thus lower semicontinuity
of $\mathscr{P}_H$ are immediate. However, this does not extend to
$H\in\L^1_\loc(\R^n)$, where for similar reasons as above one needs to prevent
that $A_k$ move away to infinity with $\lim_{k\to\infty}|A_k|=0$,
$\limsup_{k\to\infty}\P(A_k,\R^n)<\infty$, but
$\limsup_{k\to\infty}\int_{A_k}H\dx=\infty$. As our result is formulated for
\emph{locally} finite measures $\mu_\pm$, it also singles out functions
$H\in\L^1_\loc(\R^n)\setminus\L^1(\R^n)$ such that $\mathscr{P}_H$ is lower
semicontinuous. We are aware of previous results in this direction only on
specific unbounded domains in the different setting of \cite{Duzaar93,DuzSte96}
(compare also below), but still consider this aspect mostly as a side benefit of
our treatment of possibly singular measure data.

\bigskip

\paragraph{Existence results}
As standard consequences of semicontinuity we derive existence results for
minimizers of $\Pmu[\,\cdot\,;\Omega]$ with obstacles, prescribed volume, or a
Dirichlet boundary condition as side conditions. Since the obstacle and
prescribed-volume constraints fit into the full-space setting described so far,
we exemplarily state our corresponding existence results at least for the case
of finite $\mu_-$, while the somewhat more technical treatment of Dirichlet
problems is postponed to the later Section \ref{sec:Dir}. In all cases, we
impose the small-volume IC as the decisive assumption on $\mu_\pm$.

\begin{thm}[existence in obstacle and prescribed-volume problems]
    \label{thm:exist-intro}
  Consider non-negative Radon measures $\mu_+$ and\/ $\mu_-$ on $\R^n$ such that
  both $\mu_+$ and\/ $\mu_-$ satisfy the small-volume IC with constant\/ $1$ on
  $\R^n$ and such that\/ $\mu_-$ is finite. Then, with $\BVs(\R^n)$ as in
  Theorem \ref{thm:lsc-global-intro}, we have\textup{:}\vspace{-.5ex}
  {\leftmargini=3.5ex
  \begin{itemize}
  \item[]\textsc{Obstacle problem:}
    Whenever, for given measurable sets $I,O\subset\R^n$, the admissible class
    $\{A\in\BVs(\R^n):I\subset A\subset O\text{ up to negligible sets}\}$
    is non-empty, then there exists a minimizer of\/ $\Pmu$ in this
    class.\vspace{-.5ex}
  \item[]\textsc{Prescribed-volume problem with $\mu_+\equiv0$:}
    For every $v\in{(0,\infty)}$, there exists a minimizer of\/
    $\mathscr{P}_{0,\mu_-}$ in $\{A\in\BVs(\R^n):|A|=v\}$.
  \end{itemize}
  }
\end{thm}

Theorem \ref{thm:exist-intro} will be established in Section \ref{sec:exist},
where existence in the obstacle problem will also be extended to some
infinite measures $\mu_-$, while in the prescribed-volume problem we will not go
beyond the statement given above. The proof uses the direct method in the
calculus of variations and at least in the obstacle case is standard once
suitable semicontinuity is at hand. However, since in the full-space situation
out of a minimizing sequence we can only extract a subsequence which converges
\emph{locally} in measure on all of $\R^n$, we in fact need a semicontinuity
statement adapted to \emph{local} convergence in measure. As we will see in
Section \ref{sec:Rn}, such a variant can be deduced from the above statement of
Theorem \ref{thm:lsc-global-intro} by cut-off arguments. In case of the
prescribed-volume problem, the local-convergence issue additionally brings up
the more severe difficulty that a limit in the sense of local convergence may exhibit a ``volume
drop'' at infinity and thus may fall out of the admissible class. The strategy
for preventing this is technically more involved and consists in constructing an
improved minimizing sequence by ``shifting volume'' into a bounded region;
see Section \ref{sec:exist} for detailed discussion and implementation.

\bigskip

\paragraph{More on the small-volume IC: criteria and exemplary cases}
We further support the semicontinuity and existence results by identifying
wide classes of measures for which the small-volume IC holds. First let us
remark that related ICs without the additive $\eps$-term have been considered
in classical literature (compare also below for related discussion) with the
typical background idea that such conditions can be deduced for
$\mu_\pm=H_\pm\Ln$, $H\in\L^p(\R^n)$, $p>n$, by the classical estimate
via the Hölder and isoperimetric inequalities
$\int_AH_\pm\dx\le C_n\|H\|_{\L^p(\R^n)}|A|^{\frac1n-\frac1p}\,\P(A,\R^n)$,
where $C_n$ is a dimensional constant. As a first indication that our
small-volume IC is substantially different, we record that it is
in fact trivially satisfied, beyond the previous $\L^p$ cases and due to the
$\eps$-term alone, for all finite absolutely continuous measures
$\mu_\pm=H_\pm\Ln$ with $H\in\L^1(\R^n)$. Hence, our semicontinuity results
include Massari's standard case of the functional $\mathscr{P}_H$. In addition,
however, our results do admit singular measures, as will become clear from the
following abstract criterion:

\begin{thm}[divergence criterion for the small-volume IC]
    \label{thm:admis-crit-intro}
  If a non-negative Radon measure $\mu$ on $\R^n$ can be expressed as
  $\mu=H\mathcal{L}^n{+}\Div\sigma$ with $H\in\L^1(\R^n)$ and a
  divergence-measure field $\sigma\in\L^\infty(\R^n,\R^n)$ such that
  $\|\sigma\|_{\L^\infty(\R^n,\R^n)}\le1$, then $\mu$ satisfies the small-volume
  IC in $\R^n$ with constant $1$.
\end{thm}

Theorem \ref{thm:admis-crit-intro} and its proof are not very surprising.
For instance, one may read off the result from a divergence theorem for
$\L^\infty$ divergence-measure fields on sets of finite perimeter (similar to
the later formula \eqref{eq:int-Gauss-Rn}). Alternatively, one can also argue by
approximation, and this is the route we take when picking up the result in the 
somewhat wider context of the later Section \ref{sec:admis-meas}.

For the moment, we mainly record that the condition of Theorem
\ref{thm:admis-crit-intro} holds for infinite measures
$\mu=\theta\H^{n-1}\ecke S$ with $\theta\in{[0,2]}$ and with a hyperplane
$S\subset\R^n$ or a union $S$ of finitely many parallel hyperplanes in $\R^n$.
Thus, we obtain basic examples of singular measures with small-volume IC.
However, the condition remains valid for a much broader class of
$(n{-}1)$-dimensional measures, as in fact we have:

\begin{thm}[small-volume IC for rectifiable $\H^{n-1}$-measures]
    \label{thm:rect-admis-intro}
  Whenever, for a non-negative Radon measure $\mu$ on $\R^n$, we have
  $\mu\le2\H^{n-1}\ecke S$ with some $\H^{n-1}$-finite and countably
  $\H^{n-1}$-rectifiable Borel set $S\subset\R^n$, then $\mu$ satisfies the
  small-volume IC in $\R^n$ with constant $1$.
\end{thm}

Theorem \ref{thm:rect-admis-intro} will be established in Section
\ref{sec:IC-for-P}, where one could in fact take the case
$\mu=2\H^{n-1}\ecke\partial E$ with $\partial E$ Lipschitz as a starting point
and then reach the generality of Theorem \ref{thm:rect-admis-intro} by
covering. However, we prefer directly resolving the case
$\mu=2\H^{n-1}\ecke\partial^\ast\!E$ with the reduced boundary
$\partial^\ast\!E$ of a set $E$ of finite perimeter by a reasoning we
consider interesting in its own right: The argument is based on the construction
of a sub-unit extension $\sigma_E\in\L^\infty(\R^n,\R^n)$ of a unit normal
vector field to $\partial^\ast\!E$ with $\Div\sigma_E\in\L^1(\R^n)$ and then
reads off the condition of Theorem \ref{thm:admis-crit-intro} for
$\mu=2\H^{n-1}\ecke\partial^\ast\!E$ from Gauss-Green formulas which involve
weak normal traces of $\sigma_E$. In fact, for Lipschitz boundaries
$\partial E$ the existence of the field $\sigma_E$ is also guaranteed by trace
theory, while for general reduced boundaries we rely on the theory and
construction of an optimal variational mean curvature $H_E\in\L^1(\R^n)$ of $E$
due to Barozzi \& Gonzalez \& Tamanini \cite{BarGonTam87} and Barozzi
\cite{Barozzi94}, read off a certain auxiliary IC for  $H_E$, and only then
deduce the existence of $\sigma_E$ with $\Div\sigma_E=H_E$.

We postpone most of the more detailed discussion on reformulations and further
properties of ICs to the later sections. However, already at this stage we wish
to mention one specific property of the small-volume IC, which came quite
unexpected, which has a role in proving the general Theorem
\ref{thm:rect-admis-intro}, and which genericly allows to obtain further
examples of measures admissible in our theory from those already discussed:

\begin{prop}[small-volume IC for the sum of singular measures]
    \label{prop:ic-sing-intro}
  Consider non-negative Radon measures $\mu_1$ and $\mu_2$ on $\R^n$ such that
  $\mu_1$ and $\mu_2$ are singular to each other and least one of $\mu_1$ and
  $\mu_2$ is finite. If $\mu_1$ and $\mu_2$ both satisfy the small-volume IC in
  $\R^n$ with constant\/ $1$, then $\mu_1{+}\mu_2$ satisfies the small-volume IC
  in $\R^n$ still with the same constant\/ $1$ \ka and not merely in the evident
  way with an additional multiplicative factor $2$ in front of the perimeter\kz.
\end{prop}

The proof of Proposition \ref{prop:ic-sing-intro} will be given in Section
\ref{sec:admis-meas} and is based on a certain relative-perimeter
characterization of the small-volume IC and an elementary separation argument.

\bigskip

\paragraph{On the usage of ICs and related results in the literature}
To the state of our knowledge, the precise form of our small-volume IC
and its flexibility, as underlined by Theorem \ref{thm:rect-admis-intro}, are
new. Nevertheless, related linear ICs have been around in the theory of
prescribed mean curvature surfaces for a long time, and thus we now comment on
the previous literature in some more detail.

In fact, ICs have been prominently used in the theory of \emph{non-parametric}
prescribed-mean-curvature functionals, which correspond to
$\mathscr{P}_H[A;\Omega]$ from \eqref{eq:P-H} for subgraphs $A$ and
$\Omega=D\times\R$ with a bounded Lipschitz domain $D\subset\R^{n-1}$. However,
the considerations on such functionals in
\cite{BomGiu73,Miranda74,Giaquinta74a,Giaquinta74b,Gerhardt74,Giusti78a} differ
from ours, since e.\@g.\@ the assumptions in \cite{BomGiu73,Gerhardt74}
are essentially (in the terminology of our setting)
$\partial_nH\le0$, $H(\,\cdot\,,0)\in\L^{n-1}(D)$ and the settings of the other
papers tend in similar, but rather more restrictive directions. In any case,
these works exclude cancellation in the previously described sense, and thus the
perimeter and the $H$-volume are even separately lower semicontinuous for basic
reasons and without need for imposing an IC. In fact, in these non-parametric
cases it is not semicontinuity but rather coercivity of the problem which is
obtained from stronger ICs of type
\begin{equation}\label{eq:G-IC}
  \bigg|\int_A H(\overline x,0)\,\d\overline x\bigg|\le C\P(A,\R^{n-1})
  \qq\text{for all measurable }A\subset D\,,
  \qq\text{with fixed }C\in{[0,1)}\,.
\end{equation}
When comparing with our results, the need for assuming \eqref{eq:G-IC} may be
viewed as a result of considering on the unbounded cylinder $D\times\R$ an
infinite measure $H\mathcal{L}^n$, and analogous conditions occur also in our
theory when later addressing the existence issue with infinite measures in
Theorems \ref{thm:ex-obst} and \ref{thm:ex-Dir}. Moreover, in case
$H(\overline x,x_n)=H_0(\overline x)$, having \eqref{eq:G-IC} with $C=1$ is
also necessary for classical solvability of the prescribed mean curvature
equation ${-}\Div\big(\nabla u/\!\sqrt{1{+}|\nabla u|^2}\big)=H_0$ (compare
with \cite{Giusti78a} for finer related discussion). It is not clear to us if
there is an effective necessary condition of a similar type also for general $H$
with $x_n$-dependence.

Still in the non-parametric framework, directions partially analogous to ours
have been pursued in
\cite{CarLeaPas86,CarLeaPas87,Ziemer95,DaiTruWan12,DaiWanZho15}: Indeed,
Carriero \& Leaci \& Pascali \cite{CarLeaPas86,CarLeaPas87} study
semicontinuity and relaxation of non-parametric functionals with certain general
measure terms, where the assumptions of their main semicontinuity result
\cite[Theorem 5.2]{CarLeaPas87}, for instance, have aspects in common with our
small-volume IC. However, the framework is rather different, builds on some more
background notions for measures and capacities, and in detail is difficult to
compare. In any case, we stress that the results in
\cite{CarLeaPas86,CarLeaPas87} concern the non-parametric setting and remain
limited to measure terms of fixed sign and to Sobolev spaces. In particular,
these papers do not discuss a natural $\BV$ framework or any existence result.
Eventually, Ziemer \cite{Ziemer95} gives an existence result for non-parametric
functionals which involve a finite non-negative measure datum $\mu_0$ with
compact support in a bounded Lipschitz domain $D\subset\R^{n-1}$. However, his
central assumption
\begin{equation}\label{eq:Z-IC}
  \mu_0(\B_r(x))\le Cr^\kappa
  \qq\text{for all balls }\B_r(x)\subset D\,,
  \qq\text{with fixed }C\in{[0,\infty)}\text{ and }\kappa\in{(n{-}2,n{-}1)}
\end{equation}
is considerably stronger than a linear IC and in particular excludes the
interesting borderline case of $(n{-}2)$-dimensional measures $\mu_0$. Moreover,
Dai \& Trudinger \& Wang \cite{DaiTruWan12} and Dai \& Wang \& Zhou
\cite{DaiWanZho15} introduce an approximation-based notion of a mean curvature
measure and establish a corresponding existence result for generalized solutions
to the prescribed mean curvature equation on a smooth bounded domain
$D\subset\R^{n-1}$ with a finite signed measure $\mu_0$ on $D$ as right-hand
side. They require that the singular part of $\mu_0$ has compact support in $D$
and in analogy with \eqref{eq:G-IC} impose on $\mu_0$ an IC of type
\begin{equation}\label{eq:DTW-IC}
  |\mu_0(A^1)|\le C\P(A,\R^{n-1})
  \qq\text{for all measurable }A\subset D\,,
  \qq\text{with fixed }C\in{[0,1)}\,.
\end{equation}
Since the settings differ, a comparison of these results with ours is
necessarily incomplete, but one may say that the results in
\cite{Ziemer95,DaiTruWan12,DaiWanZho15} work for product measures
$\mu=\mu_0\otimes\mathcal{L}^1$ on $D\times\R$, while we admit general measures
$\mu$ on $\Omega\subset\R^n$. Alternatively, from a more PDE-based viewpoint,
one may put it the way that \cite{Ziemer95,DaiTruWan12,DaiWanZho15} treat
right-hand sides of type $H_0(x)$ with $H_0\in\L^1(D)$ replaced by a measure
$\mu_0$ on $D$, while for the non-parametric equations corresponding to
our functionals one expects right-hand sides of type $H(x,u(x))$ (with
dependence on the unknown $u$) with $H\in\L^1(D\times\R)$ replaced by a measure
$\mu$ on $D\times\R$. Beyond this partial comparison we stress that the
approaches taken are technically very different from ours and that the works
\cite{Ziemer95,DaiTruWan12,DaiWanZho15} do \emph{not} involve any semicontinuity
by cancellation. In fact, the more restrictive assumption \eqref{eq:Z-IC} of
\cite{Ziemer95} still ensures separate semicontinuity of the $\mu_0$-volume,
and the approach of \cite{DaiTruWan12,DaiWanZho15} works much more on the PDE
side rather than the variational side of the field and does not involve
semicontinuity of a functional with measure datum at all.

Finally, when a first version of this article was already finalized, an
independent preprint of Leonardi \& Comi \cite{LeoCom23v1} on non-parametric
functionals closely analogous to the parametric ones in \eqref{eq:Pmu-intro}
became available. In this interesting work the authors obtain (among other
results) lower semicontinuity and existence results over a bounded Lipschitz
domain $D\subset\R^{n-1}$ in case of specific measures
$\mu_0=h\mathcal{L}^{n-1}{+}\gamma\H^{n-2}\ecke\Gamma$ with
$h\in\L^q(D)$, $q>n{-}1$, an $(n{-}2)$-dimensional set $\Gamma\subset D$ with
bounded $(n{-}2)$-dimensional density ratio, and
$\gamma\in\L^\infty(\Gamma;\H^{n-2})$ such that moreover the IC
\eqref{eq:DTW-IC} holds. Though also these results concern the non-parametric
setting and differ considerably from ours in the framework and the technical
approach, we put on record that at its heart the work \cite{LeoCom23v1} brings up a
semicontinuity-by-cancellation effect analogous to ours.

Returning to the parametric case, we point out that ICs have been introduced
into the classical $2$-dimensional Douglas-Rad\'o theory of prescribed mean
curvature surfaces by Steffen \cite{Steffen76a,Steffen76b}. Among the
ICs considered in his work, a central type for functions $H\colon S\to\R$ on $S\subset\R^n$
reads in our terminology
\begin{equation}\label{eq:Steffen-IC}
  \bigg|\int_AH(x)\dx\bigg|\le C\P(A,\R^n)
  \qq\text{for all measurable }A\subset S\text{ with }H\in\L^1(A)\,,\,\P(A,\R^n)\le R\,,
\end{equation}
where $C\in{[0,1]}$ and $R\in{[0,\infty]}$ are fixed. In the classical case with
$n=3$ such ICs are then exploited in \cite{Steffen76a,Steffen76b} in
establishing lower semicontinuity of prescribed mean curvature functionals and
in case $C<1$ also existence results, where in a spirit similar to ours the ICs
compensate for a lack of separate lower semicontinuity of a certain $H$-volume
term. However, while in our theory the main issue originates from passing from
functions $H$ to measures $\mu_\pm$ and from a possible loss of a hypersurface
portion in the limit, in \cite{Steffen76a,Steffen76b} an analogous issue occurs
already for functions $H$ and is connected with a typical phenomenon of the
parametric theory, namely the possible bubbling-off of regions of positive
volume in the limit. In addition, Duzaar \cite{Duzaar93} and Duzaar \& Steffen
\cite{DuzSte96} have established existence results based on ICs of type
\eqref{eq:Steffen-IC} with $C<1$ also in Euclidean space $\R^n$ and in
Riemannian manifolds of arbitrary dimension $n$ by working in a general GMT
framework with codimension-$1$ currents. However, also the results in
\cite{Duzaar93,DuzSte96} are limited to functions $H$ and not measures $\mu_\pm$
in the volume term. Yet again, since bubbling off is not an issue in the
framework of currents, the role of the ICs is once more a bit different and
consists mostly in preventing a breakdown of semicontinuity at $\infty$, as it
has already been discussed and needs to be excluded in our theory as well.

\bigskip

\paragraph{Acknowledgments}\!
The author is grateful to T.\@ Ilmanen for a discussion on the extension of
Theorem \ref{thm:2P-admis} from perimeter measures to rectifiable measures, as
subsequently achieved in Corollary \ref{cor:rect-admis} and stated also in
Theorem \ref{thm:rect-admis-intro}. Moreover, the author wants to thank E.\@
Ficola and M.\@ Torres for pointing out references
\cite{Miranda74,CarLeaPas86,CarLeaPas87} and \cite{Ziemer95}, respectively, and
J.\@ Schütt for a careful reading of a preliminary version of the manuscript.
The figures in this article have been created in the vector graphics language
`Asymptote'.

\section{Preliminaries}\label{sec:prelim}

We work in Euclidean space $\R^n$ of arbitrary dimension
$n\in\N=\{1,2,3,\ldots\}$ (unless indicated otherwise).

\subsection*{Basic notation for sets and balls}

Our basic notation for sets is widely standard. However, we mention that we use
$A^\mathrm{c}$ for the complement of a set $A$ (in $\R^n$ or in some other base
set clear from the context),
$A\Delta B\coleq(A{\setminus}B)\cup(B{\setminus}A)$ for the symmetric difference
of sets $A$ and $B$, and $\1_A$ for the characteristic function of a set $A$
with $\1_A\equiv1$ on $A$ and $\1_A\equiv0$ on $A^\c$. By $\overline A$ and
$\inn(A)$ we denote the closure and the interior, respectively, of a set $A$
(taken once more in $\R^n$ or another base space). We write $A\Subset B$ if
$\overline A$ is compact and satisfies $\overline A\subset B$. Moreover, we use
$\B_r(x)\coleq\{y\in\R^n\,:\,|y{-}x|<r\}$ for balls in $\R^n$, we abbreviate
$\B_r\coleq\B_r(0)$, and we denote by $\alpha_n=|\B_1|$ the volume of the unit
ball $\B_1$ in $\R^n$. Finally, for $a\in\R^n$, $A,B\subset\R^n$ we use
$\dist(a,B)\coleq\inf_{b\in B}|a{-}b|$ and
$\dist(A,B)\coleq\inf_{a\in A}\dist(a,B)$ for Euclidean distances.

\subsection*{Measures and convergence in measure}

We write $\Bo(\R^n)$ for the Borel $\sigma$-algebra on the full space $\R^n$ and
$\Bo(\Omega)=\{A\in\Bo(\R^n)\,:\,A\subset\Omega\}$ for the Borel
$\sigma$-algebra on a Borel subset $\Omega\in\Bo(\R^n)$. By a non-negative Borel
measure $\mu$ on a set $\Omega\in\Bo(\R^n)$ we mean a $\sigma$-additive set
function on $\Bo(\Omega)$ with values in ${[0,\infty]}$. The support $\spt\mu$
of such a measure $\mu$ is the smallest closed set $S\subset\Omega$ with
$\mu(S^\c)=0$, and $\mu$ is called finite if $\mu(\Omega)<\infty$ holds. A
non-negative Radon measure on an open set $\Omega\subset\R^n$ is a non-negative
Borel measure on $\Omega$ with finite value on all compacts subsets of $\Omega$.

Specifically, we work with the $n$-dimensional Lebesgue measure $\Ln$, which is
a non-negative Radon measure on $\R^n$, and with the $(n{-}1)$-dimensional
Hausdorff measure $\H^{n-1}$, which is at least a non-negative Borel measure on
$\R^n$. In case of $\Ln$ we also consider its extension from $\Bo(\R^n)$ to the
completed $\sigma$-algebra $\M(\R^n)$ of Lebesgue measurable subsets of
$\R^n$. We write $|A|\coleq\Ln(A)$ for the volume of $A\in\M(\R^n)$ and
generally adopt the convention that \emph{measure-theoretic notions are taken
with respect to the Lebesgue measure unless indicated otherwise}.
Specifically, this applies for \ae{} properties and the following convergences.
For $\Omega,A_k,A\in\M(\R^n)$ we define
\begin{align}
  \hspace{-1.1ex}A_k\text{ converge \emph{\ka globally\kz{} in measure} on }\Omega\text{ to }A_\infty
  &\coliff\lim_{k\to\infty}|(A_k\Delta A_\infty){\cap}\Omega|=0\,,
  \label{eq:conv-in-meas}\\
  A_k\text{ converge \emph{locally in measure} on }\Omega\text{ to }A_\infty
  &\coliff\lim_{k\to\infty}|(A_k\Delta A_\infty){\cap}K|=0\text{ for all compact }K\subset\Omega\,.\hspace{-1ex}
  \label{eq:local-conv-in-meas}
\end{align}
We remark that in most of the following we will apply \eqref{eq:conv-in-meas} and
\eqref{eq:local-conv-in-meas} in the standard case of open $\Omega$ only, but in
fact we have intentionally given the definitions for arbitrary measurable
$\Omega$, since this more general viewpoint will become relevant for Theorem
\ref{thm:lsc-BV-dom} and Corollary \ref{cor:lsc-per-BV-dom} in the final section
of this paper. Indeed, the reasonableness of this framework is supported by the
fact that just as the convergence in \eqref{eq:conv-in-meas} also the
convergence in \eqref{eq:local-conv-in-meas} depends on $\Omega$ only up to
negligible sets, as one can verify in case of \eqref{eq:local-conv-in-meas} by a
short reasoning with the inner regularity of the Lebesgue measure. Moreover, the
same reasoning shows that equivalent with \eqref{eq:local-conv-in-meas} is
having $\lim_{k\to\infty}|(A_k\Delta A_\infty)\cap S|=0$ even for all $S\in\M(\R^n)$
with $|S\setminus\Omega|=0$ and $|S|<\infty$. Finally, we briefly remark that
local convergence in measure is closely tied to almost everywhere convergence in
the sense of $\lim_{k\to\infty}\1_{A_k}=\1_{A_\infty}$ \ae{} on $\Omega$: In fact, almost
everywhere convergence implies local convergence in measure, and local
convergence in measure implies almost everywhere convergence of a subsequence.

In connection with signed measures and vector measures we adopt mostly the
conventions of \cite[Sections 1.1, 1.3]{AmbFusPal00}. Specifically, as a
signed Radon measure $\nu$ on open $\Omega\subset\R^n$ we consider any set
function which is defined and $\sigma$-additive with finite real values (at
least) on the relatively compact Borel subsets of $\Omega$, and an $\R^m$-valued
Radon measure is defined analogously with values in $\R^m$. A signed or
$\R^m$-valued Radon measure $\nu$ on $\Omega$ is called finite if it extends to
a finite-valued $\sigma$-additive set function on the full Borel
$\sigma$-algebra $\Bo(\Omega)$. With these conventions the (total) variation
measure $|\nu|$ of a signed or $\R^m$-valued Radon measure $\nu$ on $\Omega$ can
always be regarded as a non-negative Radon measure on $\Omega$ (where $|\nu|$ is
finite if and only if $\nu$ is finite). Moreover, every signed Radon measure
$\nu$ on $\Omega$ admits a unique decomposition $\nu=\nu_+{-}\nu_-$ into
mutually singular non-negative Radon measures $\nu_+$ and $\nu_-$ on $\Omega$,
which also satisfy $|\nu|=\nu_+{+}\nu_-$.

Finally, for any measure $\nu$ on a measurable space $(\Omega,\mathcal{A})$, the
weighted measure $f\nu$ on $(\Omega,\mathcal{A})$ with
$f\in\L^1(\Omega\;\!;\nu)$ is defined by setting $(f\nu)(A)\coleq\int_Af\,\d\nu$
for all $A\in\mathcal{A}$. Specifically, the restriction measure $\nu\ecke S$ on
$(\Omega,\mathcal{A})$ with $S\in\mathcal{A}$ is obtained through
$(\nu\ecke S)(A)\coleq(\1_S\nu)(A)=\nu(S\cap A)$ for all $A\in\mathcal{A}$.

\subsection*{Coarea formula for Lipschitz functions}
For a (locally) Lipschitz function $\Omega\to\R$ on open $\Omega\subset\R^n$,
Rademacher's theorem guarantees the existence of the derivative
$\nabla u(x)\in\R^n$ at \ae{} $x\in\Omega$; compare e.\@g.\@ with
\cite[Section 2.3]{AmbFusPal00}, \cite[Section 3.1]{EvaGar15},
\cite[Section 7.3]{Maggi12}, or \cite[Theorem 7.3]{Mattila95}. With the
derivative at hand the coarea formula for Lipschitz functions can then be stated
as follows.

\begin{thm}[coarea formula for Lipschitz functions]\label{thm:Lip-coarea}
  Consider a Lipschitz function $u\colon\Omega\to\R$ on open
  $\Omega\subset\R^n$. Then we have 
  \[
    \int_A|\nabla u|\dx
    =\int_{-\infty}^\infty\H^{n-1}(A\cap\{u=t\})\,\d t
    \qq\text{for all }A\in\Bo(\Omega)\,.
  \]
\end{thm}

For the proof (of actually more general statements) we refer to
\cite[Section 2.12]{AmbFusPal00}, \cite[Section 3.4]{EvaGar15}, or
\cite[Section 18.1]{Maggi12}, for instance.

\subsection*{\boldmath Sets of finite perimeter (and $\BV$ functions)}
In working with spaces of integrable and weakly differentiable functions such as
$\mathrm{L}^p_{(\loc)}(\Omega)$, $\W^{1,p}_{\!(\loc)}(\Omega)$, $\BV_{\!(\loc)}(\Omega)$ we
follow once more the terminology of \cite{AmbFusPal00}. In particular, for a
real-valued $\BV$ function $u\in\BV_\loc(\Omega)$ on open $\Omega\subset\R^n$,
we write $\D u$ for the $\R^n$-valued Radon measure which represents the
distributional gradient of $u$ on $\Omega$. Moreover, we generally use
$u_\pm\coleq\max\{{\pm}u,0\}$ for the positive and negative part of functions,
but we directly warn the reader that in addition to this convention with
\emph{lower} indices ${}_\pm$ we will soon introduce \emph{upper} indices
${}^\pm$ for certain approximate limits as well.

We introduce the perimeter $\P(A,\Omega)$ of a measurable set $A\in\M(\R^n)$ in
an arbitrary Borel set $\Omega\in\Bo(\R^n)$ by setting
$\P(A,\Omega)\coleq|\D\1_A|(\Omega)$ whenever there exists an open neighborhood
$U$ of $\Omega$ in $\R^n$ such that $\1_A\in\BV_\loc(U)$ and by trivially setting
$\P(A,\Omega)\coleq\infty$ otherwise. For open $\Omega$ this coincides with more
standard distributional definitions, while in general we have
$\P(A,\Omega)=\inf\{\P(A,U)\,:\,U\text{ open neighborhood of }\Omega\text{ in }\R^n\}$.
As usual we abbreviate $\P(A)\coleq\P(A,\R^n)$.

We next record two standard results, where the former can be inferred
from \cite[Theorem 3.39]{AmbFusPal00} or \cite[Corollary 12.27]{Maggi12}, and
the later from \cite[Proposition 3.38(b)]{AmbFusPal00},
\cite[Theorem 5.2]{EvaGar15}, or \cite[Proposition 12.15]{Maggi12}.

\begin{lem}[compactness from perimeter bounds]\label{lem:cpct-per}
  Consider an open set\/ $\Omega\subset\R^n$. If\/ $(A_k)_{k\in\N}$ is a sequence
  in $\M(\R^n)$ with $\sup_{k\in\N}\P(A_k,\Omega)<\infty$, then a subsequence
  of\/ $(A_k)_{k\in\N}$ converges locally in measure on $\Omega$ to some limit
  $A_\infty\in\M(\R^n)$.
\end{lem}

\begin{lem}[lower semicontinuity of the perimeter]\label{lem:lsc-per}
  Consider an open set\/ $\Omega\subset\R^n$. If a sequence $(A_k)_{k\in\N}$ in
  $\M(\R^n)$ converges locally in measure on $\Omega$ to $A_\infty\in\M(\R^n)$, then
  we have
  \[
    \liminf_{k\to\infty}\P(A_k,\Omega)\ge\P(A_\infty,\Omega)\,.
  \]
\end{lem}

Whenever we have $\P(A,\Omega)<\infty$ for $A\in\M(\R^n)$ and
$\Omega\in\Bo(\R^n)$, we call $A$ a set of finite perimeter in $\Omega$, and we
write the class of sets of finite measure and finite perimeter in $\Omega$ as
\[
  \BVs(\Omega)
  \coleq\{A\in\M(\R^n)\,:\,|A\cap\Omega|{+}\P(A,\Omega)<\infty\}
  =\bigcup_{U\text{\! open,}\,\Omega\subset U}\{A\in\M(\R^n)\,:\,\1_A\in\BV(U)\}\,.
\]
Moreover, we call $A\in\M(\R^n)$ a set of locally finite perimeter in open
$\Omega\subset\R^n$ if $\P(A,K)<\infty$ holds for all compact $K\subset\Omega$.
The corresponding class is written, still for open $\Omega$, as
\[
  \BVs_\loc(\Omega)
  \coleq\{A\in\M(\R^n)\,:\,\P(A,K)<\infty\text{ for all compact }K\subset\Omega\}
  =\{A\in\M(\R^n)\,:\,\1_A\in\BV_\loc(\Omega)\}\,,
\]

The reduced boundary of $A\in\BVs(\Omega)$ in $\Omega\in\Bo(\R^n)$ in the sense
of \cite[Definition 3.54]{AmbFusPal00}, \cite[Definition 5.4]{EvaGar15},
\cite[Section 15]{Maggi12} is denoted by $\partial^\ast\!A$ or by
$\Omega\cap\partial^\ast\!A$. Its significance is partially highlighted by the
following result, which can be read off
from \cite[Theorem 3.59]{AmbFusPal00}, \cite[Theorem 5.15]{EvaGar15}, or
\cite[Theorem 15.9]{Maggi12}.

\begin{thm}[De Giorgi's structure theorem; partial statement]
    \label{thm:DeGiorgi}
  For $A\in\M(\R^n)$ and $\Omega\in\Bo(\R^n)$ with $\P(A,\Omega)<\infty$, it holds
  \[
    \P(A,\,\cdot\,)=|\D\1_A|=\H^{n-1}\ecke\partial^\ast\!A
    \qq\text{as measures on }\Omega\,.
  \]
\end{thm}

With this result in mind, from here on we mostly use $\P(A,\,\cdot\,)$ as the
preferred notation for the perimeter measure of a set $A$ of (locally) finite
perimeter.

In view of the conventions for $\BV$ functions and $\BVs$ sets we can also state
a variant of the coarea formula, which is contained in e.\@g.\@
\cite[Theorem 3.40]{AmbFusPal00} or \cite[Theorem 5.9]{EvaGar15}.

\begin{thm}[Fleming-Rishel coarea formula]\label{thm:BV-coarea}
  Consider an open set $\Omega\subset\R^n$ and $u\in\BV(\Omega)$. Then, for
  $\mathcal{L}^1$-\ae{} $t\in\R$, we have $\{u>t\}\in\BVs(\Omega)$, and it holds
  \[
    |\D u|(A)=\int_{-\infty}^\infty\P(\{u>t\},A)\,\d t
    \qq\text{for all }A\in\Bo(\Omega)\,.
  \]
\end{thm}

Finally, we use the following result, which in this form is provided by
\cite[Theorem 3.46]{AmbFusPal00}, for instance.

\begin{thm}[isoperimetric estimate]\label{thm:isoperi}
  For $n\ge2$ and $A\in\M(\R^n)$, we have
  \[
    \min\{|A|,|A^\c|\}\le\Gamma_n\P(A)^\frac n{n-1}
  \]
  with a constant $\Gamma_n>0$ which depends only on $n$. Evidently, in case
  $|A|<\infty$ this reduces to $|A|\le\Gamma_n\P(A)^\frac n{n-1}$.
\end{thm}

With the determination of the optimal constant
$\Gamma_n=\P(\B_1)^{-\frac n{n-1}}|\B_1|=\P(\B_r)^{-\frac n{n-1}}|\B_r|$, the
preceding statement turns into the isoperimetric inequality
\begin{equation}\label{eq:sharp-isoperi}
  \P(\B_r)\le\P(A)\qq\text{for }r\in{(0,\infty)}\text{ and all
  }A\in\M(\R^n)\text{ with }|A|=|\B_r|\,;
\end{equation}
for a proof see \cite[Chapter 14]{Maggi12}, for instance. For the purposes of this
paper we need \eqref{eq:sharp-isoperi} only at a single point in the proof of
Theorem \ref{thm:ex-prscr-vol}, while otherwise the estimate of Theorem
\ref{thm:isoperi} with any constant $\Gamma_n$ suffices.

Finally, we record the following basic estimate (which has also variants for
sets with \emph{finite} $\H^{n-1}$-measure):

\begin{lem}\label{lem:negligible}
  For every $\H^{n-1}$-negligible $N\in\Bo(\R^n)$ and every $\eps>0$, there
  exists an open set $A$ such that
  \[
    N\subset A\subset\mathcal{N}_\eps(N)\,,\qq\qq
    |A|<\eps\,,
    \qq\qq\text{and}\qq\qq
    \P(A)<\eps
  \]
  \ka with the $\eps$-neighborhood\/
  $\mathcal{N}_\eps(N)\coleq\{x\in\R^n:\dist(x,N)<\eps\}$ of\/ $N$\kz.
\end{lem}

\begin{proof}
  By definition of $\H^{n-1}$, there exist open balls
  $B_i\subset\mathcal{N}_\eps(N)$ with corresponding radii $r_i\in{(0,n]}$ such that
  $N\subset\bigcup_{i=1}^\infty B_i$ and
  $n\alpha_n\sum_{i=1}^\infty r_i^{n-1}<\eps$ hold. For the open set
  $A\coleq\bigcup_{i=1}^\infty B_i$ with $N\subset A\subset\mathcal{N}_\eps(N)$,
  we get
  \[
    |A|\le\sum_{i=1}^\infty|B_i|=\alpha_n\sum_{i=1}^\infty r_i^n
    \le n\alpha_n\sum_{i=1}^\infty r_i^{n-1}<\eps
    \qq\text{and}\qq
    \P(A)\le\sum_{i=1}^\infty\P(B_i)
    =n\alpha_n\sum_{i=1}^\infty r_i^{n-1}<\eps\,.
  \]
  This completes the proof.
\end{proof}

\subsection*{\boldmath$\H^{n-1}$-\ae{} representatives and set operations
  for sets of finite perimeter}

For $A\in\M(\R^n)$, $\t\in{[0,1]}$ we introduce the Borel sets
\begin{gather*}
  A^\t
  \coleq\bigg\{x\in\R^n\,:\,\lim_{\r\searrow0}\frac{|\B_\r(x)\cap A|}{|\B_\r|}=\t\bigg\}
  \qq\text{and}\qq
  A^+\coleq\big(A^0\big)^\c=\bigg\{x\in\R^n\,:\,\limsup_{\r\searrow0}\frac{|\B_\r(x)\cap A|}{|\B_\r|}>0\bigg\}
\end{gather*}
of density-$\t$ points and positive-upper-density points of $A$, and we record that
$A^1=A^+=A$ holds up to negligible sets (see e.\@g.\@ by \cite[Theorem 1.35]{EvaGar15},
\cite[eq.\@ (5.19)]{Maggi12}, or \cite[Corollary 2.14(1)]{Mattila95}). More can be said
in case $A$ has finite perimeter: Then the $A^\t$ are significant only for
$\t\in\{0,\frac12,1\}$, and the essential boundary
\[
  \partial^\mathrm{e}\!A\coleq A^+\setminus A^1
\]
is not only negligible, but in fact coincides with the reduced boundary $\partial^\ast\!A$ up
to an $\H^{n-1}$-negligible sets. In fact, this is made precise in the next result, for
which we refer to \cite[Theorem 3.61]{AmbFusPal00} or \cite[Theorem 16.2]{Maggi12}.

\begin{thm}[Federer's structure theorem]\label{thm:Federer}
  For $A\in\M(\R^n)$, $\Omega\in\Bo(\R^n)$ with $\P(A,\Omega)<\infty$, there hold
  $\Omega\cap\partial^\ast\!A\subset A^\frac12$ and
  \[
    \H^{n-1}((\partial^\mathrm{e}\!A\setminus\partial^\ast\!A)\cap\Omega)=0
  \]
\end{thm}

In particular, in the situation of the theorem we infer
$\H^{n-1}(A^\t\cap\Omega)=0$ for all $\t\in{[0,1]}\setminus\{0,\frac12,1\}$,
and the equalities $\partial^\ast\!A\cap\Omega=A^\frac12\cap\Omega
=\partial^\mathrm{e}\!A\cap\Omega$
and $A^+\cap\Omega=(A^1\cup\partial^\ast\!A)\cap\Omega$ hold up to
$\H^{n-1}$-negligible sets. Altogether this supports viewing $A^+$ as
measure-theoretic closure and $A^1$ as measure-theoretic interior of $A$.

\smallskip

Next we discuss basic set operations and corresponding estimates for sets of
finite perimeter.

\begin{lem}\label{lem:P(AcapB),P(A-S)}
  For $A,B\in\M(\R^n)$, $\Omega\in\Bo(\R^n)$ with
  $\P(A,\Omega)+\P(B,\Omega)<\infty$, there holds
  \begin{equation}\label{eq:P(AcapB)}
    \P(A\cap B,G)\le\P(A,B^1\cap G)+\P(B,A^+\cap G)
    \qq\qq\text{for all }G\in\Bo(\Omega)
  \end{equation}
  and in particular $\P(A\cap B,\Omega)<\infty$. If either
  $|(A\setminus B)\cap G|=0$ or $|(B\setminus A)\cap G|=0$ or
  $\H^{n-1}(\partial^\ast\!A\cap\partial^\ast\!B\cap G)=0$ holds, then we have
  equality in \eqref{eq:P(AcapB)}.

  Similarly, for $A,S\in\M(\R^n)$, $\Omega\in\Bo(\R^n)$ with
  $\P(A,\Omega)+\P(S,\Omega)<\infty$, there holds
  \begin{equation}\label{eq:P(A-S)}
    \P(A\setminus S,G)\le\P(A,S^0\cap G)+\P(S,A^+\cap G)
    \qq\qq\text{for all }G\in\Bo(\Omega)
  \end{equation}
  and in particular $\P(A\setminus S,\Omega)<\infty$. If either
  $|A\cap S\cap G|=0$ or $|G\setminus(A\cup S)|=0$ or
  $\H^{n-1}(\partial^\ast\!A\cap\partial^\ast\!S\cap G)=0$ holds, then we have
  equality in \eqref{eq:P(A-S)}.
\end{lem}

\begin{proof}
  We observe that $\P(A\cap B,\Omega)<\infty$ is ensured, for instance, by
  applying the basic product rule estimate \cite[eq.\@ (3.10)]{AmbFusPal00} for
  the derivative of $\1_{A\cap B}=\1_A\1_B$. Now we consider
  $x\in(A^1{\cup}A^\frac12{\cup}A^0)\cap(B^1{\cup}B^\frac12{\cup}B^0)$. Then
  $x\in (A\cap B)^\frac12$ necessarily implies that either
  $x\in A^\frac12\cap B^1$ or $x\in B^\frac12\cap A^+$ holds. In view of
  Theorem \ref{thm:Federer} this means $\partial^\ast(A\cap B)
  \subset(\partial^\ast\!A\cap B^1)\cup(\partial^\ast\!B\cap A^+)$ up to
  $\H^{n-1}$-negligible sets, and via Theorem \ref{thm:DeGiorgi} we arrive
  at \eqref{eq:P(AcapB)}. In order to discuss equality, one can use the full
  statement of De Giorgi's theorem as provided in
  \cite[Theorem 3.59]{AmbFusPal00} to verify more precisely
  $\partial^\ast(A\cap B)=(\partial^\ast\!A\cap B^1)\cup(\partial^\ast\!B\cap A^1)
  \cup(\partial^\ast\!A\cap\partial^\ast\!B\cap\{\nu_A=\nu_B\})$ up to
  $\H^{n-1}$-negligible sets, where $\nu_A$ and $\nu_B$ denote the generalized
  outward unit normals of $A$ and $B$. Then one reads off that equality occurs
  in \eqref{eq:P(AcapB)} if and only if $\nu_A=\nu_B$ holds $\H^{n-1}$-\ae{} on
  $\partial^\ast\!A\cap\partial^\ast\!B\cap G$, and the latter can be checked to
  follow from each of the conditions claimed to be sufficient for equality.

  We find worth recording also the following alternative derivation of
  \eqref{eq:P(AcapB)}. From the rule for the derivative of composite functions
  in \cite[Theorem 3.84]{AmbFusPal00} we get
  \[
    \P(A\cap B,G)
    =|\D(\1_A\1_B)|(G)
    =|\D\1_A|(B^1\cap G)
    +\big(|(\1_A)^{\inn}_{\partial^\ast\!B}|\H^{n-1}\big)(\partial^\ast\!B\cap G)
    \qq\text{for }G\in\Bo(\Omega)
  \]
  and specifically $\P(A\cap B,\Omega)<\infty$, where
  $(\1_A)^{\inn}_{\partial^\ast\!B}$ stands for the interior trace of $\1_A$ on
  $\partial^\ast\!B$. Since the trace is $\{0,1\}$-valued with value $1$ on
  $A^1\cap\partial^\ast\!B$ and value $0$ on
  $A^0\cap\partial^\ast\!B=(A^+)^\c\cap\partial^\ast\!B$, with the help of
  Theorem \ref{thm:DeGiorgi} we obtain
  \[
    \P(A\cap B,G)
    \le|\D\1_A|(B^1\cap G)+\H^{n-1}(\partial^\ast\!B\cap A^+\cap G)
    =\P(A,B^1\cap G)+\P(B,A^+\cap G)
    \qq\text{for }G\in\Bo(\Omega)
  \]
  and arrive once more at \eqref{eq:P(AcapB)}. From these arguments one reads
  off that equality occurs in \eqref{eq:P(AcapB)} if and only if
  $(\1_A)^{\inn}_{\partial^\ast\!B}=1$ holds $\H^{n-1}$-\ae{} on
  $(A^+\setminus A^1)\cap\partial^\ast\!B\cap G$. In view of Theorem
  \ref{thm:Federer} it is equivalent that $(\1_A)^{\inn}_{\partial^\ast\!B}=1$
  holds $\H^{n-1}$-\ae{} on $\partial^\ast\!A\cap\partial^\ast\!B\cap G$, and
  once more this can be checked to follow from each of the conditions in the
  statement.

  Finally, the inequality \eqref{eq:P(A-S)} is nothing but the inequality
  \eqref{eq:P(AcapB)} for $B=S^\c$.
\end{proof}

Also the following combined estimate for the perimeters of union and intersection
is well known.

\begin{lem}\label{lem:P-cup-cap}
  For $A,B\in\M(\R^n)$, $\Omega\in\Bo(\R^n)$ with
  $\P(A,\Omega)+\P(B,\Omega)<\infty$, we have 
  \begin{equation}\label{eq:P-cup-cap}
    \P(A\cup B,G)+\P(A\cap B,G)\le\P(A,G)+\P(B,G)
    \qq\text{for all }G\in\Bo(\Omega)
  \end{equation}
  and thus in particular $\P(A\cup B,\Omega)+\P(A\cap B,\Omega)<\infty$.
\end{lem}

\begin{proof}[Proofs]
  A basic approach is given in the proofs of
  \cite[Proposition 3.38(d)]{AmbFusPal00} and \cite[Lemma 12.22]{Maggi12}, where
  the claim is shown for open $G$ by approximating $\1_A$ and $\1_B$ with smooth
  functions. Our claim for arbitrary $G\in\Bo(\Omega)$ then follows by
  regularity of the perimeter measures.

  Alternatively, one may obtain the lemma from the equality
  $|\D u_+|{+}|\D u_-|=|\D u|$ for $u\in\BV_\loc(U)$ on open $U\subset\R^n$
  (which in turn results from an approximation argument somewhat similar to the
  previously mentioned one). In fact, using the equality for
  $u\coleq\1_A{+}\1_B{-}1$ with $u_+=\1_{A\cap B}$ and $u_-=1-\1_{A\cup B}$ we
  directly obtain $\P(A\cap B,G)+\P(A\cup B,G)
  =|\D u_+|(G)+|\D u_-|(G)=|\D u|(G)\le\P(A,G)+\P(B,G)$.

  Finally, we find worth recording that the claim can also be derived from the
  preceding Lemma \ref{lem:P(AcapB),P(A-S)}. Indeed, elementary rules for
  complements and \eqref{eq:P(AcapB)} with $B^\c$ in place of $A$ and $A^\c$ in
  place of $B$ yield
  \[
    \P(A\cup B,G)
    =\P(B^\c\cap A^\c,G)
    \le\P(B^\c,(A^\c)^1\cap G)+\P(A^\c,(B^\c)^+\cap G)
    =\P(B,(A^+)^\c\cap G)+\P(A,(B^1)^\c\cap G)\,.
  \]
  Summing up the original version of \eqref{eq:P(AcapB)} and the variant just
  derived, we arrive at \eqref{eq:P-cup-cap} once more.
\end{proof}

\subsection*{Pseudoconvexity}

Pseudoconvexity, a weak version of mean-convexity, has been introduced by
Miranda \cite{Miranda71} and will eventually be relevant for us in connection
with the discussion of a basic example. We restate the definition and a first
lemma in versions adapted to our framework.

\begin{defi}[pseudoconvexity]
  We say that $K\in\BVs(\R^n)$ is pseudoconvex if it satisfies
  \begin{equation}\label{eq:pseudoconvex}
    \P(K)\le\P(B)
    \qq\text{whenever }K\subset B\in\M(\R^n)\text{ with }|B|<\infty\,.
  \end{equation}
\end{defi}

\begin{lem}\label{lem:intersect-with-pseudoconvex}
  For every pseudoconvex set $K\in\BVs(\R^n)$, we have
  \[
    \P(A\cap K)\le\P(A)
    \qq\text{for all }A\in\M(\R^n)\text{ with }|A|<\infty\,.
  \]
\end{lem}

\begin{proof}
  From \eqref{eq:P-cup-cap} and the definition of pseudoconvexity, applied with
  $B=A\cup K$, we get
  \[
    \P(A\cap K)\le\P(A)+\P(K)-\P(A\cup K)\le\P(A)\,.\qedhere
  \]
\end{proof}

Clearly, a basic feature of pseudoconvexity is that convex sets are
pseudoconvex. Though this may be considered as geometrically quite obvious, we
prefer to sketch at least one possible precise proof.

\begin{lem}[convexity implies pseudoconvexity]\label{lem:convex->pseudoconvex}
  Every bounded, convex set $K\in\M(\R^n)$ with $\inn(K)\neq\emptyset$ satisfies
  $K\in\BVs(\R^n)$ with $\H^{n-1}(\partial K\setminus\partial^\ast\!K)=0$ and is
  actually pseudoconvex.
\end{lem}

\begin{proof}[Sketch of proof]
  The claims $K\in\BVs(\R^n)$ and
  $\H^{n-1}(\partial K\setminus\partial^\ast\!K)=0$ follow from
  \cite[Proposition 3.62]{AmbFusPal00}. We now establish the inequality
  \eqref{eq:pseudoconvex} for the convex set $K$, at first only with the extra
  assumption that $B$ is a bounded $\C^1$ domain. Indeed, for every
  $x\in\partial K$, we may choose any ray from $x$ in a direction of the
  outward normal cone to $K$ at $x$ and on this ray find some
  $y\in\partial B=\partial^\ast\!B$ with $p_{\overline K}(y)=x$ for the
  nearest-point projection $p_{\overline K}\colon\R^n\to\overline K$ onto
  $\overline K$. This shows $\partial K\subset p_{\overline K}(\partial^\ast\!B)$.
  Then, since $p_{\overline K}$ is a contraction, we get
  $\P(K)=\H^{n-1}(\partial K)\le\H^{n-1}(\partial^\ast\!B)=\P(B)$ as claimed. In
  a next step, we weaken the extra assumption to merely
  $B\in\BVs(\R^n)$ and show that \eqref{eq:pseudoconvex} still applies. To this
  end we approximate $B$ with bounded $\C^1$ domains $B_\ell$ such that
  $\lim_{\ell\to\infty}\P(B_\ell)=\P(B)$ as in \cite[Theorem 3.42]{AmbFusPal00},
  where we can additionally arrange for $K_\ell\subset B_\ell$ with the bounded,
  convex sets $K_\ell\coleq\{x\in\R^n\,:\,\dist(x,K^\c)>\eps_\ell\}$,
  suitable $\eps_\ell>0$, and $\lim_{\ell\to\infty}\eps_\ell=0$. As we infer
  $\liminf_{\ell\to\infty}\P(K_\ell)\ge\P(K)$ by Lemma \ref{lem:lsc-per}, we can
  then carry over \eqref{eq:pseudoconvex} from $K_\ell$ and $B_\ell$ to $K$ and
  $B$ as claimed. Finally, we deduce \eqref{eq:pseudoconvex} in full generality
  by approximating $B$ with $B\cap\B_R$ and exploiting the convergence
  $\liminf_{R\to\infty}\P(B\cap\B_R)=\P(B)$ (which in turn results from Lemma
  \ref{lem:lsc-per}, the estimate
  $\P(B\cap\B_R)\le\P(B)+\H^{n-1}(B^1\cap\partial\B_R)$, and
  $\int_0^\infty\H^{n-1}(B^1\cap\partial\B_R)\,\d R=|B|<\infty$).
\end{proof}

\subsection*{\boldmath$\H^{n-1}$-\ae{} representatives of $\BV$ functions}

For measurable $u\colon\Omega\to\R$ on open $\Omega\subset\R^n$, by taking the
approximate upper and lower limits in the sense of
\[
  u^+(x)\coleq\sup\{t\in\R\,:\,x\in\{u>t\}^+\}
  \qq\text{and}\qq
  u^-(x)\coleq\sup\{t\in\R\,:\,x\in\{u>t\}^1\}
  \qq\text{for }x\in\Omega
\]
(where as usual $\sup\emptyset\coleq{-}\infty$) we obtain two
extended-real-valued Borel functions $u^+\ge u^-$ on $\Omega$. Occasionally we
also work with their arithmetic mean $u^\ast\coleq\frac12(u^+{+}u^-)$. In
particular, for $A\in\M(\R^n)$, we have $(\1_A)^+=\1_{A^+}$ and
$(\1_A)^-=\1_{A^1}$ on $\R^n$. We also record that, whenever $u$ has value
$y_0\in\R$ at a Lebesgue point $x_0\in\Omega$ (in the sense that
$\lim_{r\searrow0}|\B_r|^{-1}\int_{\B_r(x_0)}|u{-}y_0|\dx=0$), then
$u^\ast(x_0)=u^+(x_0)=u^-(x_0)=y_0$ holds. Hence, it follows from
\cite[Corollary 2.23]{AmbFusPal00} that in case of $u\in\L^1_\loc(\Omega)$ the
representatives $u^+$, $u^-$, $u^\ast$ of $u$ coincide \ae{} on $\Omega$.
Moreover, as a consequence of the Federer-Volpert theorem (see e.\@g.\@
\cite[Theorem 3.78]{AmbFusPal00}), for $u\in\W^{1,1}_\loc(\Omega)$ the
coincidence $u^\ast=u^+=u^-$ stays valid even $\H^{n-1}$-\ae{} on $\Omega$, and
for $u\in\BV_\loc(\Omega)$ one has $u^\ast=u^+=u^-$ at least $\H^{n-1}$-\ae{} on
$\Omega\setminus\mathrm{J}_u$, while on the approximate jump set $\mathrm{J}_u$
the values $u^+$ and $u^-$ correspond $\H^{n-1}$-\ae{} to the two jump values in
the sense of \cite[Definition 3.67]{AmbFusPal00}.

\subsection*{\boldmath$1$-capacity}

A decisive role in at least one central proof of this paper is taken by
$1$-capacity, also known as $\BV$-capacity, in the sense of the next definition.

\begin{defi}[$1$-capacity]\label{def:Cp1}
  For an arbitrary set $E\subset\R^n$, we define
  \[
    \Cp_1(E)\coleq\inf\bigg\{\int_{\R^n}|\nabla u|\dx\,:\,
    u\in\W^{1,1}(\R^n)\,,\,u\ge1\text{ \ae{} on an open neighborhood of }E\bigg\}
    \in{[0,\infty]}
  \]
  \ka with the usual understanding that $\Cp_1(E)=\infty$ if no such $u$
  exists, as, for instance, in case $|E|=\infty$\kz.
\end{defi}

The geometric meaning of $1$-capacity is captured by the following result.

\begin{prop}[perimeter characterization of $1$-capacity]\label{prop:inf-P}
  For every set $E\subset\R^n$, we have
  \[
    \Cp_1(E)=\inf\{\P(H)\,:\,H\in\BVs(\R^n)\,,\,E\subset H^+\}\,.
  \]
\end{prop}

\begin{proof}
  By \cite[Theorem 2.1]{CarDalLeaPas88}, the claim holds with the inclusion
  $E\subset H^+$ replaced either by $E\subset\inn(H)$ (for any pointwise
  representative of $H$) or by $\H^{n-1}(E\setminus H^+)=0$. Since we trivially
  have $E\subset\inn(H)\implies E\subset H^+\implies\H^{n-1}(E\setminus H^+)=0$,
  the claimed intermediate version of the formula follows. (In fact, taking into
  account Lemma \ref{lem:negligible}, the claimed version can alternatively be
  deduced from the version with $\H^{n-1}(E\setminus H^+)=0$ only.)
\end{proof}

The following result from \cite[Section 4]{FedZie72} can also be found in
\cite[Proposition 2.2(f)]{CarDalLeaPas88} and \cite[Theorem 5.12]{EvaGar15}, for
instance (where the latter statement is made for $n\ge2$ and compact sets, but
easily extends to the remaining cases).

\begin{prop}\label{prop:Cp1-null}
  For $S\in\Bo(\R^n)$, we have
  \[
    \Cp_1(S)=0\iff\H^{n-1}(S)=0\,.
  \]
\end{prop}

Finally, we record the following (semi)continuity properties of weakly
differentiable functions.

\begin{lem}[quasi semicontinuous representatives of a $\BV$ function]\label{lem:qsc}
  For open $\Omega\subset\R^n$ and\/ $u\in\BV_\loc(\Omega)$, the
  representatives $u^+$ and\/ $u^-$ of\/ $u$ are $\Cp_1$-quasi upper
  semicontinuous and $\Cp_1$-quasi lower semicontinuous, respectively, that is,
  for every $\eps>0$, there exists an open set\/ $E\subset\Omega$ with
  $\Cp_1(E)<\eps$ such that $u^+$ is upper semicontinuous on $E^\c$ and $u^-$
  is lower semicontinuous on $E^\c$.
\end{lem}

\begin{lem}[quasi continuity of a $\W^{1,1}$ function]\label{lem:qc}
  For open $\Omega\subset\R^n$ and\/ $u\in\W^{1,1}_\loc(\Omega)$, the
  representative $u^\ast$ of\/ $u$ is $\Cp_1$-quasi continuous, that is, for every
  $\eps>0$, there exists an open set\/ $E\subset\Omega$ with $\Cp_1(E)<\eps$ such
  that $u^\ast$ is defined and continuous on $E^\c$.
\end{lem}

Here, Lemma \ref{lem:qsc} is a restatement of \cite[Theorem
2.5]{CarDalLeaPas88}, while the claim of Lemma \ref{lem:qc} follows from
original statements established in \cite[Section 9, 10]{FedZie72} for the
full-space case. Clearly, one may also view Lemma \ref{lem:qc} as a consequence
of Lemma \ref{lem:qsc} and the $\H^{n-1}$-\ae{} coincidence $u^\ast=u^+=u^-$ for
$\W^{1,1}$ functions $u$.

\subsection*{\boldmath Strict and $\H^{n-1}$-\ae{} convergence and approximation}

\begin{lem}[strong convergence in $\W^{1,1}$ implies $\H^{n-1}$-\ae{}
    convergence]\label{lem:str-implies-Hae}
  If $v_\ell$ converge to $v$ in $\W^{1,1}(\Omega)$ on an open set\/
  $\Omega\subset\R^n$, then a subsequence of $(v_\ell^\ast)_{\ell\in\N}$
  converges $\H^{n-1}$-\ae{} on $\Omega$ to $v^\ast$.
\end{lem}

The case $\Omega=\R^n$ of Lemma \ref{lem:str-implies-Hae} is contained in
\cite[Section 10]{FedZie72} (where in view of Proposition \ref{prop:Cp1-null} we
may use $\H^{n-1}$ instead of $\Cp_1$). Since the claim can be localized, one
may pass to general domains $\Omega$ by simple cut-off arguments.

\begin{lem}[one-sided $\H^{n-1}$-\ae{} approximation of a $\BV$ function]
    \label{lem:mon-approx-u+}
  For every $u\in\BV(\R^n)$, there exists a sequence of functions
  $v_\ell\in\W^{1,1}(\R^n)$ such that $v_{\ell+1}\le v_\ell$ holds \ae{} on
  $\R^n$ for all $\ell\in\N$ and $v_\ell^\ast$ converge $\H^{n-1}$-\ae{} on
  $\R^n$ to $u^+$. If $u$ is bounded from above, one can additionally achieve
  $\sup_{\R^n}v_1\le\sup_{\R^n}u$.
\end{lem}

Lemma \ref{lem:mon-approx-u+} follows by combining Lemma \ref{lem:qsc} and
\cite[Lemma 1.5, Section 6]{DalMaso83}; compare also
\cite[Section 4, Section 10]{FedZie72}. However, since Lemma
\ref{lem:mon-approx-u+} plays a crucial role in this paper, in the following we
provide a slightly adapted and explicated rereading of the relevant arguments
of \cite{DalMaso83} in our setting.

\begin{proof}[Proof of Lemma \ref{lem:mon-approx-u+}]
  We first assume ${-}M\le u\le0$ \ae{} on $\R^n$ for some $M\in{[0,\infty)}$.
  Lemma \ref{lem:qsc} yields open sets $E_k\subset\R^n$ such that $u^+$ is upper
  semicontinuous on $E_k^\c$ for all $k\in\N$ and $\lim_{k\to\infty}\Cp_1(E_k)=0$
  holds. In particular we infer $\Cp_1\big(\bigcap_{k=1}^\infty E_k\big)=0$ and
  via Proposition \ref{prop:Cp1-null} also
  $\H^{n-1}\big(\bigcap_{k=1}^\infty E_k\big)=0$. By passage to finite
  intersections we can additionally achieve $E_{k+1}\subset E_k$ for all $k\in\N$.
  From Definition \ref{def:Cp1} we get functions $0\le w_k\in\W^{1,1}(\R^n)$ with
  $w_k^\ast\ge1$ on $E_k$ for all $k\in\N$ such that
  $\lim_{k\to\infty}\int_{\R^n}|\nabla w_k|\dx=0$ holds. Via
  the Gagliardo-Nirenberg inequality we conclude that $(w_k)_{k\in\N}$ converge to
  $0$ in $\W^{1,1}_\loc(\R^n)$, and after another passage to a subsequence Lemma
  \ref{lem:str-implies-Hae} gives $\H^{n-1}$-\ae{} convergence $w_k^\ast\to0$ on
  $\R^n$. Further, we recall that $\B_k$ denotes the open ball with center $0$
  and radius $k$, and we define upper semicontinuous functions $\overline u_k$
  on all of $\R^n$ by setting $\overline u_k\coleq u^+$ on $E_k^\c\cap\B_k$ and
  $\overline u_k\coleq{-}M$ on $E_k\cap\B_k$ together with $\overline
  u_k\coleq0$ outside $\B_k$. Then, for arbitrary $\ell\in\N$, the choice
  $\overline u_{k,\ell}(x)\coleq\max_{y\in\R^n}\big[\overline u_k(y){-}\ell|x{-}y|\big]$
  for $x\in\R^n$ produces compactly supported Lipschitz functions
  $\overline u_{k,\ell}$, which in particular satisfy
  $\overline u_{k,\ell}\in\W^{1,1}(\R^n)$ and in the limit $\ell\to\infty$
  converge from above to $\overline u_k$. Now we are ready to introduce
  \[
    v_\ell\coleq\min\{\overline u_{1,\ell}{+}Mw_1,\overline u_{2,\ell}{+}Mw_2,\ldots,\overline u_{\ell,\ell}{+}Mw_\ell\}\in\W^{1,1}(\R^n)\,.
  \]
  Since the construction ensures $\overline u_{k,\ell+1}\le\overline u_{k,\ell}$,
  we evidently have $v_{\ell+1}\le v_\ell$ \ae{} on $\R^n$ for all
  $\ell\in\N$. In addition, for each $k\in\N$, we find
  $\limsup_{l\to\infty}v_\ell^\ast\le\lim_{l\to\infty}\overline u_{k,\ell}+Mw_k^\ast
  =\overline u_k+Mw_k^\ast=u^++Mw_k^\ast$ on $E_k^\c\cap\B_k$. Then, by
  passing $k\to\infty$ and exploiting the choices of $E_k$ and $w_k$ (in
  particular the observation that in view of $E_{k+1}\subset E_k$ each point of
  $(\bigcap_{k=1}^\infty E_k)^\c$ is contained in $E_k^\c\cap\B_k$ for
  arbitrarily large $k$) we conclude that
  \[
    \limsup_{l\to\infty}v_\ell^\ast\le u^+
    \qq\text{holds }\H^{n-1}\text{-\ae{} on }\R^n\,.
  \]
  As a complement, for all $k,\ell\in\N$, the construction ensures
  $\overline u_{k,\ell}+Mw_k^\ast\ge\overline u_k \ge u^+$ on $E_k^\c$ and
  $\overline u_{k,\ell}+Mw_k^\ast\ge{-}M+Mw_k^\ast\ge0\ge u^+$ on $E_k$.
  Therefore, we also get
  \[
    \liminf_{l\to\infty}v_\ell^\ast\ge u^+
    \qq\text{on }\R^n
  \]
  and have checked all claims of the lemma in the case with ${-}M\le u\le0$.

  Next, we assume merely $u\le0$ \ae{} on $\R^n$, but allow $u$ to be unbounded
  from below. Then, for each $M\in\N$, the previous reasoning applies to
  $\max\{u,{-}M\}$ and gives functions $v_{\ell,M}\in\W^{1,1}(\R^n)$ with
  $v_{\ell+1,M}\le v_{\ell,M}$ \ae{} on $\R^n$ such that $v_{\ell,M}^\ast$
  converge $\H^{n-1}$-\ae{} on $\R^n$ to
  $\max\{u,{-}M\}^+=\max\{u^+,{-}M\}$. It is then a standard matter to verify
  the claims of the lemma for
  $v_\ell\coleq\min\{v_{\ell,1},v_{\ell,2},\ldots,v_{\ell,\ell}\}\in\W^{1,1}(\R^n)$.

  Finally, to prove the lemma for arbitrary $u\in\BV(\R^n)$, we exploit the
  existence of some $w\in\W^{1,1}(\R^n)$ such that $w\ge u$ \ae{} on
  $\R^n$. We subtract $w$, apply the preceding to $u{-}w\le0$, and then add $w$
  again to obtain suitable $v_\ell$. Clearly, if $u$ is additionally bounded
  from above, we can preserve the bound $M\coleq\sup_{\R^n}u$ by replacing
  $v_\ell$ with $\min\{v_\ell,M\}$ (or alternatively by taking $w\le M$ and
  revisiting the above construction).
\end{proof}

\begin{defi}[strict convergence in $\BV$]\label{defi:strict-conv}
  We say that a sequence of functions $u_\ell\in\BV(\Omega)$ converges strictly
  in $\BV(\Omega)$ to $u\in\BV(\Omega)$ if $u_\ell$ converge to $u$ in
  $\L^1(\Omega)$ with $\lim_{\ell\to\infty}|\D u_\ell|(\Omega)=|\D u|(\Omega)$.
\end{defi}

The following statement slightly adapts the one-sided approximation result of
\cite[Theorem 3.3]{CarDalLeaPas88} in order to additionally preserve boundedness
of the support and possibly the function itself.

\begin{lem}[one-sided strict approximation of a $\BV$ function]
    \label{lem:strict-approx-above}
  Consider an open set\/ $\Omega\subset\R^n$ and $u\in\BV(\Omega)$ with
  $\spt u\Subset\Omega$. Then there exists a sequence of functions
  $v_k\in\W^{1,1}(\Omega)$ such that $v_k$ converge strictly in $\BV(\Omega)$ to
  $u$ with $\spt v_k\Subset\Omega$ and $v_k\ge u$ \ae{}
  on $\Omega$ for all $k\in\N$. If $u$ is bounded from above, one can
  additionally achieve $\sup_\Omega v_k\le\sup_\Omega u$ for all
  $k\in\N$.
\end{lem}

\begin{proof}
  Since $\spt u$ is compact in $\Omega$, there is no loss of generality in
  assuming boundedness of $\Omega$. Then, by \cite[Theorem 3.3]{CarDalLeaPas88},
  there exist $w_k\in\W^{1,1}(\Omega)$ such that $w_k$ converge strictly in
  $\BV(\Omega)$ to $u$ with
  $w_k\ge u$ \ae{} on $\Omega$ for all $k\in\N$ (where in fact the
  convergence in area guaranteed by \cite[Theorem 3.3]{CarDalLeaPas88} is even
  stronger than the strict convergence of Definition \ref{defi:strict-conv}). We
  now fix a cut-off function $\eta\in\C^\infty_\cpt(\Omega)$ with
  $\1_{\spt u}\le\eta\le1$ on $\Omega$. Then, for
  $v_k\coleq\eta w_k\in\W^{1,1}(\Omega)$ with $\spt v_k\Subset\Omega$, it is
  standard to verify that $v_k$ still converge strictly in $\BV(\Omega)$ to $u$
  with $v_k\ge u$ \ae{} on $\Omega$ for all $k\in\N$. This establishes the main
  claim.

  If $u$ is additionally bounded, we replace $v_k$ already constructed
  with $\min\{v_k,L\}$ for $L\coleq\sup_\Omega u\ge0$. Taking into
  account the lower semicontinuity of the total variation, this preserves all
  previous properties and additionally ensures boundedness from above by $L$.
\end{proof}

We conclude this subsection with one more lemma which is tailored out for
constructing approximations with suitable smallness conditions on the support in
the proof of the later Theorem \ref{thm:char-small-vol-ic}.

\begin{lem}[control on the support of strict approximations]
    \label{lem:control-spt}
  Consider an open set\/ $\Omega\subset\R^n$. If $v_k\in\W^{1,1}_0(\Omega)$
  converge to $u\in\BV(\Omega)$ strictly in $\BV(\Omega)$ with $u\ge0$ \ae{} on
  $\Omega$ and\/ $|\{u>0\}|<M<\infty$, then there also exists a modified
  sequence of functions $w_\ell\in\W^{1,1}_0(\Omega)$ such that $w_\ell$ still
  converge to $u$ strictly in $\BV(\Omega)$ with $w_\ell\ge0$ \ae{} on $\Omega$
  and\/ $|\{w_\ell>0\}|<M$ for all $\ell\in\N$. Moreover, if all $v_k$ are
  even in $\C^\infty_\cpt(\Omega)$, all $w_\ell$ can be taken in
  $\C^\infty_\cpt(\Omega)$ as well, and in this case $|\{w_\ell>0\}|<M$ can be
  strengthened to $|{\spt w_\ell}|<M$. Finally, if $v_k$ converge even in
  $\W^{1,1}(\Omega)$ \ka and thus to $u\in\W^{1,1}_0(\Omega)$\kz, also $w_\ell$ can be
  taken to converge in $\W^{1,1}(\Omega)$.
\end{lem}

\begin{proof}
  We first establish the original claim. For fixed $\ell\in\N$ we observe
  $|\{v_k>\frac2\ell\}\setminus\{u>\frac1\ell\}|
  \le\ell\|v_k{-}u\|_{\L^1(\Omega)}$ and deduce
   $\limsup_{k\to\infty}|\{v_k>\frac2\ell\}|\le|\{u>\frac1\ell\}|<M$. Hence,
  for each $\ell\in\N$, we can choose $k_\ell\in\N$ such that in addition to
  $\|v_{k_\ell}{-}u\|_{\L^1(\Omega)}<\frac1\ell$ and
  $\|\nabla v_{k_\ell}\|_{\L^1(\Omega,\R^n)}\le|\D u|(\Omega){+}\frac1\ell$ we have
  $|\{v_{k_\ell}>\frac2\ell\}|<M$. For the non-negative functions
  $w_\ell\coleq\big(v_{k_\ell}{-}\frac2\ell\big)_+\in\W^{1,1}_0(\Omega)$, the
  previous properties and the non-negativity of $u$ imply via
  $\|w_\ell{-}u\|_{\L^1(\Omega)}\le\frac3\ell$ and
  $\|\nabla w_\ell\|_{\L^1(\Omega,\R^n)}\le|\D u|(\Omega){+}\frac1\ell$ the
  claimed strict convergence of $w_\ell$, and in view of
  $\{w_\ell>0\}=\{v_{k_\ell}>\frac2\ell\}$ we additionally get
  $|\{w_\ell>0\}|<M$. This completes the main part of the reasoning.

  If all $v_k$ are even in $\C^\infty_\cpt(\Omega)$, in order to preserve
  smoothness and control the support we slightly modify the choice of
  $w_\ell$. In fact, since in this situation $\{v_{k_\ell}\ge\frac3\ell\}$ is
  compact in the open set $\{v_{k_\ell}>\frac2\ell\}$, we even get
  $\spt w_\ell\subset\{v_{k_\ell}>\frac2\ell\}$ \emph{for a suitable
  mollification} $w_\ell\in\C^\infty_\cpt(\Omega)$ of
  $(v_{k_\ell}{-}\frac3\ell)_+$. Then, also exploiting standard estimates for
  mollifications, we conclude the reasoning by a straightforward adaptation of
  the preceding arguments.

  Finally, if the convergence is even in $\W^{1,1}(\Omega)$, we still argue
  in the same way, where the gradients can even be kept $\L^1$-close in the
  sense of $\|\nabla v_{k_\ell}{-}\nabla u\|_{\L^1(\Omega,\R^n)}\le\frac1\ell$.
\end{proof}

We remark that essentially the same proof yields versions of Lemma
\ref{lem:control-spt} for sequences in other spaces, e.\@g.\@ in
$\W^{1,1}(\Omega)$ or $\BV(\Omega)$ instead of $\W^{1,1}_0(\Omega)$. However,
since the above version suffices for our later purposes, we do not discuss this
any further.

\subsection*{\boldmath Normal traces of $\L^\infty$ vector fields with $\L^1$ divergence}

We next discuss, for vector fields $\sigma$ with $\L^1$ distributional
divergence, a notion of normal trace on the reduced boundary of a set of
finite perimeter. The considerations are given for the case of a base domain
$\Omega\subset\R^n$ which need not necessarily be bounded, and in fact we are
mostly interested in the full-space situation $\Omega=\R^n$.

\begin{defi}[distributional normal traces]\label{def:dist-traces}
  Consider an open set\/ $\Omega$ in $\R^n$, a set $E\in\M(\R^n)$ with
  $\P(E,\Omega)<\infty$, and a vector field $\sigma\in\L^1_\loc(\Omega,\R^n)$
  with distributional divergence $\Div\sigma\in\L^1_\loc(\Omega)$. Then we call
  the distribution
  \[
    \mutrace\coleq\1_E(\Div\sigma)-\Div(\1_E\sigma)
  \]
  on $\Omega$ the distributional normal trace \ka with respect to the outward
  normal\kz{} of\/ $\sigma$ on $\Omega\cap\partial^\ast\!E$.
\end{defi}

We remark that, spelling out the definition of $\mutrace$, we have
\begin{equation}\label{eq:int-trace-def}
  \langle\mutrace;\p\rangle
  =\int_E(\Div\sigma)\p\dx+\int_E\sigma\cd\nabla\p\dx
  \qq\text{for all }\p\in\mathrm{C}^\infty_\cpt(\Omega)\,.
\end{equation}
Taking into account the definition of the distributional divergence (or merely
its linearity), we also infer $\mutrace={-}\mathrm{Tr}_{E^\mathrm{c}}(\sigma)
={-}\1_{E^\mathrm{c}}\Div\sigma+\Div(\1_{E^\mathrm{c}}\sigma)$ in the sense of
distributions on $\Omega$, that is,
\begin{equation}\label{eq:ext-trace-def}
  \langle\mutrace;\p\rangle
  ={-}\int_{E^\mathrm{c}}(\Div\sigma)\p\dx-\int_{E^\mathrm{c}}\sigma\cd\nabla\p\dx
  \qq\text{for all }\p\in\mathrm{C}^\infty_\cpt(\Omega)\,.
\end{equation}

For \emph{bounded} $\sigma$, the distributional normal trace actually admits
a more concrete representation:

\begin{lem}[measure representation of the distributional normal trace]
    \label{lem:measure}
  Consider an open set $\Omega$ in $\R^n$, a set $E$ of finite perimeter in
  $\Omega$, and a bounded vector field $\sigma\in\L^\infty(\Omega,\R^n)$ with
  distributional divergence $\Div\sigma\in\L^1_\loc(\Omega)$. Then $\mutrace$ is a
  finite signed Radon measure on $\Omega$ and satisfies
  \[
    |\mutrace|\le\|\sigma\|_{\L^\infty;\Omega}\H^{n-1}\ecke(\Omega\cap\partial^\ast\!E)
    \qq\text{as measures on }\Omega\,.
  \]
\end{lem}

\begin{proof}
  We fix $\p\in\C^\infty_\cpt(\Omega)$ and consider standard mollifications
  $\sigma_\eps$ of $\sigma$, which are defined on all of $\spt\p$ at least for
  $0<\eps\ll1$. Then from \eqref{eq:int-trace-def} and standard properties of
  mollifications we deduce
  \[\begin{aligned}
    \big|\langle\mutrace;\p\rangle\big|
    &=\lim_{\eps\searrow0}\bigg|\int_E(\Div\sigma_\eps)\p\dx+\int_E\sigma_\eps\cd\nabla\p\dx\bigg|
    =\lim_{\eps\searrow0}\bigg|\int_{\Omega}\1_E\Div(\p\sigma_\eps)\dx\bigg|\\
    &=\lim_{\eps\searrow0}\bigg|\int_{\Omega}\p\sigma_\eps\cd\d\D\1_E\bigg|
    \le\|\sigma\|_{\L^\infty;\Omega}\int_{\Omega}|\p|\,\d|\D\1_E|\,,
  \end{aligned}\]
  where specifically in the last step we used the bound
  $\|\sigma_\eps\|_{\L^\infty;\spt\p}\le\|\sigma\|_{\L^\infty;\Omega}$. This
  implies that $\mutrace$ extends to a continuous linear functional on
  $\C^0_0(\Omega)$, which satisfies the resulting estimate $|\langle\mutrace;\p\rangle|
  \le\|\sigma\|_{\L^\infty;\Omega}\int_{\Omega}|\p|\,\d|\D\1_E|$ for arbitrary
  $\p\in\C^0_\cpt(\Omega)$. An application of the Riesz representation theorem
  now identifies $\mutrace$ as finite signed Radon measure with
  $|\mutrace|\le\|\sigma\|_{\L^\infty;\Omega}|\D\1_E|$ as measures on
  $\Omega$. Since we have $|\D\1_E|=\H^{n-1}\ecke(\Omega\cap\partial^\ast\!E)$
  from Theorem \ref{thm:DeGiorgi}, the claimed estimate follows.
\end{proof}

Lemma \ref{lem:measure} and the Radon-Nikod\'ym theorem yield the representation
\begin{equation}\label{eq:trace-densities}
  \mutrace=(\nortrace)\H^{n-1}\ecke(\Omega\cap\partial^\ast\!E)
\end{equation}
with a density $\nortrace\in\L^\infty(\Omega\cap\partial^\ast\!E;\H^{n-1})$ such
that $|\nortrace|\le\|\sigma\|_{\L^\infty;\Omega}$ holds $\H^{n-1}$-\ae{} on
$\Omega\cap\partial^\ast\!E$.

\begin{defi}[generalized normal traces]\label{def:gen-traces}
  Consider an open set $\Omega$ in $\R^n$, a set $E$ of finite perimeter in
  $\Omega$, and a bounded vector field $\sigma\in\L^\infty(\Omega,\R^n)$ with
  distributional divergence $\Div\sigma\in\L^1_\loc(\Omega)$. Then we call the
  density $\nortrace$ from \eqref{eq:trace-densities} the generalized normal trace
  of\/ $\sigma$ on $\Omega\cap\partial^\ast\!E$.
\end{defi}

In the setting of Definition \ref{def:gen-traces}, the formulas
\eqref{eq:int-trace-def}, \eqref{eq:ext-trace-def} can be recast in form of
the Gauss-Green formulas
\begin{align}
  \int_E\p(\Div\sigma)\dx+\int_E\sigma\cd\nabla\p\dx
  &=\int_{\partial^\ast\!E}\p\,\nortrace\,\d\H^{n-1}\,,\label{eq:int-Gauss}\\
  {-}\int_{E^\mathrm{c}}\p(\Div\sigma)\dx-\int_{E^\mathrm{c}}\sigma\cd\nabla\p\dx
  &=\int_{\partial^\ast\!E}\p\,\nortrace\,\d\H^{n-1}\,,\label{eq:ext-Gauss}
\end{align}
valid for all $\p\in\C^\infty_\cpt(\Omega)$. If we additionally assume
$\Div\sigma\in\L^1(\Omega\cap E)$ and $|\Omega\cap E|<\infty$, then
\eqref{eq:int-Gauss} stays valid for bounded functions $\p\in\C^\infty(\Omega)$
with bounded gradient $\nabla\p$ and possibly unbounded support
$\spt\p\subset\Omega$. This is straightforwardly verified by approximating $\p$
with $\eta_k\p$, where $\eta_k\in\C^\infty_\cpt(\R^n)$ are cut-off functions
with $0\le\eta_k\nearrow1$ and $|\nabla\eta_k|\le1/k$ on $\R^n$. Specifically,
we record for later application that in case $\Omega=\R^n$, we can use
$\p\equiv1$ to obtain
\begin{equation}\label{eq:int-Gauss-Rn}
  \int_E\Div\sigma\dx=\int_{\partial^\ast\!E}\,\nortrace\,\d\H^{n-1}\,.
\end{equation}
for all $E\in\BVs(\R^n)$ and all $\sigma\in\L^\infty(\R^n,\R^n)$ with
$\Div\sigma\in\L^1(\R^n)$.

\section{Isoperimetric conditions}\label{sec:isoperimetric}

In order to conveniently specify assumptions on the measure data we introduce
the following terminology (which for our main results will mostly be needed in
the small-volume version with the optimal constant $1$):

\begin{defi}[isoperimetric conditions]\label{def:ICs}
  Consider a non-negative Radon measure $\mu$ on an open set\/ $\Omega\subset\R^n$
  and $C\in{[0,\infty)}$. We say that $\mu$ satisfies the \emph{strong}
  isoperimetric condition \ka strong IC\kz{} in $\Omega$ with constant\/ $C$ if
  we have
  \begin{equation}\label{eq:mu<P}
    \mu(A^+)\le C\P(A)
    \qq\text{for all }A\in\M(\R^n)\text{ with }\overline A\subset\Omega\text{ and\/ }|A|<\infty\,.
  \end{equation}
  We say that $\mu$ satisfies the \emph{small-volume} isoperimetric condition
  \ka small-volume IC\kz{} in $\Omega$ with constant\/ $C$ if, for every
  $\eps>0$, there exists some $\delta>0$ such that we have
  \begin{equation}\label{eq:mu<P+eps}
    \mu(A^+)\le C\P(A)+\eps
    \qq\text{for all }A\in\M(\R^n)\text{ with }\overline A\subset\Omega\text{ and\/ }|A|<\delta\,.
  \end{equation}
\end{defi}

We briefly point out two equivalent reformulations of ICs in $\Omega$, which
will be treated in detail only in Section \ref{sec:admis-meas}. First, it is
equivalent to require the ICs merely for $A\Subset\Omega$ or to admit even for
$A^+\subset\Omega$ instead of $\overline A\subset\Omega$. Second, it is
equivalent to replace $\mu(A^+)$ in the ICs with $\mu(A^1)$ (or to use any other
precise representative between $A^1$ and $A^+$ at this point). The latter
possibility is in sharp contrast, however, with the necessity of sticking to
$A^+$ in the $\mu_-$-term and to $A^1$ in the $\mu_+$-term of the functional
$\Pmu$, as explained in the introduction.

We next record some basic properties which are somewhat reminiscent of the
theory of charges discussed e.\@g.\@ in \cite{Pfeffer96,BucDePPfe99}. However,
as we are not aware of a precise link between our ICs with fixed constant $C$
and that theory, we work out the details in our framework. We first recall that,
if a finite measure $\mu$ is absolutely continuous with respect to the Lebesgue
measure, then the absolute continuity of the integral gives, for every $\eps>0$
some $\delta>0$ such that we have even $\mu(A^+)=\mu(A)<\eps$ whenever
$|A|<\delta$ holds. Therefore, for this type of $n$-dimensional measures, we
trivially have the small-volume IC even with constant $0$. Back to the general
case we now show by a basic covering argument that a measure with IC cannot have
any part of dimension smaller than $n{-}1$:

\begin{lem}\label{lem:abs-con-H}
  If a Radon measure $\mu$ on open $\Omega\subset\R^n$ satisfies, for
  $C\in{[0,\infty)}$, the small-volume IC in $\Omega$ with constant\/ $C$, then,
  for every $\H^{n-1}$-negligible set\/ $N\in\Bo(\Omega)$, we have $\mu(N)=0$.
\end{lem}

\begin{proof}
  By inner regularity of $\mu$ it suffices to treat an $\H^{n-1}$-negligible
  Borel set $N\Subset\Omega$. Consider an arbitrary
  $\eps>0$ with corresponding $\delta>0$. By Lemma \ref{lem:negligible},
  there exists an open set $A$ (in particular $A\subset A^+$) such that
  $N\subset A\Subset\Omega$, $|A|<\delta$, $\P(A)<\eps$. Bringing in the
  IC, we get $\mu(N)\le\mu(A^+)\le C\P(A){+}\eps<(C{+}1)\eps$. As $\eps>0$ is
  arbitrary, this means $\mu(N)=0$.
\end{proof}

In other words, measures with IC can only have parts of dimension in
${[n{-}1,n]}$, and for the limit case of $(n{-}1)$-dimensional measures we will
actually show in Section \ref{sec:IC-for-P} that $\H^{n-1}$-rectifiable measures
satisfy the small-volume IC with constant $C$ if and only if the
$(n{-}1)$-dimensional density of $\mu$ does not exceed $2C$. Moreover, examples
with fractional dimension $\kappa$ between $n{-}1$ and $n$ can be obtained from
the basic observation that a Radon measure $\mu$ on $\R$ satisfies the strong IC
in $\R$ with constant $C$ if and only if $\mu(\R)\le2C$ holds. In particular,
for every fractal $F\in\Bo(\R)$ with $0<\H^\kappa(F)\le2C$, the measure
$\H^\kappa\ecke F$ satisfies even the strong IC in $\R$ with constant $C$. With
the help of a slicing theory similar to \cite[Theorem 18.11]{Maggi12} it follows
successively for arbitrary $n\in\N$ that the product measure
$(\H^\kappa\ecke F){\otimes}(\mathcal{L}^{n-1}\ecke{[0,1]})$ satisfies the
strong IC in $\R^n$ with constant $C$. However, since we do not work with such
fractional examples or with slicing elsewhere in this paper, we refrain from
going into details on these issues.

Next, as a technical preparation, which in the sequel ensures finiteness of
our functionals even on unbounded sets $A$, we record:

\begin{lem}\label{lem:finite}
  Consider a Radon measure $\mu$ on open $\Omega\subset\R^n$, which satisfies,
  for $C\in{[0,\infty)}$, the small-volume IC in $\Omega$ with constant\/ $C$
  or at least satisfies the defining condition \eqref{eq:mu<P+eps} for one
  fixed choice of $\eps>0$ and $\delta>0$. Then, for every $A\in\BVs(\R^n)$
  with\/ $\overline A\subset\Omega$, we have $\mu(A^+)<\infty$.
\end{lem}

\begin{proof}
  We fix $\eps$ and $\delta$ such that \eqref{eq:mu<P+eps} applies. Since we
  have $|A|<\infty$ and since $t\mapsto|A\cap({(t_0,t)}{\times}\R^{n-1})|$ is
  continuous, we can divide $\R^n$ into finitely many parallel strips
  $S_i\coleq{(t_{i-1},t_i)}{\times}\R^{n-1}$ with
  ${-}\infty=t_0<t_1<t_2<\ldots t_{k-1}<t_k=\infty$ such that
  $|A\cap S_i|<\delta$ holds for $i=1,2,\ldots,k$. Since we assumed in fact
  $A\in\BVs(\R^n)$, we have $\P(A\cap S_i)\le\P(A)<\infty$, and via the IC we
  get $\mu((A\cap S_i)^+)<\infty$ for $i=1,2,\ldots,k$. Taking into account
  $A^+\subset\bigcup_{i=1}^k(A\cap S_i)^+$, we conclude $\mu(A^+)<\infty$.
\end{proof}

At the end of this section we wish to underline that the small-volume
requirement $|A|<\delta$ in \eqref{eq:mu<P+eps} is absolutely decisive
for our purposes. As a first indication in this direction, we record
that an analogous small-\emph{diameter} IC, in which the condition
$\mathrm{diam}(A)<\delta$ substitutes for $|A|<\delta$, does not share
the same desirable features. Indeed, a compactness argument shows
that the small-diameter IC with any constant $C\in{[0,\infty)}$ for a
non-negative finite Radon measure $\mu$ on open $\Omega\subset\R^n$,
$n\ge2$, reduces to the simple requirement that $\mu$ is non-atomic
(i.\@e.\@ $\mu(\{x\})=0$ for all $x\in\Omega$). Hence, in
case\footnote{For $n=1$, in contrast, the small-volume and
small-diameter ICs with constant $C\in{(0,\infty)}$ are in fact
equivalent, since small-length sets of finite perimeter can always be
decomposed into short intervals with disjoint closures.} $n\ge2$, the
small-diameter IC admits many measures of dimension strictly smaller
than $n{-}1$ and cannot yield any semicontinuity results for the
functionals $\Pmu[\,\cdot\,;\Omega]$ considered here.

\section{Lower semicontinuity on full space}
  \label{sec:Rn}
  
After the preparations of Section \ref{sec:isoperimetric} we are ready to state,
in extension of Theorem \ref{thm:lsc-global-intro}, our main semicontinuity
result for the full-space case. The result applies under ICs on given
non-negative Radon measures $\mu_+$ and $\mu_-$ on $\R^n$ and yields lower
semicontinuity of a functional $\Pmu$, in which $\mu_+$ and $\mu_-$ are each
evaluated on a suitable representative. In fact, this functional is defined by
\begin{equation}\label{eq:P}
  \Pmu[E]\coleq\P(E)+\mu_+(E^1)-\mu_-(E^+)
\end{equation}
for $E\in\M(\R^n)$ with at least one of $\P(E){+}\mu_+(E^1)$ and $\mu_-(E^+)$
finite. In the sequel we keep $\mathcal\Pmu[E]$ well-defined either by
generally requiring finiteness of $\mu_-$ (in which case $\P(E)$ and
$\mu_+(E^1)$ may be finite or infinite) or by drawing on the ICs and Lemma
\ref{lem:finite} to ensure finiteness of all three terms in \eqref{eq:P} at
least for the restricted class of sets $E\in\BVs(\R^n)$. We find it worth
pointing out that, whenever the measures $\mu_+$ and $\mu_-$ are singular to
each other, they may be viewed as positive and negative part of a signed Radon
measure $\mu_+{-}\mu_-$, and we presently consider this the most relevant case.
However, our actual semicontinuity result does not depend on any relation
between $\mu_+$ and $\mu_-$.

\begin{thm}[lower semicontinuity on full space]\label{thm:lsc-Rn}
  Consider non-negative Radon measures $\mu_+$ and $\mu_-$ on $\R^n$, which both satisfy
  the \emph{small-volume} IC in $\R^n$ with constant\/ $1$. For a set $A_\infty\in\M(\R^n)$,
  and a sequence $(A_k)_{k\in\N}$ in $\M(\R^n)$, assume that \emph{one of the
  following sets of additional assumptions} is valid\textup{:}
  \begin{enumerate}[{\rm(a)}]
  \item The measure $\mu_-$ is \emph{finite}, and $A_k$ converge to $A_\infty$
    \emph{locally} in measure on $\R^n$.\label{item:lsc-Rn-a}
  \item The measure $\mu_-$ additionally satisfies an \emph{almost-strong IC
    with constant $1$ near $\infty$} in the sense that, for every $\eps>0$,
    there exists some $R_\eps\in{(0,\infty)}$ such that
    \begin{equation}\label{eq:strong-ic-at-infty}
      \mu_-(A^+)\le\P(A)+\eps
      \qq\text{for all }A\in\M(\R^n)\text{ with }|A\cap\B_{R_\eps}|=0\text{ and }|A|<\infty\,,
    \end{equation}
    and $A_k\in\BVs(\R^n)$ converge to $A_\infty\in\BVs(\R^n)$ \emph{locally} in
    measure on $\R^n$.\label{item:lsc-Rn-b}
  \item The sets $A_k\in\BVs(\R^n)$ converge to $A_\infty\in\BVs(\R^n)$
    \emph{globally} in measure on $\R^n$.\label{item:lsc-Rn-c}
  \end{enumerate}
  Then we have
  \begin{equation}\label{eq:lsc-Rn}
    \liminf_{k\to\infty}\Pmu[A_k]\ge\Pmu[A_\infty]
  \end{equation}
\end{thm}

We emphasize that the $\mu_+$- and $\mu_-$-terms in Theorem \ref{thm:lsc-Rn}
behave fully dual to each other only for finite measures $\mu_\pm$. In contrast,
in case of infinite measures, the $\mu_-$-term features a more subtle interplay
with the perimeter term due to the opposite signs and the resulting
well-definedness and cancellation issues whenever both these terms are infinite
or approach infinity. This is in fact the reason why the settings
\eqref{item:lsc-Rn-a}, \eqref{item:lsc-Rn-b}, \eqref{item:lsc-Rn-c} in the
theorem differ in the assumptions only on $\mu_-$ and not on $\mu_+$. In brief,
the actual differences are that in \eqref{item:lsc-Rn-a} we assume
\emph{finiteness} of $\mu_-$, that in \eqref{item:lsc-Rn-b} we impose the
\emph{almost-strong} IC near $\infty$ on $\mu_-$, and that finally in
\eqref{item:lsc-Rn-c} we have neither finiteness nor any strong IC for $\mu_-$,
but in exchange require the convergence of $A_k$ to $A_\infty$ in a more
restrictive \emph{global} $\L^1$ sense. We point out that a finite measure
$\mu_-$ generally fulfills $\lim_{R\to\infty}\mu_-((\B_R)^\c)=0$ and thus satisfies
\eqref{eq:strong-ic-at-infty}. Thus, the result under \eqref{item:lsc-Rn-a} is a
special case of the one under \eqref{item:lsc-Rn-b} when disregarding the
marginal point that in \eqref{item:lsc-Rn-a} we can formally allow infinite
perimeters of $A_k$ and $A_\infty$. Nevertheless, we believe that also the much simpler
setting \eqref{item:lsc-Rn-a} deserves its explicit recording in the above
statement (and in similar ones to follow later on).

Interestingly, having at least one of the extra features from the settings
\eqref{item:lsc-Rn-a}, \eqref{item:lsc-Rn-b}, \eqref{item:lsc-Rn-c} is necessary
for having \eqref{eq:lsc-Rn}, as shown by the following examples with sequences
$(A_k)_{k\in\N}$ which \glqq loose mass at infinity\grqq.

\begin{exmp}[for the failure of lower semicontinuity]
  For $n\ge2$, we consider the infinite Radon measure
  \[
    \mu_-\coleq2\H^{n-1}\ecke(\R^{n-1}{\times}\{0,1\})
  \]
  \ka twice the area measure on two parallel hyperplanes\kz. Then $\mu_-$
  satisfies the small-volume IC in $\R^n$ with constant\/ $1$ by Proposition
  \ref{prop:inf-model-meas} in the appendix, while it satisfies the strong IC in
  $\R^n$ and its variant of type \eqref{eq:strong-ic-at-infty} only with
  constant\/ $2$, but not with constant\/ $1$. Furthermore, for fixed
  $B\in\BVs(\R^{n-1})$ with $\P(B)<2|B|<\infty$ \ka a large ball in $\R^{n-1}$,
  for instance\kz{} and a fixed direction $0\neq v\in\R^{n-1}$, we consider the
  shifted cylinders $A_k\coleq(B{+}kv){\times}{[0,1]}\in\BVs(\R^n)$; see Figure
  \ref{fig:basic-failure-lsc} for a basic illustration. Then $A_k$ converge
  only locally, but not globally in measure on $\R^n$ to $\emptyset$, and from
  $\P(A_k)=2|B|{+}\P(B)$ and $\mu_-(A_k^+)=\mu_-(A_k)=4|B|$ we deduce
  \[
    \lim_{k\to\infty}\mathscr{P}_{0,\mu_-}[A_k]
    =\P(B){-}2|B|
    <0
    =\mathscr{P}_{0,\mu_-}[\emptyset]\,.
  \]
  Thus, lower semicontinuity of $\mathscr{P}_{0,\mu_-}$ fails along this sequence.
  \begin{figure}[h]\centering
    \includegraphics{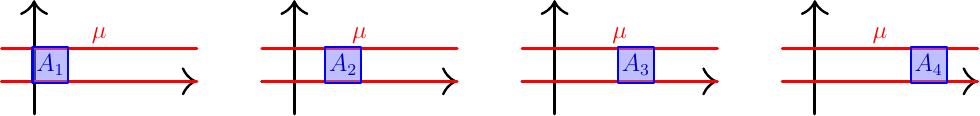}
    \vspace{-2.7ex}
    \caption{The sets $A_k$, which converge \emph{locally} in measure on $\R^2$
      to $\emptyset$, in case $n=2$, $B={[{-}1,0]}$, $v=1$.\label{fig:basic-failure-lsc}}
  \end{figure}
  
  For $n=1$, essentially the same phenomenon occurs for the measure
  $\mu_-\coleq2\H^0\ecke\mathds{Z}$ \ka with the counting measure $\H^0$\kz{}
  and $A_k\coleq I{+}k$ with any bounded interval $I\subset\R$ such that\/
  $\overline I$ contains at least two integers.\label{item:counterex}
\end{exmp}

Before proceeding to the proof of the theorem we add a brief remark on technical
infinite-volume variants of the assumptions in \eqref{item:lsc-Rn-b} and
\eqref{item:lsc-Rn-c}. While the issue is rather marginal and could also be
skipped, we find it worth recording mainly for better comparability with the
later Theorem \ref{thm:lsc-Dir}.

\begin{rem}\label{rem:lsc-Rn-infinite-vol}\text{}
    In the settings \eqref{item:lsc-Rn-b} and \eqref{item:lsc-Rn-c} of Theorem
    \ref{thm:lsc-Rn} we may replace the requirements $A_k,A_\infty\in\BVs(\R^n)$ by
    $A_k^\c,A_\infty^\c\in\BVs(\R^n)$ together with
    $\min\{\mu_+(A_k^1),\mu_-(A_k^+)\}<\infty$ and\/
    $\min\{\mu_+(A_\infty^1),\mu_-(A_\infty^+)\}<\infty$.

    \begin{proof}
      From $\P(A_k^\c)=\P(A_k)$ and $\P(A_\infty^\c)=\P(A_\infty)$ we see that
      $\Pmu[A_k]$ and $\Pmu[A_\infty]$ are still well-defined. With the result for the
      setting \eqref{item:lsc-Rn-a} at hand it suffices to consider the case
      $\mu_-(\R^n)=\infty$. Then, starting from $|A_k^\c|<\infty$ and using
      Lemma \ref{lem:finite} we infer first
      $\mu_-((A_k^+)^\c)\le\mu_-((A_k^\c)^+)<\infty$, then
      $\mu_-(A_k^+)=\infty$, then $\mu_+(A_k^1)<\infty$, and finally
      $\Pmu[A_k]={-}\infty$ for $k\gg1$. As in the same way we get
      $\Pmu[A_\infty]={-}\infty$, the semicontinuity inequality \eqref{eq:lsc-Rn} is
      trivially valid with ${-}\infty$ on both sides.
    \end{proof}

    For $n\ge2$, in view of Theorem \ref{thm:isoperi} we may express that either
    $A_k,A_\infty\in\BVs(\R^n)$ \ka as in the theorem\kz{} or
    $A_k^\c,A_\infty^\c\in\BVs(\R^n)$ \ka as in this remark\kz{} holds by requiring the
    unifying condition $|A_k\Delta A_\infty|{+}\P(A_k){+}\P(A_\infty)<\infty$.
    For $n=1$, the condition $|A_k\Delta A_\infty|{+}\P(A_k){+}\P(A_\infty)<\infty$ includes
    further cases, but still semicontinuity remains valid in all of these \ka
    as it can be read off from the later proofs or the refined results in Theorems
    \ref{thm:lsc-Dir}, \ref{thm:lsc-BV-dom}, \ref{thm:lsc-dom} and, in fact, in
    the one-dimensional situation can also be proved by much simpler means\kz.
\end{rem}

The proof of Theorem \ref{thm:lsc-Rn} is approached step by step and
will be finalized only at the end of this section. We start by establishing an
approximation lemma, which is illustrated in Figure \ref{fig:cut-off-tentacle}
and plays a key role.

\begin{lem}[good exterior approximation]\label{lem:good-ext-approx}
  For a non-negative Radon measure $\mu$ on $\R^n$ with $\mu(N)=0$ for all\/
  $\H^{n-1}$-negligible $N\in\Bo(\R^n)$, assume that condition
  \eqref{eq:mu<P+eps} holds in $\Omega=\R^n$ for some fixed choice of
  $\eps>0$, $\delta>0$, and $C\in{[0,\infty)}$. Then, if a sequence
  $(A_k)_{k\in\N}$ in $\BVs(\R^n)$ converges globally in measure on $\R^n$ to
  $A_\infty\in\BVs(\R^n)$, there exists a Borel set $S\in\BVs(\R^n)$ such that we have
  \[
    A_\infty^+\subset\inn(S)\,,\qq\qq
    \mu\big(\overline S\big)<\mu(A_\infty^+)+3\eps\,,
    \qq\qq\text{and}\qq\qq
    \liminf_{k\to\infty}\P(S,A_k^+)<\eps\,.
  \]
\end{lem}

\vspace{-3ex}

\begin{figure}[H]\centering
  \includegraphics{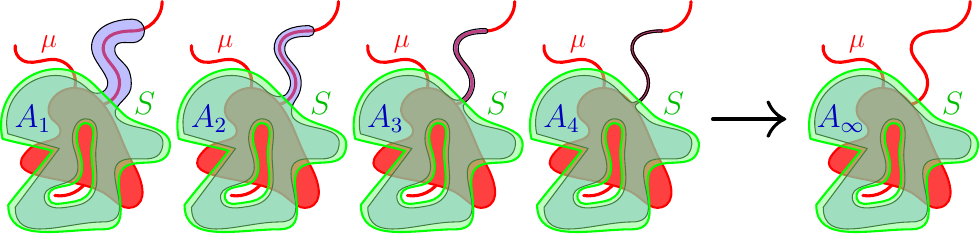}
  \vspace{-2.7ex}
  \caption{A set $S$ which cuts off the tentacle of Figure \ref{fig:tentacle} in
    the sense of Lemma \ref{lem:good-ext-approx} (for mildly small
    $\eps$).\label{fig:cut-off-tentacle}}
\end{figure}

\begin{proof}[Proof of Lemma \ref{lem:good-ext-approx}]
  We first treat the main case $n\ge2$. Applying Lemma \ref{lem:mon-approx-u+}
  to $\1_{A_\infty}\in\BV(\R^n)$, we find $v_\ell\in\W^{1,1}(\R^n)$ such that
  $1\ge v_1\ge v_2\ge v_3\ge\ldots$ holds \ae{} on $\R^n$ and $v_\ell^\ast$
  converge $\H^{n-1}$-\ae{} on $\R^n$ to $\1_{A_\infty^+}$. By assumption on
  $\mu$, this convergence holds also $\mu$-\ae{} on $\R^n$. Next, possibly
  decreasing $\delta>0$ from the statement, we can assume
  $C(\delta/\Gamma_n)^\frac{n-1}n\le\eps$ for the constant $\Gamma_n$ of Theorem
  \ref{thm:isoperi}. Lemma \ref{lem:qc} then gives open sets $E_\ell$ in $\R^n$
  with $\Cp_1(E_\ell)<(\delta/\Gamma_n)^\frac{n-1}n$ (and in particular
  $|E_\ell|<\infty$) such that $v_\ell^\ast$ is defined and continuous on
  $\R^n\setminus E_\ell$. From Proposition \ref{prop:inf-P} we further obtain
  $H_\ell\in\BVs(\R^n)$ with $\P(H_\ell)<(\delta/\Gamma_n)^\frac{n-1}n$ such
  that $E_\ell\subset H_\ell^+$. By the isoperimetric estimate of Theorem
  \ref{thm:isoperi} we infer
  $|H_\ell|\le\Gamma_n\P(H_\ell)^\frac n{n-1}<\delta$, and via
  \eqref{eq:mu<P+eps} we end up with $\mu(H_\ell^+)\le C\P(H_\ell){+}\eps
  <C(\delta/\Gamma_n)^\frac{n-1}n{+}\eps\le2\eps$. For the following we can thus
  record
  \begin{equation}\label{eq:prop-Hl}
    \mu(H_\ell^+)<2\eps
    \qq\qq\text{and}\qq\qq
    \P(H_\ell)<\eps\,.
  \end{equation}
  Next we observe that $\P(\{v_\ell>t\})<\infty$ holds for \ae{} $t\in{(0,1)}$ by
  Theorem \ref{thm:BV-coarea}. Furthermore, with the help of Fatou's lemma,
  again Theorem \ref{thm:BV-coarea}, $|A_k^+\Delta A_k|=0$,
  $\lim_{k\to\infty}|A_k\Delta A_\infty|=0$, and $v_\ell\equiv1$, $\nabla
  v_\ell\equiv0$ \ae{} on $A_\infty$ we obtain
  \[
    \int_0^1\liminf_{k\to\infty}\P(\{v_\ell>t\},A_k^+)\,\d t
    \le\liminf_{k\to\infty}\int_0^1\P(\{v_\ell>t\},A_k^+)\,\d t
    =\liminf_{k\to\infty}\int_{A_k}|\nabla v_\ell|\dx
    =\int_{A_\infty}|\nabla v_\ell|\dx=0\,.
  \]
  As a consequence, we have
  $\liminf_{k\to\infty}\P(\{v_\ell>t\},A_k^+)=0$ for
  \ae{} $t\in{(0,1)}$. Therefore, we can choose a level $t_\mathrm{o}\in{(0,1)}$
  such that we have
  \[
    \P(\{v_\ell>t_\mathrm{o}\})<\infty
    \qq\text{for all }\ell\in\N
  \]
  and
  \begin{equation}\label{eq:tl-liminf}
    \liminf_{k\to\infty}\P(\{v_\ell>t_\mathrm{o}\},A_k^+)=0
    \qq\text{for all }\ell\in\N\,.
  \end{equation}
  Moreover, since the measure $\mu$ has positive mass on at most countably many
  level sets\footnote{Since $v_\ell^\ast$ is defined $\H^{n-1}$-\ae{} and then
  by assumption also $\mu$-ae{}, the level sets $\{v_\ell^\ast=t\}$ are defined
  up to $\mu$-negligible sets, and this will suffice for our purposes. Clearly,
  one may also agree on a concrete convention such as simply excluding the
  non-existence points of $v_\ell^\ast$ from the level sets.} $\{v_\ell^\ast=t\}$
  with $t\in{(0,1)}$, the choice can be made such that additionally
  $\mu(\{v_1^\ast=t_\mathrm{o}\})=0$. Now we introduce the sets\footnote{The
  sets $U_\ell$ are defined up to single points, since the non-existence set of
  $v_\ell^\ast$ is contained in $E_\ell$.}
  \[
    U_\ell\coleq\{v_\ell^\ast>t_\mathrm{o}\}\setminus E_\ell\,.
  \]
  We observe that $U_\ell$ are relatively open in $\R^n\setminus E_\ell$ with
  $\overline{U_\ell}\subset\{v_\ell^\ast\ge t_\mathrm{o}\}\setminus E_\ell$ by the
  openness of $E_\ell$ and the continuity of $v_\ell^\ast$ outside $E_\ell$.
  Furthermore, we can estimate
  \begin{equation}\label{eq:pre-Ul-est}
    \mu\big(\overline{U_\ell}\big)
    \le\mu(\{v_\ell^\ast\ge t_\mathrm{o}\})\,.
  \end{equation}
  Here, from $v_1\in\L^1(\R^n)$, $\P(\{v_1>t_\mathrm{o}\})<\infty$, and Lemma
  \ref{lem:finite} we infer
  $\mu(\{v_1^\ast>t_\mathrm{o}\}\setminus E_1)\le\mu(\{v_1>t_\mathrm{o}\}^+)<\infty$.
  Then in view of $\mu(\{v_1^\ast=t_\mathrm{o}\})=0$ we get also
  $\mu(\{v_1^\ast\ge t_\mathrm{o}\})
  \le\mu(\{v_1^\ast>t_\mathrm{o}\}\setminus E_1)+\mu(H_1^+)<\infty$.
  Combining this with the $\mu$-\ae{} monotone convergence
  $v_\ell^\ast\to\1_{A_\infty^+}$, we conclude that the right-hand side
  $\mu(\{v_\ell^\ast\ge t_\mathrm{o}\})$ in \eqref{eq:pre-Ul-est} converges to
  $\mu(\{\1_{A_\infty^+}\ge t_\mathrm{o}\})=\mu(A_\infty^+)$ for
  $\ell\to\infty$. Therefore, for a suitably large $\ell\in\N$, which we fix at
  this point for the remainder of the proof, we have
  \begin{equation}\label{eq:Ul-est}
    \mu\big(\overline{U_\ell}\big)<\mu(A_\infty^+)+\eps\,.
  \end{equation}
  Now we are ready to introduce
  \[
    S\coleq U_\ell\cup H_\ell^+\,,
  \]
  and using $E_\ell\subset H_\ell^+$ we see
  \[
    A_\infty^+\subset\{v_\ell^\ast=1\}\cup E_\ell\subset U_\ell\cup E_\ell\subset S\,.
  \]
  Since $U_\ell$ is relatively open in $\R^n\setminus E_\ell$ and $E_\ell$ is
  open in $\R^n$, also $U_\ell\cup E_\ell$ is open in $\R^n$, and we can deduce
  even
  \[
    A_\infty^+\subset\inn(S)\,.
  \]
  Furthermore, from \eqref{eq:Ul-est} and \eqref{eq:prop-Hl} we infer
  \[
    \mu\big(\overline S\big)
    \le\mu\big(\overline{U_\ell}\big)+\mu(H_\ell^+)
    \le\mu(A_\infty^+)+3\eps\,.
  \]
  At this stage we observe $S=\{v_\ell>t_\mathrm{o}\}\cup H_\ell$ up to negligible
  sets with $\{v_\ell>t_\mathrm{o}\},H_\ell\in\BVs(\R^n)$. Thus, by Lemma
  \ref{lem:P-cup-cap} we obtain $S\in\BVs(\R^n)$ and
  $\P(S,\,\cdot\,)\le\P(\{v_\ell>t_\mathrm{o}\},\,\cdot\,)+\P(H_\ell,\,\cdot\,)$. Therefore,
  involving also \eqref{eq:tl-liminf} and \eqref{eq:prop-Hl} we can estimate
  \[
    \liminf_{k\to\infty}\P(S,A_k^+)
    \le\liminf_{k\to\infty}\P(\{v_\ell>t_\mathrm{o}\},A_k^+)+\P(H_\ell)
    <\eps\,.
  \]
  At this point, all claims on $S$ are verified.

  Finally, in the simpler case $n=1$ the previous reasoning applies with major
  simplifications, which are mostly due to the full continuity of $\W^{1,1}(\R)$
  functions. In particular there is no need to construct $E_\ell$ and $H_\ell$,
  which can be replaced with $\emptyset$, and one can directly obtain an open
  set $S=U_\ell=\{v_\ell^\ast>t_\mathrm{o}\}$.
\end{proof}

With the lemma at hand, we now proceed to a proof of Theorem
\ref{thm:lsc-Rn}\eqref{item:lsc-Rn-c}, which corresponds to Theorem
\ref{thm:lsc-global-intro} from the introduction. We start with the special case
$\mu_+\equiv0$, which is here restated as follows.

\begin{prop}[$\L^1$ lower semicontinuity in case $\mu_+\equiv0$]
    \label{prop:lsc-global}
  Consider a non-negative Radon measure $\mu$ on $\R^n$ which satisfies the
  small-volume IC in $\R^n$ with constant\/ $1$. Moreover,
  assume that $A_k\in\BVs(\R^n)$ converge \emph{globally} in measure on $\R^n$
  to $A_\infty\in\BVs(\R^n)$. Then we have
  \begin{equation}\label{eq:lsc}
    \liminf_{k\to\infty}\big[\P(A_k)-\mu(A_k^+)\big]\ge\P(A_\infty)-\mu(A_\infty^+)\,.
  \end{equation}
\end{prop}

\begin{proof}
  Possibly passing to a subsequence, we can assume that
  $\lim_{k\to\infty}\big[\P(A_k){-}\mu(A_k^+)\big]$ exists. We now fix an
  arbitrary $\eps>0$. Drawing on Lemma \ref{lem:abs-con-H} and the assumed
  IC, we then apply Lemma \ref{lem:good-ext-approx} with
  the given $\eps$, the corresponding $\delta$, and $C=1$, and we work with the
  corresponding set $S\in\BVs(\R^n)$. We start by splitting terms in the
  sense of the inequality
  \[
    \P(A_k)-\mu(A_k^+)
    \ge\P(A_k,\inn(S))
    -\mu\big(\overline{S}\big)
    +\P(A_k,\inn(S)^\c)-\mu\big(A_k^+\setminus\overline{S}\big)\,.
  \]
  Then we use the elementary rule
  $\lim_{k\to\infty}[a_k{+}b_k]\ge\liminf_{k\to\infty}a_k+\limsup_{k\to\infty}b_k$
  for $a_k,b_k\in\R$, valid whenever the limit on the left-hand side exists
  and the right-hand side does not yield the undefined expression
  $\infty-\infty$. By the initial assumption and the observation
  that neither ${-}\mu(\overline S)$ nor $\limsup\big[\ldots\big]$ equal
  ${-}\infty$ (see the subsequent estimate \eqref{eq:lsc-third} for the
  latter), we may write
  \begin{equation}\label{eq:lsc-main}
    \lim_{k\to\infty}\big[\P(A_k)-\mu(A_k^+)\big]
    \ge\liminf_{k\to\infty}\P(A_k,\inn(S))
    -\mu\big(\overline{S}\big)
    +\limsup_{k\to\infty}\big[\P(A_k,\inn(S)^\c)-\mu\big(A_k^+\setminus\overline{S}\big)\big]\,.
  \end{equation}
  The terms on the right-hand side of \eqref{eq:lsc-main} are now estimated
  separately. For the first term, by an application of Lemma \ref{lem:lsc-per}
  on the open set $\inn(S)$ and the inclusion $A_\infty^+\subset\inn(S)$ from
  Lemma \ref{lem:good-ext-approx}, we have
  \begin{equation}\label{eq:lsc-first}
    \liminf_{k\to\infty}\P(A_k,\inn(S))\ge\P(A_\infty,\inn(S))\ge\P(A_\infty,A_\infty^+)=\P(A_\infty)\,.
  \end{equation}
  For the second term, the estimate
  \begin{equation}\label{eq:lsc-second}
    \mu\big(\overline{S}\big)<\mu(A_\infty^+)+3\eps\,,
  \end{equation}
  also provided by Lemma \ref{lem:good-ext-approx}, suffices. In order to
  control the last term in \eqref{eq:lsc-main}, we first record that in view of
  $A_\infty^+\subset S$ we get
  $|A_k{\setminus}S|\le|A_k{\setminus}A_\infty|\le|A_k\Delta A_\infty|$ and that
  consequently the assumed global convergence implies
  $\lim_{k\to\infty}|A_k{\setminus}S|=0$. This permits the crucial
  application of the small-volume IC with constant $1$ to $A_k{\setminus}S$
  for $k\gg1$, which is now combined with the inclusion
  $A_k^+\setminus\overline{S}\subset(A_k\setminus S)^+$, Lemma
  \ref{lem:P(AcapB),P(A-S)}, and the inclusion $S^0\subset\inn(S)^\c$. All in
  all, for $k\gg1$, we deduce
  \[
    \mu\big(A_k^+\setminus\overline{S}\big)
    \le\mu((A_k\setminus S)^+)
    \le\P(A_k\setminus S)+\eps
    \le\P(A_k,S^0)+\P(S,A_k^+)+\eps
    \le\P(A_k,\inn(S)^\c)+\P(S,A_k^+)+\eps\,.
  \]
  Now we rearrange terms in the resulting estimate and take limits. Then, also
  employing the last property from Lemma \ref{lem:good-ext-approx}, we conclude
  \begin{equation}\label{eq:lsc-third}
    \limsup_{k\to\infty}\big[\P(A_k,\inn(S)^\c)-\mu\big(A_k^+\setminus\overline{S}\big)\big]
    \ge-\liminf_{k\to\infty}\P(S,A_k^+)-\eps
    >{-}2\eps\,.
  \end{equation}
  Collecting the estimates \eqref{eq:lsc-main}, \eqref{eq:lsc-first},
  \eqref{eq:lsc-second}, and \eqref{eq:lsc-third} we finally arrive at
  \[
    \lim_{k\to\infty}\big[\P(A_k)-\mu(A_k^+)\big]
    \ge\P(A_\infty)-\mu(A_\infty^+)-5\eps\,.
  \]
  Since $\eps>0$ is arbitrary, with this we have proven the claim \eqref{eq:lsc}.
\end{proof}

Next, essentially by passing to complements, we establish a variant of
Proposition \ref{prop:lsc-global} with opposite sign convention for the measure
$\mu$. This dual statement is analogous except for the fact that in the dual
case we can allow for \emph{local} convergence of sets of potentially infinite
perimeter, while in the original case we cannot generally relax the
corresponding global assumptions. In terms of the general Theorem
\ref{thm:lsc-Rn} this means that we achieve a treatment of the setting
\eqref{item:lsc-Rn-a} with $\mu_-\equiv0$.

\begin{prop}[$\L^1_\loc$ lower semicontinuity in case $\mu_-\equiv0$]
    \label{prop:lsc-global'}
  Consider a non-negative Radon measure $\mu$ on $\R^n$ which satisfies the
  small-volume IC in $\R^n$ with constant\/ $1$. Moreover, assume that
  $A_k\in\M(\R^n)$ converge \emph{locally} in measure on $\R^n$ to
  $A_\infty\in\M(\R^n)$. Then we have
  \begin{equation}\label{eq:lsc'}
    \liminf_{k\to\infty}\big[\P(A_k)+\mu(A_k^1)\big]
    \ge\P(A_\infty)+\mu(A_\infty^1)\,.
  \end{equation}
\end{prop}

We remark that the deduction of Proposition \ref{prop:lsc-global'} from
Proposition \ref{prop:lsc-global} is quite straightforward \emph{if} $A_k$ are
uniformly bounded and thus we can simply take complements in a fixed, suitably
large ball $B\subset\R^n$ (for which we clearly have $B\in\BVs(\R^n)$ and
$\mu(B)<\infty$). However, in general we are not in this situation, and thus in
the following proof we need additional cut-off arguments.

\begin{proof}[Proof of Proposition \ref{prop:lsc-global'}]
  As usual we can assume that $\lim_{k\to\infty}\big[\P(A_k){+}\mu(A_k^1)\big]$
  exists and is finite. Taking into account the sign of the $\mu$-term we can
  further assume $\sup_{k\in\N}\P(A_k)<\infty$, which implies $\P(A_\infty)<\infty$
  by Lemma \ref{lem:lsc-per}. Next, by a classical version of the coarea formula
  (which can be seen as the case $u(x)=|x|$ in either Theorem
  \ref{thm:BV-coarea} or Theorem \ref{thm:Lip-coarea}), for
  every $R_0\in{(0,\infty)}$ we have
  \[\begin{aligned}
    \int_0^{R_0}\liminf_{k\to\infty}\H^{n-1}((A_k^0\Delta A_\infty^0)\cap\partial\B_R)\,\d R
    &\le\liminf_{k\to\infty}\int_0^{R_0}\H^{n-1}((A_k^0\Delta A_\infty^0)\cap\partial\B_R)\,\d R\\
    &=\liminf_{k\to\infty}|(A_k^0\Delta A_\infty^0)\cap\B_{R_0}|=0\,,
  \end{aligned}\]
  and thus $\liminf_{k\to\infty}\H^{n-1}((A_k^0\Delta A_\infty^0)\cap\partial\B_R)=0$
  holds for \ae{} $R\in{(0,\infty)}$. In addition, the Radon measures
  $\gamma_k\coleq\P(A_k,\,\cdot\,)+\P(A_\infty,\,\cdot\,)+\mu$ satisfy
  $\gamma_k(\partial\B_R)=0$ for all but at most countably many
  $R\in{(0,\infty)}$. Altogether, this allows to choose radii
  $R_i\in{(0,\infty)}$ with $\lim_{i\to\infty}R_i=\infty$ such that, for the
  corresponding open balls $B_i\coleq\B_{R_i}$ centered at $0$, we have
  \begin{align}
    \P(A_k,\partial B_i)=\P(A_\infty,\partial B_i)=0
      &\qq\text{for all }i,k\in\N\,,\label{eq:P-Ak-dB=0}\\
    \mu(\partial B_i)=0
      &\qq\text{for all }i\in\N\,,\label{eq:mu-dB=0}\\
    \liminf_{k\to\infty}\H^{n-1}((A_k^0\Delta A_\infty^0)\cap\partial\B_i)=0
      &\qq\text{for all }i\in\N\,.\nonumber
  \end{align}
  Here, by successively passing to subsequences of $A_k$ and using a diagonal
  sequence argument, the last property can be strengthened to hold with $\lim$
  in place of $\liminf$ and then also gives
  \begin{equation}\label{eq:conv-cut-off-mass}
    \lim_{k\to\infty}\H^{n-1}(A_k^0\cap\partial B_i)=\H^{n-1}(A_\infty^0\cap\partial B_i)
    \qq\text{for all }i\in\N\,.
  \end{equation}
  Now, for arbitrary $i\in\N$, we consider the complements $B_i\setminus A_k$,
  which converge for $k\to\infty$ in measure to $B_i\setminus A_\infty$. (Observe here
  that indeed \emph{local} convergence in measure of $A_k$ implies \emph{global}
  convergence in measure of the bounded sets $B_i\setminus A_k$.) Hence, by an
  application of Proposition \ref{prop:lsc-global}, we get
  \begin{equation}\label{eq:lsc-complements}
    \liminf_{k\to\infty}\big[\P(B_i\setminus A_k)-\mu((B_i\setminus A_k)^+)\big]
    \ge\P(B_i\setminus A_\infty)-\mu((B_i\setminus A_\infty)^+)\,.
  \end{equation}
  We now estimate and rewrite terms in \eqref{eq:lsc-complements}. On one hand
  we exploit \eqref{eq:P-Ak-dB=0} (which can also be written as
  $\H^{n-1}(\partial^\ast\!A_k\cap\partial
  B_i)=\H^{n-1}(\partial^\ast\!A_\infty\cap\partial B_i)=0$) in order to apply the
  equality case of \eqref{eq:P(A-S)} in Lemma \ref{lem:P(AcapB),P(A-S)}. In this
  way we derive
  \begin{gather*}
    \P(B_i\setminus A_k)=\P(A_k,B_i^+)+\P(B_i,A_k^0)
    =\P(A_k,B_i)+\H^{n-1}(A_k^0\cap\partial B_i)\,,\\
    \P(B_i\setminus A_\infty)=\P(A_\infty,B_i^+)+\P(B_i,A_\infty^0)
    =\P(A_\infty,B_i)+\H^{n-1}(A_\infty^0\cap\partial B_i)\,.
  \end{gather*}
  On the other hand, keeping \eqref{eq:mu-dB=0} in mind, we have
  \begin{gather*}
    \mu((B_i\setminus A_k)^+)=\mu(B_i\setminus A_k^1)=\mu(B_i)-\mu(A_k^1\cap B_i)\,,\\
    \mu((B_i\setminus A_\infty)^+)=\mu(B_i\setminus A_\infty^1)=\mu(B_i)-\mu(A_\infty^1\cap B_i)\,.
  \end{gather*}
  We plug these findings into \eqref{eq:lsc-complements} and are left with
  \begin{multline*}
    \liminf_{k\to\infty}\big[\P(A_k,B_i)+\mu(A_k^1\cap B_i)+\H^{n-1}(A_k^0\cap\partial B_i)\big]-\mu(B_i)\\
    \ge\P(A_\infty,B_i)+\mu(A_\infty^1\cap B_i)+\H^{n-1}(A_\infty^0\cap\partial B_i)-\mu(B_i)\,.
  \end{multline*}
  Adding the finite number $\mu(B_i)$ and subtracting the finite number in
  \eqref{eq:conv-cut-off-mass}, the inequality reduces to
  \[
    \liminf_{k\to\infty}\big[\P(A_k,B_i)+\mu(A_k^1\cap B_i)\big]
    \ge\P(A_\infty,B_i)+\mu(A_\infty^1\cap B_i)\,.
  \]
  At this stage, we further enlarge the terms on the left-hand side and use the
  initial assumption on the existence of the limit to get
  \[
    \lim_{k\to\infty}\big[\P(A_k)+\mu(A_k^1)\big]
    \ge\P(A_\infty,B_i)+\mu(B_i\cap A_\infty^1)\,.
  \]
  Finally, sending $i\to\infty$ and taking into account
  $\lim_{i\to\infty}R_i=\infty$, we arrive at the claim \eqref{eq:lsc'}.
\end{proof}

By combining Propositions \ref{prop:lsc-global} and \ref{prop:lsc-global'} we
are able to treat the global-convergence setting \eqref{item:lsc-Rn-c} in
Theorem \ref{thm:lsc-Rn} in its full generality.

\begin{proof}[Proof of Theorem \ref{thm:lsc-Rn} under assumptions
    \eqref{item:lsc-Rn-c}]
  For $A_k$ and $A_\infty$ as in the statement, we record that both
  $A_k\cup A_\infty\in\BVs(\R^n)$ and $A_k\cap A_\infty\in\BVs(\R^n)$ converge globally in
  measure to $A_\infty$. Then, since we assumed the small-volume IC for both $\mu_+$
  and $\mu_-$, we can apply Proposition \ref{prop:lsc-global} to $A_k\cup A_\infty$ and
  Proposition \ref{prop:lsc-global'} to $A_k\cap A_\infty$ to deduce
  \begin{align*}
    \liminf_{k\to\infty}\big[\P(A_k\cup A_\infty)-\mu_-((A_k\cup A_\infty)^+)\big]&\ge\P(A_\infty)-\mu_-(A_\infty^+)\,,\\
    \liminf_{k\to\infty}\big[\P(A_k\cap A_\infty)+\mu_+((A_k\cap A_\infty)^1)\big]&\ge\P(A_\infty)+\mu_+(A_\infty^1)\,.
  \end{align*}
  We now add these two inequalities and use \eqref{eq:P-cup-cap} in the
  form $\P(A_k\cup A_\infty)+\P(A_k\cap A_\infty)\le\P(A_k)+\P(A_\infty)$ together with
  $(A_k\cup A_\infty)^+=A_k^+\cup A_\infty^+\supset A_k^+$ and
  $(A_k\cap A_\infty)^1=A_k^1\cap A_\infty^1\subset A_k^1$. Then we end up with
  \[
    \P(A_\infty)+\liminf_{k\to\infty}\big[\P(A_k)+\mu_+(A_k^1)-\mu_-(A_k^+)\big]
    \ge2\P(A_\infty)+\mu_+(A_\infty^1)-\mu_-(A_\infty^+)\,,
  \]
  which by subtraction of $\P(A_\infty)$ yields the claim in \eqref{eq:lsc-Rn}.
\end{proof}

Before treating the remaining settings and finalizing the discussion of
semicontinuity on the full space, we record the following localized
semicontinuity property, which comes out from the cut-off argument in the proof
of Proposition \ref{prop:lsc-global'} and a \glqq dual\grqq{} variant of this
argument. This localized statement will in fact be very convenient in the
sequel.

\begin{lem}[localized semicontinuity]\label{lem:loc-lsc}
  Consider non-negative Radon measures $\mu_+$ and $\mu_-$ on $\R^n$ which both
  satisfy the small-volume IC in $\R^n$ with constant\/ $1$. If $A_k\in\M(\R^n)$
  converge to $A_\infty\in\M(\R^n)$ locally in measure in $\R^n$, then, for every
  $R\in{(0,\infty)}$, we have
  \[
    \liminf_{k\to\infty}\big[\P(A_k,\B_R)+\mu_+(A_k^1\cap\B_R)-\mu_-(A_k^+\cap\B_R)\big]
    \ge\P(A_\infty,\B_R)+\mu_+(A_\infty^1\cap\B_R)-\mu_-(A_\infty^+\cap\B_R)\,.
  \]
\end{lem}

\begin{proof}
  We first establish the claim simultaneously for the case $\mu_-\equiv0$, in
  which we set $\mu\coleq\mu_+$, and for the case $\mu_+\equiv0$, in which we
  set $\mu\coleq\mu_-$. For the case $\mu_-\equiv0$ we can follow quite closely
  the lines of the proof of Proposition \ref{prop:lsc-global'}, while for the
  case $\mu_+\equiv0$ we use an analogous but dual argument based on the
  convergence of $A_k\cap B_i$ to $A_\infty\cap B_i$. In the sequel we only point out
  the relevant modifications. First of all, we now work with a fixed
  $R\in{(0,\infty)}$ and may initially assume existence and finiteness of
  $\lim_{k\to\infty}\big[\P(A_k,\B_R){-}\mu(A_k^+\cap\B_R)\big]$ and
  $\lim_{k\to\infty}\big[\P(A_k,\B_R){+}\mu(A_k^1\cap\B_R)\big]$, respectively,
  which leads to $\sup_{k\in\N}\P(A_k,\B_R)<\infty$ and $\P(A_\infty,\B_R)<\infty$
  (where we have exploited $\mu(\B_R)<\infty$ in case $\mu_+\equiv0$). Then, the
  good radii $R_i$ are taken in ${(0,R)}$ with $\lim_{i\to\infty}R_i=R$, where
  in case $\mu_+\equiv0$ the coarea argument is implemented with $A_k^1$ and
  $A_\infty^1$ instead of $A_k^0$ and $A_\infty^0$ to subsequently achieve
  $\lim_{k\to\infty}\H^{n-1}(A_k^1\cap\partial B_i)
  =\H^{n-1}(A_\infty^1\cap\partial B_i)$ in place of \eqref{eq:conv-cut-off-mass}. The
  remainder of the reasoning stays unchanged in case $\mu_-\equiv0$ and in case
  $\mu_+\equiv0$ is done with $A_k\cap B_i$ and $A_\infty\cap B_i$ instead of
  $B_i\setminus A_k$ and $B_i\setminus A_\infty$ (which slightly simplifies the
  handling of the $\mu$-terms). When adapting the final step in the proof of
  Proposition \ref{prop:lsc-global'} to the case $\mu_+\equiv0$, we may no
  longer pass from ${-}\mu(A_k^+\cap B_i)$ to ${-}\mu(A_k^+\cap B_R)$ on the
  left-hand side by simply enlarging the term, but we can still conclude, as in
  view of $\mu(\B_R)<\infty$ we have
  $\lim_{i\to\infty}\mu(A_k^+\cap B_i)=\mu(A_k^+\cap B_R)$ uniformly in $k$.

  Finally, in order to reach the general case, in which both $\mu_+$ and $\mu_-$
  do not vanish, we return to the reasoning used above to prove Theorem
  \ref{thm:lsc-Rn} in the setting \eqref{item:lsc-Rn-c}. The adaptation of this
  reasoning to a ball $\B_R$ is straightforward and exploits
  \eqref{eq:P-cup-cap} in the form
  $\P(A_k\cup A_\infty,\B_R)+\P(A_k\cap A_\infty,\B_R)\le\P(A_k,\B_R)+\P(A_\infty,\B_R)$.
\end{proof}

We proceed by addressing the proof of semicontinuity in the settings
\eqref{item:lsc-Rn-a} and \eqref{item:lsc-Rn-b} of Theorem \ref{thm:lsc-Rn}.
We only sketch the relevant arguments, since we will later provide further
details in connection with even more general cases contained in Theorem
\ref{thm:lsc-Dir}.

In fact, in order to complete the treatment of the setting \eqref{item:lsc-Rn-a}
the observation needed is essentially the one that, for finite measures, the
cases $\mu_+\equiv0$ and $\mu_-\equiv0$ are fully dual to each other:

\begin{proof}[Sketch of proof for Theorem \ref{thm:lsc-Rn} under assumptions
    \eqref{item:lsc-Rn-a}]
  In case $\mu_-\equiv0$ the claim is covered by Proposition
  \ref{prop:lsc-global'}. Moreover, we can move back from this case to the case
  $\mu_+\equiv0$ once more by taking complements. Indeed, since we are assuming
  $\mu_-(\R^n)<\infty$, this works rather straightforwardly by exploiting
  $\P(A_k^\c)=\P(A_k)$ and $\mu_-((A_k^\c)^1)=\mu_-(\R^n)-\mu_-(A_k^+)$ together
  with the analogous formulas for $A_\infty^\c$. Alternatively, we can obtain the claim
  in the case $\mu_+\equiv0$ by passing $R\to\infty$ in the case $\mu_+\equiv0$
  of Lemma \ref{lem:loc-lsc}. Finally, the general case with non-zero $\mu_+$
  and $\mu_-$ can be reached by the same reasoning used under assumptions
  \eqref{item:lsc-Rn-c}.
\end{proof}

In connection with the setting \eqref{item:lsc-Rn-b} the final key observation
is that the strong IC for $\mu_-$ keeps cut-off terms
(almost) non-negative and prevents the failure of lower semicontinuity at
$\infty$:

\begin{proof}[Sketch of proof for Theorem \ref{thm:lsc-Rn} under assumptions
    \eqref{item:lsc-Rn-b}]
  Once more the case $\mu_-\equiv0$ is covered by Proposition
  \ref{prop:lsc-global'}, and once we manage to additionally treat the case
  $\mu_+\equiv0$, the general case follows as well. Thus, we now describe yet
  another cut-off argument used to deal with the case $\mu_+\equiv0$. As usual
  we assume that the $\liminf$ in \eqref{eq:lsc-Rn} is in fact a limit. By Lemma
  \ref{lem:loc-lsc} we have
  \begin{equation}\label{eq:proof-str-ic-1}
    \liminf_{k\to\infty}[\P(A_k,\B_R){-}\mu_-(A_k^+\cap\B_R)]
    \ge\P(A_\infty,\B_R){-}\mu_-(A_\infty^+\cap\B_R)
  \end{equation}
  for all $R\in{(0,\infty)}$. For arbitrary $\eps>0$, we claim that we can
  choose balls $B_i=\B_{R_i}$ with $R_i\in{(R_\eps,\infty)}$ and
  $\lim_{i\to\infty}R_i=\infty$ such that $\mu_-(\partial B_i)=0$ and
  \begin{equation}\label{eq:proof-str-ic-2}
    \lim_{k\to\infty}\H^{n-1}(A_k^1\cap\partial B_i)
    =\H^{n-1}(A_\infty^1\cap\partial B_i)<\eps
  \end{equation}
  hold for all $i\in\N$ and at least for a subsequence of $(A_k)_{k\in\N}$, to
  which we pass without reflecting this in notation. Indeed, the condition
  $\mu_-(\partial B_i)=0$ and the convergence of the $\H^{n-1}$-measures
  in \eqref{eq:proof-str-ic-2} have already been discussed (see the proofs of
  Proposition \ref{prop:lsc-global'} and Lemma \ref{lem:loc-lsc}), while the
  $\eps$-bound in \eqref{eq:proof-str-ic-2} can be achieved by writing out
  $|A_\infty^1|<\infty$ via the coarea formula in a similar way. From
  $\mu_-(\partial B_i)=0$, the almost-strong IC with constant $1$ near $\infty$
  (applicable for $A_k\cap B_i^\c$ in view of $R_i>R_\eps$), and Lemma
  \ref{lem:P(AcapB),P(A-S)} we get
  \[\begin{aligned}
    \mu_-\big(A_k^+\cap B_i^\c\big)
    &=\mu_-((A_k\cap B_i^\c)^+)\\
    &\le\P(A_k\cap B_i^\c)+\eps\\
    &\le\P(A_k,(B_i^\c)^+)+\P(B_i^\c,A_k^1)+\eps\\
    &=\P(A_k,B_i^\c)+\H^{n-1}(A_k^1\cap\partial B_i)+\eps\,.
  \end{aligned}\]
  Rearranging terms and bringing in \eqref{eq:proof-str-ic-2} then gives
  control on the terms cut off in the sense of
  \begin{equation}\label{eq:proof-str-ic-3}
    \liminf_{k\to\infty}[\P(A_k,B_i^\c){-}\mu_-(A_k^+\cap B_i^\c)]
    \ge{-}\eps-\lim_{k\to\infty}\H^{n-1}(A_k^1\cap\partial B_i)
    >{-}2\eps\,.
  \end{equation}
  To conclude, we add up \eqref{eq:proof-str-ic-1} (for $R=R_i$, $\B_R=B_i$) and
  \eqref{eq:proof-str-ic-3}, send $i\to\infty$, and finally exploit the
  arbitrariness of $\eps$. Then we arrive at \eqref{eq:lsc-Rn} in the
  case $\mu_+\equiv0$.
\end{proof}

\section{Existence with obstacles or volume-constraints}
  \label{sec:exist}

In this section we apply the preceding semicontinuity results on full $\R^n$ in
proving the existence of minimizers in obstacle problems or volume-constrained
problems for the functional $\Pmu$ introduced in \eqref{eq:P}.

In fact, for obstacle problems with \ae{} obstacle constraint, the existence
proof is mostly straightforward and leads to the following statement.

\begin{thm}[existence in obstacle problems]\label{thm:ex-obst}
  For sets $I,O\in\M(\R^n)$, $n\ge2$, consider the admissible class
  \[
    \mathscr{G}_{I,O}
    \coleq\{E\in\BVs(\R^n)\,:\,I\subset E\subset O\text{ up to negligible sets}\}\,.
  \]
  If there exists some $A_0\in\mathscr{G}_{I,O}$ at all and if, for non-negative
  Radon measures $\mu_+$ and $\mu_-$ on $\R^n$, which both satisfy the
  small-volume IC in $\R^n$ with constant\/ $1$, \ldots
  \begin{enumerate}[{\rm(a)}]
  \item either, $\mu_-(O^+)<\infty$ holds,\label{item:ex-obst-a}
  \item or, for some $R_0\in{(0,\infty)}$ and some $\gamma\in{(0,1]}$, the
    measure $\mu_-$ also satisfies the strong IC\footnote{To be fully consistent
    with Definition \ref{def:ICs}, which was given on open sets, we should speak
    of the IC in ${\big(\overline{\B_{R_0}}\big)}^\c$ here. However, since $R_0$ can be
    increased, it  does not make a difference if we work with
    $\overline A\subset{\big(\overline{\B_{R_0}}\big)}^\c$ or rather simply
    $A\subset{(\B_{R_0})}^\c$ instead. Thus, the slight inconsistency of writing
    ``in $(\B_{R_0})^\c$'' here and in the following seems justifiable.}
    in ${(\B_{R_0})}^\c$ with constant\/ $1{-}\gamma$,\label{item:ex-obst-b}
  \end{enumerate}
  then there exists the minimum of the obstacle problem
  \begin{equation}\label{eq:obst-min}
    \min\{\Pmu[E]\,:\,E\in\mathscr{G}_{I,O}\}\,,
  \end{equation}
  with a minimum value in ${({-}\mu_-(O^+),\infty)}$ in case
  \eqref{item:ex-obst-a} and in
  ${({-}(1{-}\gamma)\P(\B_{R_0}){-}\mu_-(\overline{\B_{R_0}}),\infty)}$ in case
  \eqref{item:ex-obst-b}.
\end{thm}

\medskip

\newlength\parind\setlength\parind{\parindent}
\noindent\begin{minipage}[t]{8cm}
  \hspace{\parind}As a basic case, which illustrates the applicability of
  Theorem \ref{thm:ex-obst}, we consider measurable obstacles
  $I\Subset O\subset\R^n$ and $(n{-}1)$-dimensional measures
  $\mu_\pm=\theta_\pm\H^{n-1}\ecke(\R^{n-1}{\times}\{0\})$ with
  $\theta_+,\theta_-\in{[0,\infty)}$. Then indeed, the setting
  \eqref{item:ex-obst-a} applies for $\mu(O^+)<\infty$ (e.\@g.\@ if $O$ is
  bounded) and $\theta_+\le2$, $\theta_-\le2$, while the setting
  \eqref{item:ex-obst-b} covers even fully arbitrary $O$ up to $O=\R^n$ in case
  $\theta_+\le2$, $\theta_-<2$ (but now with $\theta_-=2$ excluded).
  Specifically for $n=2$, $O=\R^2$, $\theta_+=0$, one may also identify
  minimizers $A$ in the obstacle problem \eqref{eq:obst-min} in a geometrically
  intuitive way, illustrated in Figure \ref{fig:obst-min}, as a certain convex
  hull of $I$ with an additional $\theta_-$-dependent constraint on the angles
  at the intersection of $\partial A$ and $\spt\mu_-=\R^{n-1}{\times}\{0\}$.
  However, we leave more detailed considerations on such specific geometric
  cases for study elsewhere.
\end{minipage}
\hfill
\begin{minipage}[t]{7.6cm}\centering
\vspace{-.3cm}
\begin{figure}[H]\centering
  \includegraphics{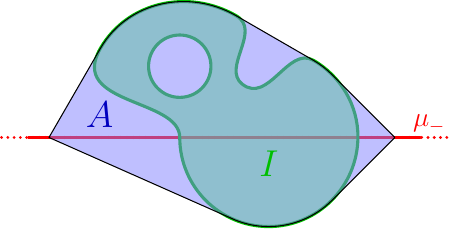}
  \vspace{-.3cm}
  \caption{A minimizer $A$ in the obstacle problem \eqref{eq:obst-min} for
    $n=2$, some smooth $I\Subset\R^2$, $O=\R^2$, $\mu_+\equiv0$, and
    $\mu_-=\sqrt2\H^1\ecke(\R{\times}\{0\})$.\label{fig:obst-min}}
\end{figure}
\end{minipage}

\vspace{.3ex}

Here, we additionally remark that if we have $I=\emptyset$ and $\mu_-$ satisfies
the strong IC even in full $\R^n$ with constant $1{-}\gamma$, then in view of
$\Pmu[E]\ge\gamma\P(E)$ for all $E\in\BVs(\R^n)$ the situation of the theorem
trivializes insofar that the unique minimizer up to negligible sets in
\eqref{eq:obst-min} is $\emptyset$. However, our settings \eqref{item:ex-obst-a}
and \eqref{item:ex-obst-b} allow for situations which do not trivialize to the
same extent \emph{even in the absence of the inner obstacle}. To demonstrate
this, we consider $I\coleq\emptyset$, an arbitrary $O\in\M(\R^n)$, any
non-empty, bounded, open, convex $K\in\mathscr{G}_{I,O}$, $\mu_+\colequiv0$, and
the finite measure $\mu_-\coleq\theta\H^{n-1}\ecke\partial K$ with
$\theta\in{[0,\infty)}$. Then it can be checked that the obstacle problem in
\eqref{eq:obst-min} has the unique minimizer $\emptyset$ in case $\theta<1$, has
both $\emptyset$ and $K$ as minimizers in case $\theta=1$, and has the unique
minimizer $K$ in case $\theta>1$. Here, the measure
$\mu_-=\theta\H^{n-1}\ecke\partial K$ trivially satisfies the strong IC in
$(\B_{R_0})^\c$ for $R_0$ large enough and by the later Theorem
\ref{thm:2P-admis} satisfies the small-volume IC in $\R^n$ with constant
$\theta/2$, while by the later Proposition \ref{prop:strong-ic-for-pseudoconvex}
it satisfies the strong IC in full $\R^n$ only with constant $\theta$. All in
all, this means that the non-trivial cases with $\theta\in{[1,2]}$ are indeed
included in the regimes of \eqref{item:ex-obst-a} and \eqref{item:ex-obst-b}
above, but would not be covered by a statement with the strong IC on full
$\R^n$.

\begin{proof}[Proof of Theorem \ref{thm:ex-obst}]
  We first record that $A_0\in\BVs(\R^n)$ implies
  $\mu_+(A_0^1)\le\mu_+(A_0^+)<\infty$ by Lemma \ref{lem:finite}, and thus the
  minimum value in \eqref{eq:obst-min} is bounded from above by
  $\P(A_0){+}\mu_+(A_0^1){-}\mu_-(A_0^+)<\infty$.

  Now we treat the situation \eqref{item:ex-obst-a}. In view of
  \[
    \Pmu[E]=\P(E){+}\mu_+(E^1){-}\mu_-(E^+)\ge\P(E)-\mu_-(O^+)
  \]
  for all $E\in\mathscr{G}_{I,O}$, every minimizing sequence
  $(A_k)_{k\in\N}$ for $\Pmu$ in $\mathscr{G}_{I,O}$ satisfies
  $\limsup_{k\to\infty}\P(A_k)<\infty$. By the standard compactness and
  semicontinuity results from Lemmas \ref{lem:cpct-per} and \ref{lem:lsc-per}, a
  subsequence of $(A_k)_{k\in\N}$ converges \emph{locally} in measure on $\R^n$
  to some $A_\infty\in\M(\R^n)$ with $\P(A_\infty)<\infty$ and $I\subset A_\infty\subset O$ up to
  negligible sets. Taking into account $|A_k|<\infty$, the isoperimetric
  estimate of Theorem \ref{thm:isoperi} ensures
  $\limsup_{k\to\infty}|A_k|<\infty$, and by a basic semicontinuity property we
  infer $|A_\infty|<\infty$ and thus $A_\infty\in\mathscr{G}_{I,O}$. Then, Theorem
  \ref{thm:lsc-Rn}\eqref{item:lsc-Rn-a}, applied with the \emph{finite} Radon
  measure $\mu_-\ecke O^+$ instead of $\mu_-$, ensures that the limit $A_\infty$
  is a minimizer.

  Next we turn to the situation \eqref{item:ex-obst-b}. Since the strong IC for
  $\mu_-$ in $(\B_{R_0})^\c$ yields
  \[\begin{aligned}
    \Pmu[E]&\ge\P(E)-\mu_-((E\setminus\B_{R_0})^+)-\mu_-(\overline{\B_{R_0}})\\
    &\ge\P(E)-(1{-}\gamma)\P(E\setminus\B_{R_0})-\mu_-(\overline{\B_{R_0}})\\
    &\ge\gamma\P(E)-(1{-}\gamma)\P(\B_{R_0})-\mu_-(\overline{\B_{R_0}})
  \end{aligned}\]
  for all $E\in\mathscr{G}_{I,O}$, again every minimizing sequence $(A_k)_{k\in\N}$
  for $\Pmu$ in $\mathscr{G}_{I,O}$ satisfies $\limsup_{k\to\infty}\P(A_k)<\infty$.
  At this stage the arguments given for the the situation \eqref{item:ex-obst-a}
  still yield that a subsequence of $(A_k)_{k\in\N}$ converges \emph{locally} in
  measure on $\R^n$ to some $A_\infty\in\mathscr{G}_{I,O}$. Finally, by Theorem
  \ref{thm:lsc-Rn}\eqref{item:lsc-Rn-b} we conclude that the limit $A_\infty$ is
  a minimizer.
\end{proof}

To conclude the discussion of obstacle problems we remark that a more general
point of view with thin obstacles and $\H^{n-1}$-\ae{} obstacle constraints
(compare \cite{DeGColPic72,Giusti72,DeAcutis79,CarDalLeaPas88,SchSch18}, for instance) might
be naturally connected to our setting, but we leave such issues for study at
another point.

\medskip

We now turn to volume-constrained minimization problems for $\Pmu$, where the
special case $\mu\equiv0$ corresponds to the classical isoperimetric problem.
We provide an existence statement for minimizers of $\Pmu$ at least in case that
$\mu_+$ vanishes and $\mu_-$ is finite.

\begin{thm}[existence in prescribed-volume problems]\label{thm:ex-prscr-vol}
  Consider a non-negative Radon measure $\mu$ on $\R^n$ with $\mu(\R^n)<\infty$
  and a constant $\r\in{(0,\infty)}$. If $\mu$ satisfies the small-volume
  IC in $\R^n$ with constant $1$, then there exists the
  minimum of the prescribed-volume problem
  \[
    \min\{\P(A){-}\mu(A^+)\,:\,
    A\in\BVs(\R^n)\,,\,|A|=\alpha_n\r^n\}
  \]
  with a minimum value in
  ${\big({-}\mu(\R^n),n\alpha_n\r^{n-1}\big]}$.
\end{thm}

Here, the bounds for the minimum value leave room for improvement. For instance,
estimating via the isoperimetric inequality we find that the minimum value is in
fact in ${\big[n\alpha_n\r^{n-1}{-}\mu(\R^n),n\alpha_n\r^{n-1}\big]}$. In
addition, let us point out that if $\mu$ has bounded support and $\r$ is large
enough such that $\spt\mu\subset\overline{\B_\r(x)}$ for some $x\in\R^n$, then
$\B_\r(x)$ is a minimizer and the theorem holds trivially. In the general case,
however, the result is non-trivial and the proof is somewhat involved, since
(subsequences of) minimizing sequences may converge only locally, but not
globally in measure, and in view of a ``volume drop'' at infinity the limit then
violates the volume constraint and is not admissible as a minimizer. Our
strategy to circumvent this phenomenon is not really new and is vaguely inspired
by considerations of \cite{GonMasTam83,RitRos04}, for instance. The basic idea is
to suitably shift volume into a fixed ball, which in our case with
$\mu_+\equiv0$ and $\mu_-(\R^n)<\infty$ can be implemented with suitable control
on the values of $\Pmu$ along the sequence. Indeed, in this way we are able to
construct refined minimizing sequences with global convergence in measure and
an admissible limit, which turns out to be a minimizer.

\begin{proof}
  We start with the main case $n\ge2$ and record that $\B_\r$ is admissible with
  $\P(\B_\r){-}\mu(\B_\r^+)\le\P(\B_\r)=n\alpha_n\r^{n-1}<\infty$. Taking into
  account
  \[
    \P(A)-\mu(A^+)\ge\P(A){-}\mu(\R^n)
  \]
  for all admissible $A$, it is thus clear that every minimizing sequence
  $(A_k)_{k\in\N}$ satisfies $\limsup_{k\to\infty}\P(A_k)<\infty$. Using
  compactness and semicontinuity and possibly passing to a subsequence, we get
  that $(A_k)_{k\in\N}$ converges \emph{locally} in measure on $\R^n$ to some
  $A_\infty\in\BVs(\R^n)$ with $|A_\infty|\le\alpha_n\r^n$.

  We next choose good cut-off radii. By Fatou's lemma, the coarea formula,
  and the volume constraint we get
  \[
    \int_0^\infty\liminf_{k\to\infty}\H^{n-1}(A_k^+\cap\partial\B_R)\,\d R
    \le\lim_{k\to\infty}\int_0^\infty\H^{n-1}(A_k^+\cap\partial\B_R)\,\d R
    =\lim_{k\to\infty}|A_k|=\alpha_n\r^n<\infty\,.
  \]
  Thus, there is a sequence of radii $R_i\in{(2\r,\infty)}$ with
  $\lim_{i\to\infty}R_i=\infty$ and
  $\liminf_{k\to\infty}\H^{n-1}(A_k^+{\cap}\partial\B_{R_i})<i^{-1}$ for all
  $i\in\N$. In particular, for a suitable subsequence $(A_{k_i})_{i\in\N}$ of
  $(A_k)_{k\in\N}$, by the local convergence in measure and the preceding choice
  of radii we can achieve
  \begin{equation}\label{eq:good-Ri-1}
    |(A_{k_i}\Delta A_\infty)\cap\B_{R_i}|<i^{-1}
  \end{equation}
  and
  \begin{equation}\label{eq:good-Ri-2}
    \H^{n-1}(A_{k_i}^+\cap\partial\B_{R_i})<i^{-1}
    \qq\text{for all }i\in\N\,.
  \end{equation}
  Next, since $s\mapsto|\B_s{\setminus}A_{k_i}|$ is continuous with
  $|\B_0{\setminus}A_{k_i}|=0$ (where we understand $\B_0\coleq\emptyset$ from
  here on) and $|\B_\r{\setminus}A_{k_i}|=|A_{k_i}{\setminus}\B_\r|
  \ge|A_{k_i}{\setminus}\B_{R_i}|$ (a consequence of $|A_{k_i}|=|\B_\r|$), we
  can also choose radii $r_i\in{(0,\r]}$ such that
  \[
    |\B_{r_i}{\setminus}A_{k_i}|=|A_{k_i}{\setminus}\B_{R_i}|
    \qq\text{for all }i\in\N\,,
  \]
  and we will now attempt to produce a modified minimizing sequence without
  loss of volume at infinity by removing $A_{k_i}{\setminus}\B_{R_i}$ from
  $A_{k_i}$ and at the same time adding $\B_{r_i}{\setminus}A_{k_i}$ for volume
  compensation. Indeed, this reasoning works out directly in case of
  \begin{equation}\label{eq:good-ri}
    \P(A_{k_i},\partial\B_{r_i})=0
    \qq\text{for all }i\in\N\,,
  \end{equation}
  but unfortunately \eqref{eq:good-ri} cannot be ensured in general.
  Nonetheless, in the sequel we first complete the proof under the simplifying
  assumption \eqref{eq:good-ri}, and we postpone the discussion how to
  compensate for a failure of \eqref{eq:good-ri} to the end of our
  reasoning. For now, we use the announced competitors
  \[
    E_i\coleq(A_{k_i}\cap\B_{R_i})\cup\B_{r_i}
    =(A_{k_i}\cap\B_{R_i})\dcup(\B_{r_i}\setminus A_{k_i})\,,
  \]
  which in view of $|E_i|=|A_{k_i}{\cap}\B_{R_i}|{+}|\B_{r_i}{\setminus}A_{k_i}|
  =|A_{k_i}{\cap}\B_{R_i}|{+}|A_{k_i}{\setminus}\B_{R_i}|=|A_{k_i}|$ satisfy the
  volume constraint. In order to estimate the perimeter of $E_i$, we first
  observe
  \[
    \P(E_i)
    \le\H^{n-1}((\partial\B_{r_i})\setminus A_{k_i}^1)
    +\P(A_{k_i},\B_{R_i}\setminus\overline{\B_{r_i}})
    +\H^{n-1}(A_{k_i}^+\cap\partial\B_{R_i})
  \]
  and then continue by estimating the first term on the right-hand side. We rewrite
  \[
    \H^{n-1}((\partial\B_{r_i}){\setminus}A_{k_i}^1)
    =\P(\B_{r_i}){-}\H^{n-1}(A_{k_i}^1{\cap}\partial\B_{r_i})
  \]
  and then on the basis of $|\B_{r_i}|
  =|\B_{r_i}{\cap}A_{k_i}|{+}|\B_{r_i}{\setminus}A_{k_i}|
  =|\B_{r_i}{\cap}A_{k_i}|{+}|A_{k_i}{\setminus}\B_{R_i}|$ exploit the
  isoperimetric inequality \eqref{eq:sharp-isoperi} to deduce
  \[
    \P(\B_{r_i})\le\P(\B_{r_i}{\cap}A_{k_i})+\P(A_{k_i}{\setminus}\B_{R_i})\,.
  \]
  Further we can control
  \[
    \P(\B_{r_i}\cap A_{k_i})
    \le\P(A_{k_i},\B_{r_i})
    {+}\H^{n-1}(A_{k_i}^+\cap\partial\B_{r_i})\,,\qq
    \P(A_{k_i}\setminus\B_{R_i})
    \le\P(A_{k_i},\R^n{\setminus}\overline{\B_{R_i}})
    {+}\H^{n-1}(A_{k_i}^+\cap\partial\B_{R_i})\,.
  \]
  Putting together the estimates and collecting the three terms
  $\P(A_{k_i},\B_{r_i})$, 
  $\P(A_{k_i},\B_{R_i}{\setminus}\overline{\B_{r_i}})$,
  $\P(A_{k_i},\R^n{\setminus}\overline{\B_{R_i}})$ simply in $\P(A_{k_i})$, we
  arrive at
  \[
    \P(E_i)
    \le\P(A_{k_i})
    +\H^{n-1}((A_{k_i}^+\setminus A_{k_i}^1)\cap\partial\B_{r_i})
    +2\H^{n-1}(A_{k_i}^+\cap\partial\B_{R_i})\,.
  \]
  Here, the middle term on the right-hand side can be rewritten as
  $\P(A_{k_i},\partial\B_{r_i})$ and vanishes under the simplifying assumption
  \eqref{eq:good-ri}, while the last term on the right-hand side is controlled
  by $2i^{-1}$ through \eqref{eq:good-Ri-2}. Also bringing in that we have
  $\mu(E_i^+)\ge\mu(A_{k_i}^+{\cap}\B_{R_i})=\mu(A_{k_i}^+){-}\mu(\R^n{\setminus}\B_{R_i})$,
  we finally arrive at
  \[
    \P(E_i)-\mu(E_i^+)
    \le\P(A_{k_i})-\mu(A_{k_i}^+)+2i^{-1}+\mu((\B_{R_i})^\c)\,.
  \]
  Then, crucially exploiting $\lim_{i\to\infty}R_i=\infty$ and
  $\mu(\R^n)<\infty$, we have $\lim_{i\to\infty}\mu(\R^n{\setminus}\B_{R_i})=0$
  and can conclude that with $(A_k)_{k\in\N}$ also $(E_i)_{i\in\N}$ is a
  minimizing sequence in the volume-constrained problem.

  Now, in view of $r_i\le\r$ for all $i\in\N$, by passing to subsequences we can
  assume that $r\coleq\lim_{i\to\infty}r_i\in{[0,\r]}$ exists, and we finally
  proceed to establish that $A_\infty{\cup}\B_r$ is a minimizer in the
  volume-constrained problem. To this end we record that
  $E_i=(A_{k_i}{\cap}\B_{R_i}){\cup}\B_{r_i}$ converge \emph{locally} in
  measure on $\R^n$ to $A_\infty{\cup}\B_r$, since in this local sense we have the
  convergences $A_{k_i}\to A_\infty$, $\B_{R_i}\to\R^n$, $\B_{r_i}\to\B_r$. In order to
  show admissibility of $A_\infty{\cup}\B_r$, for arbitrary $i\in\N$, we split
  \[
    \alpha_n\r^n=|A_{k_i}|=|A_{k_i}\cap\B_{R_i}|+|A_{k_i}\setminus\B_{R_i}|\,,
  \]
  and via \eqref{eq:good-Ri-1}, the choice of $r_i$, and the \emph{local}
  convergence in measure $A_k\to A_\infty$ deduce for the right-hand volumes the
  convergences
  \[
    \lim_{i\to\infty}|A_{k_i}\cap\B_{R_i}|=\lim_{i\to\infty}|A_\infty\cap\B_{R_i}|=|A_\infty|
    \qq\text{and}\qq
    \lim_{i\to\infty}|A_{k_i}\setminus\B_{R_i}|
    =\lim_{i\to\infty}|\B_{r_i}\setminus A_{k_i}|
    =|\B_r\setminus A_\infty|\,.
  \]
  This implies that $A_\infty{\cup}\B_r$ fulfills the volume constraint
  $\alpha_n\r^n=|A_\infty|{+}|\B_r{\setminus}A_\infty|=|A_\infty{\cup}\B_r|$. Thus, we are in the
  position to finally use the semicontinuity in Theorem
  \ref{thm:lsc-Rn}\footnote{More precisely, one way of reasoning at this point
  is to use the semicontinuity assertion from Theorem
  \ref{thm:lsc-Rn}\eqref{item:lsc-Rn-a}, which draws on the finiteness of $\mu$
  and needs local convergence only. Another way is to rely only on the case
  covered in each of Theorem \ref{thm:lsc-global-intro}, Theorem
  \ref{thm:lsc-Rn}\eqref{item:lsc-Rn-c}, and Proposition \ref{prop:lsc-global}
  on the basis of the observation that the coincidence of volumes
  $|E_i|=\alpha_n\r^n=|A_\infty{\cup}\B_r|$ improves the local convergence to global
  convergence required in these statements.} along the
  minimizing sequence $E_i$ with limit $A_\infty{\cup}\B_r$ and deduce that
  $A_\infty{\cup}\B_r$ is a minimizer in the volume-constrained problem.

  It remains to provide an argument in case \eqref{eq:good-ri} fails. In this
  situation, since $\P(A_{k_i},\partial\B_q)=0$ holds for all but countably many
  $q\in{(0,\infty)}$ (and trivially for $q=0$), we can pass to
  ever-so-slightly-decreased good radii $q_i\in{[0,r_i]}$. However, in view of
  the volume constraint we cannot directly use
  $(A_{k_i}{\cap}\B_{R_i}){\cup}\B_{q_i}$ as competitors but rather need to
  compensate once more for the slight loss of volume. In fact, fixing
  arbitrary points $x_i\in(\B_{R_i}{\setminus}\overline{\B_{r_i}})$
  with $|\B_\delta(x_i){\setminus}A_{k_i}|>0$ for all $\delta>0$
  (such points exist, since $|A_{k_i}|=\alpha_n\r^n\le|\B_{2\r}{\setminus}\B_\r|
  <|\B_{R_i}{\setminus}\overline{\B_{r_i}}|$), for every $q_i\in{[0,r_i]}$, we
  find by continuity some $\delta_i\in{[0,\infty)}$ with
  $|\B_{q_i}{\setminus}A_{k_i}|{+}|\B_{\delta_i}(x_i){\setminus}A_{k_i}|
  =|\B_{r_i}{\setminus}A_{k_i}|$. Moreover, if we take $q_i$ arbitrarily close
  to $r_i$, then in view of $|\B_\delta(x_i){\setminus}A_{k_i}|>0$ for all
  $\delta>0$ this results in $\delta_i$
  coming arbitrarily close to $0$. We can thus choose $q_i\in{[0,r_i]}$ with
  $\P(A_k,\partial\B_{q_i})=0$ close enough to $r_i$ to ensure for a
  corresponding $\delta_i\in{[0,\infty)}$ that $\delta_i<i^{-1}$ and
  $\B_{\delta_i}(x_i)\Subset\B_{R_i}{\setminus}\overline{\B_{r_i}}$. Then it can
  be checked that
  \[
    \widetilde E_i\coleq(A_{k_i}\cap\B_{R_i})\cup\B_{q_i}\cup\B_{\delta_i}(x_i)
  \]
  satisfies the volume constraint. Moreover, we can estimate
  $\P(\widetilde E_i)$ essentially in the same way as $\P(E_i)$, just with an
  extra term controlled by
  $\P(\B_{\delta_i}(x_i))=n\alpha_n\delta_i^{n-1}<n\alpha_ni^{1-n}$. In this way
  we deduce
  \[
    \P(\widetilde E_i){-}\mu(\widetilde E_i^+)
    \le\P(A_{k_i})-\mu(A_{k_i}^+)+2i^{-1}+n\alpha_ni^{1-n}+\mu((\B_{R_i})^\c)\,,
  \]
  which is still sufficient to conclude that the modified sequence
  $(\widetilde E_i)_{i\in\N}$ is a minimizing sequence for the
  volume-constrained problem. From this point onwards, taking into account
  $\lim_{i\to\infty}|\B_{\delta_i}(x_i)|=0$ the verification of the volume
  constraint for $A_\infty{\cup}\B_r$ with $r=\lim_{i\to\infty}r_i=\lim_{i\to\infty}q_i$
  and the remainder of the reasoning work almost exactly as described before.

  Finally, in the case $n=1$ a similar reasoning with major simplifications
  applies, where now each $A_k$ with volume constraint $|A_k|=2\r$ can be
  represented as a union of finitely many bounded intervals and in particular
  satisfies $A_k^+=\overline{A_k}$ and $A_k^1=\inn(A_k)$. Indeed, the beginning
  of the reasoning up to the choice of the radii $R_i$ stays essentially
  unchanged with \eqref{eq:good-Ri-2} now simplifying to
  ${\pm}R_i\notin\overline{A_{k_i}}$. However, the construction of competitors
  with compensated volume vastly simplifies with the need for \eqref{eq:good-ri}
  completely dropping out. In fact, we claim that by choice of an interval
  $I_i\subset\B_\r\subset\B_{R_i}$ (where the balls are also intervals, but for
  brevity we keep the $\B$-notation) one can ensure that
  \[
    E_i\coleq(A_{k_i}\cap\B_{R_i})\cup I_i
  \]
  satisfies the constraint $|E_i|=2\r$ and the simple bound
  $\P(E_i)\le\P(A_{k_i})$. To prove this claim, first consider the case
  $|A_{k_i}\cap\B_\r|>0$. Then a continuity argument gives an interval
  $I_i\subset\B_\r$ with $|I_i\cap A_{k_i}|>0$ and
  $|I_i\setminus A_{k_i}|=|A_{k_i}\setminus\B_{R_i}|$, and this suffices to
  conclude $|E_i|=|A_{k_i}|=2\r$ and
  $\P(E_i)\le\P(A_{k_i}\cap\B_{R_i})\le\P(A_{k_i})$ (where the former estimate
  holds, since $I_i$ intersects at least one interval of
  $A_{k_i}\cap\B_{R_i}$). In case $|A_{k_i}\cap\B_\r|=0$ the simple choice
  $I_i\coleq\B_{r_i}$ with
  $r_i\coleq\frac12|A_{k_i}\setminus\B_{R_i}|\in{[0,\r]}$ gives
  $|E_i|=|A_{k_i}|=2\r$ and $\P(I_i)\le\P(A_{k_i}\setminus\B_{R_i})$ (as either
  $\P(I_i)=0=\P(A_{k_i}\setminus\B_{R_i})$ or
  $\P(I_i)=2\le\P(A_{k_i}\setminus\B_{R_i})$). Then in view of
  ${\pm}R_i\notin\overline{A_{k_i}}$ one still gets
  $\P(E_i)\le\P(A_{k_i}\cap\B_{R_i})+\P(A_{k_i}\setminus\B_{R_i})=\P(A_{k_i})$.
  With these properties of $E_i$ and the unchanged estimate for $\mu(E_i^+)$,
  one directly infers that $(E_i)_{i\in\N}$ is a minimizing sequence in the
  volume-constrained problem with (after passage to a subsequence) limit
  $A_\infty\cup I$ for some interval $I\subset\B_\r$. As in the case $n\ge2$ one
  then concludes that the convergence $E_i\to A_\infty\cup I$ looses no volume at
  infinity and that $A_\infty\cup I$ is a minimizer.
\end{proof}

\section{Lower semicontinuity and existence for Dirichlet problems}\label{sec:Dir}

In this section we adapt the semicontinuity results of Section
\ref{sec:Rn} to a setting with a (generalized) Dirichlet condition on
the boundary of an open set $\Omega\subset\R^n$. To this end we prescribe the
Dirichlet datum by means of a set $A_0\in\M(\R^n)$ and consider the class
\begin{equation}\label{eq:Dir-class}\begin{aligned}
  \mathscr{D}_{A_0}(\Omega)
  &\coleq\{E\in\M(\R^n)\,:\,\P(E,\overline\Omega)<\infty\,,\,E\setminus\Omega=A_0\setminus\Omega\}\\
  &\hspace{.45ex}=\hspace{.2ex}\{E\in\M(\R^n)\,:\,\P(E,\overline\Omega)<\infty\,,\,E\Delta A_0\subset\Omega\}\,,
\end{aligned}\end{equation}
in which sets of finite perimeter are extended from $\Omega$ to (a neighborhood
of) $\overline\Omega$ by coincidence with the given $A_0$ outside $\Omega$. In
addition, we prescribe once more measures $\mu_+$ and $\mu_-$, which in
principle live on $\overline\Omega$, but for which we can indeed express
finiteness on all bounded sets and suitable ICs in a
convenient way by considering them as a Radon measure on all of $\R^n$ such that
$\mu_\pm\ecke(\overline\Omega)^\c\equiv0$. Given the data $A_0$ and $\mu_\pm$ we
then aim at minimizing among all $E\in\mathscr{D}_{A_0}(\Omega)$ the adaptation
of the previously considered functional
\begin{equation}\label{eq:P-Dir}
  \Pmu[E;\overline\Omega]\coleq\P(E,\overline\Omega)+\mu_+(E^1)-\mu_-(E^+)\,,
\end{equation}
which is defined for $E\in\M(\R^n)$ if at least one of
$\P(E,\overline\Omega){+}\mu_+(E^1)$ and $\mu_-(E^+)$ is finite and specifically
for $E\in\mathscr{D}_{A_0}(\Omega)$ with $\min\{\mu_+(E^1),\mu_-(E^+)\}<\infty$.
Here --- as customary in the $\BVs$ setting and essentially required by the lack
of weak closedness of traces --- it is tolerated for $E\in\mathscr{D}_{A_0}(\Omega)$
that $\partial E$ deviates from $\partial A_0$ at $\partial\Omega$, but such
deviations are accounted for by taking the perimeter on $\overline\Omega$ and
thus including $\P(E,\partial\Omega)$ in the functional.

With view towards non-parametric Dirichlet problems we will include --- to the
extent possible in a general parametric theory --- unbounded domains $\Omega$
(e.\@g.\@ cylinders $\Omega=D\times\R$ over open $D\subset\R^{n-1}$) and
infinite measures $\mu_\pm$ (e.\@g.\@  product measures
$\mu_\pm=\lambda_\pm\otimes\mathcal{L}^1$ with finite Radon measures
$\lambda_\pm=\lambda_\pm\ecke\overline D$). Thus, the application of our results
in the case $\Omega=D\times\R$, $\mu_\pm=\lambda_\pm\otimes\mathcal{L}^1$ is
possible, but nonetheless does not directly yield a satisfactory non-parametric
theory, since in this case the $\mu$-terms in \eqref{eq:P-Dir} are usually
infinite on subgraphs of functions and thus do not detect the finer behavior
of such non-parametric competitors. In this article, we do \emph{not}
elaborate on this technical point, but indeed we presume that it can be overcome
by first looking at one-sided cases with $\Omega=D\times{(z,\infty)}$,
$\mu_\pm=\lambda_\pm\otimes(\mathcal{L}^1\ecke{(z,\infty)})$ with $z\in\R$
(which are fully accessible by our means), then normalizing the $\mu$-terms
relative to a zero level or another reference configuration, and finally sending
$z\to{-}\infty$. However, all further details of such a procedure are deferred
for treatment elsewhere.

We now come back to the parametric cases under consideration here and provide our
results in form of a semicontinuity theorem and an existence theorem, which both
apply for the functional in \eqref{eq:P-Dir} inside Dirichlet classes of type
\eqref{eq:Dir-class}.

\begin{thm}[lower semicontinuity in a Dirichlet class]\label{thm:lsc-Dir}
  Consider an open set\/ $\Omega$ in $\R^n$, a set $A_\infty\in\M(\R^n)$, a
  sequence $(A_k)_{k\in\N}$ in $\M(\R^n)$, and assume that non-negative Radon
  measures $\mu_+$ and $\mu_-$ on $\R^n$ with
  $\mu_\pm\ecke(\overline\Omega)^\c\equiv0$ both satisfies the small-volume IC
  in $\R^n$ with constant\/ $1$. Furthermore, assume that \emph{one of the
  following sets of additional assumptions} is valid\textup{:}
  \begin{enumerate}[{\rm(a)}]
  \item The measure $\mu_-$ is \emph{finite}, and $A_k$ converge to $A_\infty$
    \emph{locally} in measure on $\R^n$ with
    $A_k\setminus\Omega=A_\infty\setminus\Omega$ for all
    $k\in\N$.\label{item:lsc-Dir-a}
  \item The measure $\mu_-$ additionally satisfies the \emph{almost-strong IC
    with constant $1$ near $\infty$} in the sense that, for every $\eps>0$,
    there exists some $R_\eps\in{(0,\infty)}$ with
    \eqref{eq:strong-ic-at-infty}, and $A_k$ converge to $A_\infty$ \emph{locally} in
    measure on $\R^n$ with
    $|A_k\Delta A_\infty|{+}\P(A_k,\overline\Omega){+}\P(A_\infty,\overline\Omega)<\infty$,
    $A_k\setminus\Omega=A_\infty\setminus\Omega$, and\/ $\min\{\mu_+(A_k^1),\mu_-(A_k^+)\}<\infty$ for
    all $k\in\N$.\label{item:lsc-Dir-b}
  \item The sets $A_k$ converge to $A_\infty$ \emph{globally} in measure on $\R^n$ with
    $\P(A_k,\overline\Omega){+}\P(A_\infty,\overline\Omega)<\infty$,
    $A_k\setminus\Omega=A_\infty\setminus\Omega$, and\/
    $\min\{\mu_+(A_k^1),\mu_-(A_k^+)\}<\infty$ for all
    $k\in\N$.\label{item:lsc-Dir-c}
  \end{enumerate}
  Then we have $\min\{\mu_+(A_\infty^1),\mu_-(A_\infty^+)\}<\infty$ and
  \begin{equation}\label{eq:lsc-Dir}
    \liminf_{k\to\infty}\Pmu[A_k;\overline\Omega]\ge\Pmu[A_\infty;\overline\Omega]\,.
  \end{equation}
\end{thm}

Before approaching the proof of Theorem \ref{thm:lsc-Dir} we address some
interconnected technical points.

First we remark that the hypotheses
$\P(A_k,\overline\Omega){+}\P(A_\infty,\overline\Omega)<\infty$ and
$A_k\setminus\Omega=A_\infty\setminus\Omega$ of the situations \eqref{item:lsc-Dir-b}
and \eqref{item:lsc-Dir-c} can be expressed alternatively as
$A_k,A_\infty\in\mathscr{D}_{A_0}(\Omega)$ for some $A_0\in\M(\R^n)$ or --- by
considering the limit $A_\infty$ itself as the boundary datum --- also as
$A_k,A_\infty\in\mathscr{D}_{A_\infty}(\Omega)$. Moreover, introducing, for open
$\Omega\subset\R^n$ and $A_0\in\M(\R^n)$, the subclass
\[
  \mathscr{F}_{A_0}(\Omega)
  \coleq\{E\in\M(\R^n)\,:\,|E\Delta A_0|{+}\P(E,\overline\Omega)<\infty\,,\,E\setminus\Omega=A_0\setminus\Omega\}
\]
of $\mathscr{D}_{A_0}(\Omega)$, we may include the additional requirement
$|A_k\Delta A_\infty|<\infty$ by writing $A_k,A_\infty\in\mathscr{F}_{A_0}(\Omega)$ for
some $A_0\in\M(\R^n)$ or $A_k,A_\infty\in\mathscr{F}_{A_\infty}(\Omega)$. If there exists some
$E_0\in\mathscr{F}_{A_0}(\Omega)$ at all (e.\@g.\@ if
$\P(A_0,\overline\Omega)<\infty$), we can also rewrite\footnote{Indeed, the
alternative characterization of $\mathscr{F}_{A_0}(\Omega)$ results from the
following elementary observations (for $\Omega,A_0,E_0$ as above). For
$E\in\M(\R^n)$, we have
$E\setminus\Omega=A_0\setminus\Omega\iff E\Delta E_0\subset\Omega$ and also
$|E\Delta A_0|<\infty\iff|E\Delta E_0|<\infty$. Moreover, for $E\in\M(\R^n)$ with
$E\Delta E_0\subset\Omega$, in view of
$\P(E\Delta E_0)=\P(E\Delta E_0,\overline\Omega)$ we get
$\P(E\Delta E_0)<\infty\iff\P(E,\overline\Omega)<\infty$.}
\[
  \mathscr{F}_{A_0}(\Omega)=\{E\in\M(\R^n)\,:\,E\Delta E_0\in\BVs(\R^n)\,,\,E\Delta E_0\subset\Omega\}\,.
\]
Furthermore, we record the following generalization of Lemma \ref{lem:finite},
which is adapted for the class $\mathscr{F}_{A_0}(\Omega)$.

\begin{lem}\label{lem:finite-Dir}
  Consider an open set $\Omega\subset\R^n$ and a set $A_0\in\M(\R^n)$. If a
  non-negative Radon measure $\mu$ on $\R^n$ satisfies the small-volume
  IC in $\R^n$ with constant $C\in{[0,\infty)}$, then
  $\mu(E_0^1)<\infty$ for \emph{some} $E_0\in\mathscr{F}_{A_0}(\Omega)$
  implies in fact $\mu(E^1)<\infty$ for \emph{all} $E\in\mathscr{F}_{A_0}(\Omega)$,
  and similarly $\mu(E_0^+)<\infty$ for \emph{some} $E_0\in\mathscr{F}_{A_0}(\Omega)$
  implies $\mu(E^+)<\infty$ for \emph{all} $E\in\mathscr{F}_{A_0}(\Omega)$.
\end{lem}

\begin{proof}
  For $E,E_0\in\mathscr{F}_{A_0}(\Omega)$, we have already recorded
  $E\Delta E_0\in\BVs(\R^n)$, and then by Lemma \ref{lem:finite} we infer
  $\mu(E^1\Delta E_0^1)\le\mu((E\Delta E_0)^+)<\infty$
  and $\mu(E^+\Delta E_0^+)\le\mu((E\Delta E_0)^+)<\infty$. Therefore,
  $\mu(E_0^1)<\infty$ implies $\mu(E^1)<\infty$, and $\mu(E_0^+)<\infty$ implies
  $\mu(E^+)<\infty$.
\end{proof}

Next some more remarks on the requirement $|A_k\Delta A_\infty|<\infty$ are in order.

\begin{rem}[on the role of $|A_k\Delta A_\infty|<\infty$ in Theorem
    \ref{thm:lsc-Dir}]\label{rem:lsc-Dir}
  While most requirements in Theorem \ref{thm:lsc-Dir} are natural
  and\textup{/}or resemble features from Theorem \ref{thm:lsc-Rn}, we find it
  worth pointing out that the finite-volume requirement for $A_k\Delta A_\infty$
  of the setting \eqref{item:lsc-Dir-b} is automatically satisfied in many
  cases, but cannot be omitted in full generality. This is clarified by the
  following points, which apply for any open $\Omega\subset\R^n$ and
  $A_0\in\M(\R^n)$\textup{:}
  \begin{enumerate}[{\rm(i)}]
  \item In analogy with Theorem \ref{thm:lsc-Rn}, in the setting
    \eqref{item:lsc-Dir-a} the requirement $|A_k\Delta A_\infty|<\infty$ is simply not
    necessary. Moreover, in the setting \eqref{item:lsc-Dir-c} we do not require
    $|A_k\Delta A_\infty|<\infty$ explicitly, but have it implicitly \ka at least for
    $k\gg1$\kz{} through the global convergence assumed there.
  \item If we have $n\ge2$ and\/ $\Omega$ is not too close to full space in the sense of\/
    $\Cp_1((\Omega^1)^\c)=\infty$ \ka as it follows from
    $|\Omega^\c|=\infty$, for instance\kz, then, for
    $A,E\in\mathscr{D}_{A_0}(\Omega)$ we always have $|E\Delta A|<\infty$. Thus, in
    this case we have $\mathscr{F}_{A_0}(\Omega)=\mathscr{D}_{A_0}(\Omega)$
    whenever $\mathscr{F}_{A_0}(\Omega)\neq\emptyset$, and also in the setting
    \eqref{item:lsc-Dir-b} the condition $|A_k\Delta A_\infty|<\infty$ is automatically
    satisfied and need not be required explicitly.\label{item:rem-lsc-Dir-ii}

    \begin{proof}
      From $E\Delta A\subset(A\Delta A_0)\cup(E\Delta A_0)\subset\Omega$
      we get $(E\Delta A)^1\subset\Omega^1$ and
      $\P(E\Delta A)\le\P(E,\overline\Omega){+}\P(A,\overline\Omega)<\infty$.
      Then the isoperimetric estimate of Theorem \ref{thm:isoperi} yields
      $\min\{|E\Delta A|,|(E\Delta A)^\c|\}<\infty$. In case
      $|(E\Delta A)^\c|<\infty$, however, observing
      $(\Omega^1)^\c\subset((E\Delta A)^1)^\c=((E\Delta A)^\c)^+$
      together with $(E\Delta A)^\c\in\BVs(\R^n)$ we get
      $\Cp_1((\Omega^1)^\c)<\infty$ from Proposition \ref{prop:inf-P}. This
      leaves $|E\Delta A|<\infty$ as the sole possibility.
    \end{proof}
  \item If we have $n\ge2$ and\/ $\Omega$ is close enough to full space in the
    sense of\/ $\Cp_1((\overline\Omega)^\c)<\infty$ \ka as it follows from
    $(\overline\Omega)^\c\in\BVs(\R^n)$, for instance\kz, then
    from Proposition \ref{prop:inf-P} we get $(\overline\Omega)^\c\subset H^+$
    for some $H\in\BVs(\R^n)$, and for every $E\in\M(\R^n)$ with
    $\P(E,\overline\Omega)<\infty$ either $E$ or $E^\c$ is in
    $\BVs(\overline\Omega)$. Specifically, for
    $A,E\in\mathscr{D}_{A_0}(\Omega)$, the requirement $|E\Delta A|<\infty$
    then means that either $A,E\in\BVs(\overline\Omega)$ or
    $A^\c,E^\c\in\BVs(\overline\Omega)$ holds, and the hypotheses of
    the setting \eqref{item:lsc-Dir-b} can be reformulated
    correspondingly.\label{item:rem-lsc-Dir-iii}
    
    \begin{proof}[Proof that either $E$ or $E^\c$ is in $\BVs(\overline\Omega)$]
      By assumption we have $\P(E,U)<\infty$ for an open
      $U\supset\overline\Omega$, from which we infer $\P(E\cup H)<\infty$, since $\R^n$
      is covered by the open sets $U$ and $(\overline\Omega)^\c$ and since
      $E\cup H$ has finite perimeter in $U$ and even zero perimeter in
      $(\overline\Omega)^\c$. This enforces
      $\min\{|E\cup H|,|(E\cup H)^\c|\}<\infty$ once more by Theorem
      \ref{thm:isoperi}. In view of $|H|<\infty$ we deduce
      $\min\{|E|,|E^\c|\}<\infty$ and consequently either
      $E\in\BVs(\overline\Omega)$ or $E^\c\in\BVs(\overline\Omega)$.
    \end{proof}
  \item In case $\Cp_1((\overline\Omega)^\c)<\infty$, $\mu_-(\R^n)=\infty$ the
    explicit requirement $|A_k\Delta A_\infty|<\infty$ can\emph{not} be dropped from
    the setting \eqref{item:lsc-Dir-b}, since lower semicontinuity fails with
    $\Pmu[A_k;\overline\Omega]={-}\infty$ for $k\in\N$, but
    $\Pmu[A_\infty;\overline\Omega]=0$, for instance, if we use $H$ from
    point \eqref{item:rem-lsc-Dir-iii} and take
    $A_k\coleq(\B_k\cup H)^\c$ with $A_k^\c\in\BVs(\R^n)$,
    $A_k\setminus\Omega=\emptyset$ and $A_\infty\coleq\emptyset\in\BVs(\R^n)$.
  \item For each open $\Omega\subset\R^n$, $n\ge2$, in view of
    $\Omega^1\subset\overline\Omega$ at least one of the points
    \eqref{item:rem-lsc-Dir-ii} and \eqref{item:rem-lsc-Dir-iii} applies, and
    sometimes even both apply. For instance, the latter happens for dense
    open $\Omega\subset\R^n$ with $|\Omega^\c|=\infty$.
  \end{enumerate}
\end{rem}

Finally, we turn to the proof of the theorem.

\begin{proof}[Proof of Theorem \ref{thm:lsc-Dir}]
  The subsidiary claim $\min\{\mu_+(A_\infty^1),\mu_-(A_\infty^+)\}<\infty$ is trivially
  satisfied in the situation \eqref{item:lsc-Dir-a} with finite
  $\mu_-$. It is also satisfied in the situations \eqref{item:lsc-Dir-b} and
  \eqref{item:lsc-Dir-c}, since in these we have $A_k,A_\infty\in\mathscr{F}_{A_\infty}(\Omega)$
  (at least for $k\gg1$) and since we know from Lemma \ref{lem:finite-Dir} that
  $\mu_+(A_k^1)<\infty$ even for a single $A_k\in\mathscr{F}_{A_\infty}(\Omega)$ implies
  $\mu_+(A_\infty^1)<\infty$ and likewise $\mu_-(A_k^+)<\infty$ implies
  $\mu_-(A_\infty^+)<\infty$.
  
  To shorten notation, in the remainder of this proof we abbreviate
  \[
    \langle\mu_\pm\,;A\rangle\coleq\mu_+(A^1)-\mu_-(A^+)\,,
  \]
  and we record that, in all three situations, Lemma \ref{lem:loc-lsc} yields
  \[
    \liminf_{k\to\infty}\big[\P(A_k,\B_R)+\langle\mu_\pm\ecke\B_R\,;A_k\rangle\big]
    \ge\P(A_\infty,\B_R)+\langle\mu_\pm\ecke\B_R\,;A_\infty\rangle
    \qq\text{for all }R\in{(0,\infty)}\,.
  \]
  Moreover, whenever we additionally ensure $A_k,A_\infty\in\BVs_\loc(\R^n)$ for
  $k\gg1$, then in view of $A_k\setminus\Omega=A_\infty\setminus\Omega$ we may subtract
  $\P(A_k,\B_R\setminus\overline\Omega)=\P(A_\infty,\B_R\setminus\overline\Omega)<\infty$
  on both sides to arrive at
  \begin{equation}\label{eq:lsc-Dir-pre}
    \liminf_{k\to\infty}\big[\P(A_k,\overline\Omega\cap\B_R)+\langle\mu_\pm\ecke\B_R\,;A_k\rangle\big]
    \ge\P(A_\infty,\overline\Omega\cap\B_R)+\langle\mu_\pm\ecke\B_R\,;A_\infty\rangle
  \end{equation}
  Taking these preliminary observations as a starting point, we now deal with
  the three situations separately, where throughout we can and do assume that
  $\lim_{k\to\infty}\big[\P(A_k,\overline\Omega){+}\langle\mu_\pm\,;A_k\rangle\big]$
  exists and is finite.

  We first treat the situation \eqref{item:lsc-Dir-a}. Since in this case $\mu_-$
  is finite, we directly get
  $\limsup_{k\to\infty}\P(A_k,\overline\Omega)<\infty$, and then, using the
  lower semicontinuity of the perimeter and
  $A_k\setminus\Omega=A_\infty\setminus\Omega$, we infer $\P(A_k,U)+\P(A_\infty,U)<\infty$ for
  $k\gg1$ on a fixed open $U\supset\overline\Omega$. This finding and the
  assumption $\mu_\pm\ecke(\overline\Omega)^\c=0$ open the way to modify $A_k$ and
  $A_\infty$ away from $\overline\Omega$ and ensure that there is no loss of
  generality in assuming $A_k,A_\infty\in\BVs_\loc(\R^n)$ for $k\gg1$ and the validity
  of \eqref{eq:lsc-Dir-pre}. Trivially estimating on the left-hand side of
  \eqref{eq:lsc-Dir-pre}, we deduce, for all $R\in{(0,\infty)}$,
  \[
    \liminf_{k\to\infty}\Pmu[A_k;\overline\Omega]+\mu_-((\B_R)^\c)
    \ge\P(A_\infty,\overline\Omega\cap\B_R)+\langle\mu_\pm\ecke\B_R\,;A_\infty\rangle\,,
  \]
  and then, sending $R\to\infty$ and crucially exploiting the finiteness
  of $\mu_-$, we arrive at the claim \eqref{eq:lsc-Dir}.

  Next we turn to the situation \eqref{item:lsc-Dir-b}. From
  the assumptions $A_k,A_\infty\in\mathscr{D}_{A_\infty}(\Omega)$ we get
  $\P(A_k,U){+}\P(A_\infty,U)<\infty$ for all $k\in\N$ on a fixed open
  $U\supset\overline\Omega$. Again this means that we may modify $A_k$ and $A_\infty$
  away from $\overline\Omega$ and may assume the validity of
  \eqref{eq:lsc-Dir-pre}. For arbitrary $\eps>0$, relying on cut-off arguments
  as in the proofs of Proposition \ref{prop:lsc-global'} and Lemma
  \ref{lem:loc-lsc} we obtain radii $R_i\in{(R_\eps,\infty)}$ with
  $\lim_{i\to\infty}R_i=\infty$ and replace $(A_k)_{k\in\N}$ with one of its
  subsequences such that there hold $\mu_-(\partial\B_{R_i})=0$ and
  $\lim_{k\to\infty}\H^{n-1}((A_k\Delta A_\infty)^+\cap\partial\B_{R_i})=0$ for all
  $i\in\N$. We exploit $\mu_-(\partial\B_{R_i})=0$ and bring in the
  assumptions $A_k\Delta A_\infty\subset\Omega$, $|A_k\Delta A_\infty|<\infty$ and the
  assumed almost-strong IC near $\infty$ (applicable in view of $R_i>R_\eps$) in
  the decisive estimate
  \begin{align}
    \mu_-((A_k^+\Delta A_\infty^+)\setminus\B_{R_i})
    &\le\mu_-(((A_k\Delta A_\infty)\setminus\B_{R_i})^+)
      \le\P((A_k\Delta A_\infty)\setminus\B_{R_i})+\eps
      =\P((A_k\Delta A_\infty)\setminus\B_{R_i},\overline\Omega)+\eps
      \nonumber\\
    &\le\P(A_k,\overline\Omega\setminus\B_{R_i})+\P(A_\infty,\overline\Omega\setminus\B_{R_i})+\H^{n-1}((A_k\Delta A_\infty)^+\cap\partial\B_{R_i})+\eps\,.
      \label{eq:mu-at-infty}
  \end{align}
  Taking into account $\mu_-(\B_{R_i})<\infty$, the estimate
  \eqref{eq:mu-at-infty} yields in particular $\mu_-(A_k^+\Delta A_\infty^+)<\infty$ and
  thus leaves us with the alternative that either $\mu_-(A_k^+)=\mu_-(A_\infty^+)=\infty$
  holds for all $k\in\N$ or $\mu_-(A_k^+){+}\mu_-(A_\infty^+)<\infty$ holds for all
  $k\in\N$. In the case $\mu_-(A_k^+)=\mu_-(A_\infty^+)=\infty$, taking into
  account $\min\{\mu_+(A_k^1),\mu_-(A_k^+)\}<\infty$ and
  $\min\{\mu_+(A_\infty^1),\mu_-(A_\infty^+)\}<\infty$, we necessarily have
  $\mu_+(A_k^1){+}\mu_+(A_\infty^1)<\infty$ for all $k\in\N$, and \eqref{eq:lsc-Dir} is
  trivially satisfied with value ${-}\infty$ on both sides. Thus, from here on
  we deal with the case $\mu_-(A_k^+){+}\mu_-(A_\infty^+)<\infty$ only. We  rearrange
  the terms in \eqref{eq:mu-at-infty}, pass $k\to\infty$, and involve
  $\lim_{k\to\infty}\H^{n-1}((A_k\Delta A_\infty)^+\cap\partial\B_{R_i})=0$ to conclude
  \begin{equation}\label{eq:Pmu-at-infty}
    \liminf_{k\to\infty}\big[\P(A_k,\overline\Omega\setminus\B_{R_i})-\mu_-(A_k^+\setminus\B_{R_i})\big]
    \ge-\P(A_\infty,\overline\Omega\setminus\B_{R_i})-\mu_-(A_\infty^+\setminus\B_{R_i})-\eps\,,
  \end{equation}
  where now all the single terms are finite. Clearly, on the left-hand side we
  may replace ${-}\mu_-(A_k^+\setminus\B_{R_i})$ with
  $\langle\mu_\pm\ecke(\B_{R_i})^\c\,;A_k\rangle$, which is only larger.
  Adding up \eqref{eq:lsc-Dir-pre} (with $R=R_i$) and this slightly modified
  version of \eqref{eq:Pmu-at-infty}, we get
  \[
    \liminf_{k\to\infty}\Pmu[A_k;\overline\Omega]
    \ge\P(A_\infty,\overline\Omega\cap\B_{R_i})+\langle\mu_\pm\ecke\B_{R_i}\,;A_\infty\rangle
    -\P(A_\infty,\overline\Omega\setminus\B_{R_i})-\mu_-(A_\infty^+\setminus\B_{R_i})-\eps
  \]
  for all $i\in\N$. We now rewrite
  $\langle\mu_\pm\ecke\B_{R_i}\,;A_\infty\rangle-\mu_-(A_\infty^+\setminus\B_{R_i})
  =\mu_+(A_\infty^1\cap\B_{R_i})-\mu_-(A_\infty^+)$ on the right-hand side, send $i\to\infty$,
  and exploit $\lim_{i\to\infty}R_i=\infty$. Keeping in mind that
  $\P(A_\infty,\overline\Omega)<\infty$ and $\mu_-(A_\infty^+)<\infty$ in the presently
  considered case and finally exploiting the arbitrariness of $\eps$, we then
  obtain the claim \eqref{eq:lsc-Dir} also in the situation
  \eqref{item:lsc-Dir-b}.

  Finally, in order to handle the situation \eqref{item:lsc-Dir-c} it suffices
  to slightly adapt the estimate \eqref{eq:mu-at-infty} in the reasoning used for
  \eqref{item:lsc-Dir-b}. Indeed, now we simply take $R_i\in{(0,\infty)}$ rather
  than $R_i\in{(R_\eps,\infty)}$, and only eventually, given an arbitrary
  $\eps>0$, we exploit the global convergence $\lim_{k\to\infty}|A_k\Delta A_\infty|=0$
  assumed in \eqref{item:lsc-Dir-c} to find
  \[\begin{aligned}
    \mu_-((A_k^+\Delta A_\infty^+)\setminus\B_{R_i})
    &\le\mu_-(((A_k\Delta A_\infty)\setminus\B_{R_i})^+)
    \le\P((A_k\Delta A_\infty)\setminus\B_{R_i})+\eps
    =\P((A_k\Delta A_\infty)\setminus\B_{R_i},\overline\Omega)+\eps\\
    &\le\P(A_k,\overline\Omega\setminus\B_{R_i})+\P(A_\infty,\overline\Omega\setminus\B_{R_i})+\H^{n-1}((A_k\Delta A_\infty)^+\cap\partial\B_{R_i})+\eps
  \end{aligned}\]
  for $k\gg1$. This is enough to establish in the limit $k\to\infty$ the
  estimate \eqref{eq:Pmu-at-infty}\footnote{In fact, since in the line of
  argument based on \eqref{item:lsc-Dir-c} the radii $R_i$ do not depend on
  $\eps$, one can exploit the arbitrariness of $\eps$ earlier in the argument to
  deduce the validity of \eqref{eq:Pmu-at-infty} in fact even without the
  $\eps$-term.} --- now under the assumptions of \eqref{item:lsc-Dir-c}, but
  still only in case $\mu_-(A_k^+){+}\mu_-(A_\infty^+)<\infty$. We can thus carry out
  the remainder of the reasoning and establish \eqref{eq:lsc-Dir} exactly as in
  the situation \eqref{item:lsc-Dir-b}.
\end{proof}

Exploiting the semicontinuity result in a more or less standard way we obtain
the following existence theorem for the functional in \eqref{eq:P-Dir}.

\begin{thm}[existence in Dirichlet problems]\label{thm:ex-Dir}
  For an open set\/ $\Omega$ in $\R^n$, assume that non-negative Radon measures
  $\mu_+$ and $\mu_-$ on $\R^n$ with $\mu_\pm\ecke(\overline\Omega)^\c\equiv0$
  both satisfy the small-volume IC in $\R^n$ with constant\/ $1$. Moreover,
  consider $A_0\in\M(\R^n)$ with
  $\mu_+(A_0^1){+}\P(A_0,\overline\Omega)<\infty$, and assume that \emph{one of the
  following sets of additional assumptions} is valid\textup{:}
  \begin{enumerate}[{\rm(a)}]
  \item The measure $\mu_-$ is \emph{finite}.\label{item:ex-Dir-a}
  \item For some $R_0\in{(0,\infty)}$ and $\gamma\in{(0,1]}$, the measure
    $\mu_-$ additionally satisfies the \emph{strong} IC in ${(\B_{R_0})}^\c$
    with constant\/ $1{-}\gamma$.\label{item:ex-Dir-b}
  \end{enumerate}
  Then, for $n\ge2$, there exists the minimum of the \ka generalized\kz{}
  Dirichlet problem
  \begin{equation}\label{eq:Dir-fvol}
    \min\{\Pmu[E;\overline\Omega]\,:\,E\in\mathscr{F}_{A_0}(\Omega)\}\,,
  \end{equation}
  and moreover, in situation \eqref{item:ex-Dir-a} with $n\ge1$, there also
  exists the minimum of the variant of the problem
  \begin{equation}\label{eq:Dir-nfvol}
    \min\{\Pmu[E;\overline\Omega]\,:\,E\in\mathscr{D}_{A_0}(\Omega)\}\,.
  \end{equation}
  The minimum values in the situation \eqref{item:ex-Dir-a} are in
  ${[{-}\mu_-(\R^n),\infty)}$, and the minimum value in the situation
  \eqref{item:ex-Dir-b} is in
  ${[{-}(1{-}\gamma)\P(A_0,\overline\Omega){-}(1{-}\gamma)\P(\B_{R_0}){-}\mu_-(A_0^+){-}\mu_-(\overline{\B_{R_0}}),\infty)}$.
\end{thm}

In connection with this theorem let us first set clear that the functional
$\Pmu[\,\cdot\,;\overline\Omega]$ is well-defined on the admissible $E$.
Indeed, in the situation \eqref{item:ex-Dir-a} thanks to the
finiteness of $\mu_-$ we evidently have
$\Pmu[E;\overline\Omega]\in{({-}\infty,\infty]}$ for all
$E\in\mathscr{D}_{A_0}(\Omega)$ and a fortiori for
$E\in\mathscr{F}_{A_0}(\Omega)$. Moreover, in the situation
\eqref{item:ex-Dir-b} we get from the assumption
$\mu_+(A_0^1){+}\P(A_0,\overline\Omega)<\infty$ and
Lemma \ref{lem:finite-Dir} that $\mu_+(E^1){+}\P(E,\overline\Omega)<\infty$
and consequently $\Pmu[E;\overline\Omega]\in{[{-}\infty,\infty)}$
hold at least for all $E\in\mathscr{F}_{A_0}(\Omega)$.

We further remark that if only \eqref{item:ex-Dir-b} but not
\eqref{item:ex-Dir-a} is satisfied (in particular $\mu_-(\R^n)=\infty$), we
may still consider \eqref{eq:Dir-nfvol} in the form
\begin{equation}\label{eq:Dir-nfvol'}
  \min\{\Pmu[E;\overline\Omega]\,:\,
  E\in\mathscr{D}_{A_0}(\Omega)\,,\,\Pmu[E;\overline\Omega]\text{ defined}\}\,,
\end{equation}
where we recall that $\Pmu[E;\overline\Omega]$ is defined for
$E\in\mathscr{D}_{A_0}(\Omega)$ precisely if $\min\{\mu_+(E^1),\mu_-(E^+)\}<\infty$.
However, in fact this does not win much when compared to \eqref{eq:Dir-fvol},
and thus we have excluded this situation above and only comment on it
briefly. Indeed, in case $n\ge2$, $\Cp_1((\Omega^1)^\c)=\infty$, Remark
\ref{rem:lsc-Dir}\eqref{item:rem-lsc-Dir-ii} gives
$\mathscr{D}_{A_0}(\Omega)=\mathscr{F}_{A_0}(\Omega)$, and
\eqref{eq:Dir-nfvol'} reduces to precisely \eqref{eq:Dir-fvol} (also keeping in
mind that we have already argued for the finiteness of the $\mu_+$-term on
$\mathscr{F}_{A_0}(\Omega)$). Moreover, in case $n\ge2$,
$\Cp_1((\Omega^1)^\c)<\infty$ we can modify\footnote{In fact,
in view of $\Cp_1((\Omega^1)^\c)<\infty$ there exists $H\in\BVs(\R^n)$ with
$\Omega^\c\subset H$ up to negligible sets, and the problem under consideration
stays unchanged when replacing $A_0$ with $A_0\cap H$, which clearly satisfies
$|A_0\cap H|\le|H|<\infty$.} $A_0$ inside $\Omega$ to ensure $|A_0|<\infty$ and
then obtain from Remark \ref{rem:lsc-Dir}\eqref{item:rem-lsc-Dir-iii} that the
sets $E\in\mathscr{D}_{A_0}(\Omega)$ split into some with
$E\in\BVs(\overline\Omega)$ and thus $E\in\mathscr{F}_{A_0}(\Omega)$ on one hand
and some with $E^\c\in\BVs(\overline\Omega)$ on the other hand. However, in the
case considered it turns out\footnote{The precise reasoning proceeds as follows
and exploits that $H\in\BVs(\R^n)$ from the previous footnote also satisfies
$(\overline\Omega)^\c\subset H^1$. In case $\mu_+(\R^n)<\infty=\mu_-(\R^n)$,
from $E^\c\in\BVs(\overline\Omega)$ we get first $E^\c\cup H\in\BVs(\R^n)$, then
$\mu_-((E^+)^\c)\le\mu_-((E^\c\cup H)^+)<\infty$ via Lemma \ref{lem:finite},
then $\mu_-(E^+)=\infty$, and finally
$\Pmu[E;\overline\Omega]={-}\infty$. In case
$\mu_+(\R^n)=\infty=\mu_-(\R^n)$, essentially the same reasoning leads from
$E^\c\in\BVs(\overline\Omega)$ to $\mu_-(E^+)=\mu_+(E^1)=\infty$, and thus
$\Pmu[E;\overline\Omega]$ is undefined.} that either 
$\Pmu[E;\overline\Omega]$ equals ${-}\infty$ whenever
$E^\c\in\BVs(\overline\Omega)$ or $\Pmu[E;\overline\Omega]$ is
undefined whenever $E^\c\in\BVs(\overline\Omega)$. Thus, either
\eqref{eq:Dir-nfvol'} is a rather trivial extension of \eqref{eq:Dir-fvol}, or
\eqref{eq:Dir-nfvol'} reduces to precisely \eqref{eq:Dir-fvol} once more.

\begin{proof}
  The admissible classes in both \eqref{eq:Dir-fvol} and \eqref{eq:Dir-nfvol}
  contain $A_0$. Thus, these classes are non-empty, and in view of
  $\mu_+(A_0^1){+}\P(A_0,\overline\Omega)<\infty$ the corresponding infima are
  in ${[-\infty,\infty)}$. Moreover, in view of
  $\mu_\pm\ecke(\overline\Omega)^\c\equiv0$ the problems in \eqref{eq:Dir-fvol} and
  \eqref{eq:Dir-nfvol} remain unchanged if we modify $A_0$ away from
  $\overline\Omega$. Hence, we can and do assume $A_0\in\BVs_\loc(\R^n)$, which
  implies that the admissible classes are contained in $\BVs_\loc(\R^n)$.
  
  We now focus, for a moment, on the situation \eqref{item:ex-Dir-a}. In view of
  $\mu_-(\R^n)<\infty$ and
  \[
    \Pmu[E;\overline\Omega]\ge\P(E,\overline\Omega)-\mu_-(\R^n)
    \qq\text{for all }E\in\M(\R^n)
  \]
  we find that every minimizing sequence $(A_k)_{k\in\N}$ in either
  \eqref{eq:Dir-fvol} or \eqref{eq:Dir-nfvol} satisfies
  $\limsup_{k\to\infty}\P(A_k,\overline\Omega)<\infty$.

  Next we turn to the situation \eqref{item:ex-Dir-b}. We can assume
  $\mu_-(A_0^+)<\infty$, as otherwise $A_0$ with
  $\Pmu[A_0;\overline\Omega]={-}\infty$ clearly minimizes.
  For $E\in\mathscr{F}_{A_0}(\Omega)$, since we have $|E\Delta A_0|<\infty$ and
  $E\Delta A_0\subset\Omega$, the strong IC yields
  \[\begin{aligned}
    \mu_-((E^+\Delta A_0^+)\setminus\overline{\B_{R_0}})
    &\le\mu_-(((E\Delta A_0)\setminus\B_{R_0})^+)\\
    &\le(1{-}\gamma)\P((E\Delta A_0)\setminus\B_{R_0})\\
    &=(1{-}\gamma)\P((E\Delta A_0)\setminus\B_{R_0},\overline\Omega)\\
    &\le(1{-}\gamma)\P(E,\overline\Omega)+(1{-}\gamma)\P(A_0,\overline\Omega)+(1{-}\gamma)\P(\B_{R_0})\,,
  \end{aligned}\]
  and from this estimate we infer $\mu_-(E^+)<\infty$ and
  \[
    \Pmu[E;\overline\Omega]
    \ge\gamma\P(E,\overline\Omega)-(1{-}\gamma)\P(A_0,\overline\Omega)-(1{-}\gamma)\P(\B_{R_0})-\mu_-(A_0^+)-\mu_-(\overline{\B_{R_0}})
    \qq\text{for all }E\in\mathscr{F}_{A_0}(\Omega)\,.
  \]
  Thus, for every minimizing sequence $(A_k)_{k\in\N}$ in \eqref{eq:Dir-fvol},
  we obtain once more $\limsup_{k\to\infty}\P(A_k,\overline\Omega)<\infty$.

  In any of the cases considered in the statement we further proceed as
  follows. Fixing a minimizing sequence $(A_k)_{k\in\N}$, from
  $\limsup_{k\to\infty}\P(A_k,\overline\Omega)<\infty$ together with
  $A_k\setminus\Omega=A_0\setminus\Omega$ we get
  $\limsup_{k\to\infty}\P(A_k,U)<\infty$ for some open neighborhood $U$ of
  $\overline\Omega$ and in view of $A_0\in\BVs_\loc(\R^n)$ also
  $\limsup_{k\to\infty}\P(A_k,\B_R)<\infty$ for every $R\in{(0,\infty)}$. By
  compactness, a diagonal argument, and lower semicontinuity of the perimeter,
  we deduce that a subsequence of $(A_k)_{k\in\N}$ converges \emph{locally}
  in measure on $\R^n$ to $A_\infty\in\mathscr{D}_{A_0}(\Omega)$ (with even
  $\P(A_\infty,U)<\infty$). In case of problem \eqref{eq:Dir-fvol} we additionally
  involve the isoperimetric estimate of Theorem \ref{thm:isoperi} to derive
  the subsidiary estimate $|A_k\Delta A_0|\le\Gamma_n\P(A_k\Delta A_0)^\frac n{n-1}
  \le\Gamma_n[\P(A_k,\overline\Omega){+}\P(A_0,\overline\Omega)]^\frac n{n-1}$,
  which implies $|A_\infty\Delta A_0|<\infty$ also for the limit $A_\infty$ and thus ensures
  the admissibility of $A_\infty$ and $|A_k\Delta A_\infty|<\infty$ for all $k\in\N$. Finally,
  we apply Theorem \ref{thm:lsc-Dir}\eqref{item:lsc-Dir-a} in situation
  \eqref{item:ex-Dir-a} and Theorem \ref{thm:lsc-Dir}\eqref{item:lsc-Dir-b} in
  situation \eqref{item:ex-Dir-b} to conclude that the limit $A_\infty$ is a minimizer
  in \eqref{eq:Dir-fvol} and \eqref{eq:Dir-nfvol}, respectively (where, as we
  recall, in situation \eqref{item:ex-Dir-b} we consider \eqref{eq:Dir-fvol}
  only).
\end{proof}

\section{Properties and reformulations of isoperimetric conditions}\label{sec:admis-meas}

In this section we take a closer look at ICs, specifically small-volume ICs,
and equivalent ways to express these conditions. Most (though not really all) of
the results obtained in this regard will find uses in the subsequent sections.

\begin{rem}
  Even though we will not work with the observations of this remark any further,
  we briefly record that the $\eps$-$\delta$-feature of the small-volume IC
  can be reformulated in the following standard way. Given a Radon measure $\mu$
  on an open set\/ $\Omega\subset\R^n$, the small-volume IC for $\mu$ in
  $\Omega$ with constant\/ $C\in{[0,\infty)}$ means nothing but the existence of
  a modulus $\omega\colon{[0,\infty]}\to{[0,\infty]}$ with
  $\lim_{t\searrow0}\omega(t)=\omega(0)=0$ such that we have
  \begin{equation}\label{eq:IC-omega}
    \mu(A^+)\le C\P(A)+\omega(|A|)
    \qq\text{for all }A\in\M(\R^n)\text{ with\/ }\overline A\subset\Omega\,.
  \end{equation}
  Introducing a modified $1$-capacity ${}^C\mathrm{K}_1^\omega$ by
  ${}^C\mathrm{K}_1^\omega(S)\coleq\inf\{C\P(A)+\omega(|A|)\,:\,
  A\in\M(\R^n)\,,\,S\subset A^+\,,\,\overline A\subset\Omega\}$ \ka with
  understanding $\inf\emptyset=\infty$\kz, one may further recast
  \eqref{eq:IC-omega} in the \ka still\/\kz{} equivalent form
  \[
    \mu(S)\le{}^C\mathrm{K}_1^\omega(S)
    \qq\text{for all }S\in\Bo(\R^n)\,.
  \]
\end{rem}

As shown by the next lemma, there is also some flexibility concerning the
precise class of test sets for ICs.

\begin{lem}\label{lem:closed-vs-compact}
  Consider a Radon measure $\mu$ on an open set\/ $\Omega\subset\R^n$ and
  $C\in{[0,\infty)}$. Then the following assertions \ka where
  \eqref{item:ic-closed} is exactly the definition of the small-volume
  IC in $\Omega$ with constant\/ $C$\kz{} are
  \textbf{equivalent}\textup{:}
  \begin{enumerate}[{\rm(a)}]
  \item For every $\eps>0$, there exists $\delta>0$ such that
    $\mu(A^+)\le C\P(A){+}\eps$ for all $A\in\M(\R^n)$ with
    $\overline A\subset\Omega$, $|A|<\delta$.\label{item:ic-closed}
  \item For every $\eps>0$, there exists $\delta>0$ such that
    $\mu(A^+)\le C\P(A){+}\eps$ for all $A\in\M(\R^n)$ with $A\Subset\Omega$,
    $|A|<\delta$.\label{item:ic-compact}
  \item For every $\eps>0$, there exists $\delta>0$ such that
    $\mu(A^+)\le C\P(A){+}\eps$ for all $A\in\M(\R^n)$ with $A^+\!\subset\!\Omega$,
    $|A|\!<\!\delta$.\label{item:ic-m-closed}
  \end{enumerate}
  The equivalence carries over to corresponding versions of the strong \ka
  instead of small-volume\kz{} IC.
\end{lem}

In the sequel, from this lemma we will only need the equivalence of
\eqref{item:ic-closed} and \eqref{item:ic-compact}, which is trivial for bounded
$\Omega$ and results from a simple cut-off argument in general. In order to
prove the equivalence with \eqref{item:ic-m-closed} in the full generality
stated here, we will make crucial use of the fine approximation result
\cite[Teorema 2]{TamGia90} (which in turn draws on \cite{TamGia89,TamCon91}).

\begin{proof}[Proof of Lemma \ref{lem:closed-vs-compact}]
  Clearly, \eqref{item:ic-m-closed} implies \eqref{item:ic-closed}, and
  \eqref{item:ic-closed} implies \eqref{item:ic-compact}.

  In addition, we now show that \eqref{item:ic-compact} implies
  \eqref{item:ic-closed}. To this end, we fix $\eps>0$ and consider
  $A\in\M(\R^n)$ with $\overline A\subset\Omega$, $|A|<\delta$ for the
  corresponding $\delta$. In view of $A\cap\B_R\Subset\Omega$, from
  \eqref{item:ic-compact} we then get
  \[
    \mu(A^+\cap\B_R)
    =\mu((A\cap\B_R)^+)
    \le C\P(A\cap\B_R)+\eps
    \le C\P(A)+\eps
    \qq\qq\text{for each }R\in{(0,\infty)}\,,
  \]
  where the last estimate can be obtained from Lemmas
  \ref{lem:intersect-with-pseudoconvex} and \ref{lem:convex->pseudoconvex}, for
  instance. In the limit $R\to\infty$ we read off $\mu(A^+)\le C\P(A)+\eps$.

  Next we prove that \eqref{item:ic-closed} implies
  \eqref{item:ic-m-closed}. For this, we fix again $\eps>0$ and consider
  some $A\in\M(\R^n)$ with $A^+\subset\Omega$, $|A|<\delta$ for the
  corresponding $\delta$. Clearly, we can additionally assume
  $\P(A)<\infty$. From the interior approximation result
  \cite[Teorema 2]{TamGia90} we then obtain a sequence of sets $A_k\in\M(\R^n)$
  such that
  \[
    A_k\subset A_{k+1}\subset A\,,\qq\qq
    \overline{A_k}=A_k^+\,,\qq\qq
    \P(A_k)\le\P(A)\qq\qq
    \text{for all }k\in\N
  \]
  (where the crucial condition $\overline{A_k}=A_k^+$ is stated in
  \cite[Teorema 2]{TamGia90} in the equivalent form
  $A_k^0\cap\partial A_k=\emptyset$) and
  \[
    \lim_{k\to\infty}\P(A\setminus A_k)=0\,.
  \]
 In view of $\overline{A_k}=A_k^+\subset A^+\subset\Omega$, from
 \eqref{item:ic-closed} and the preceding properties of $A_k$ we conclude
  \begin{equation}\label{eq:mu<P+eps-TC}
    \mu(A_k^+)\le C\P(A_k)+\eps\le C\P(A)+\eps
    \qq\qq\text{for each }k\in\N\,.
  \end{equation}
  Evidently the above conditions imply $\bigcup_{k=1}^\infty A_k^+\subset A^+$,
  and we now show that, decisively, they also ensure
  \begin{equation}\label{eq:conv-A+-H-ae}
    \mu\bigg(A^+\setminus\bigcup_{k=1}^\infty A_k^+\bigg)=0\,.
  \end{equation}
  Indeed, observing $A^+\setminus\bigcup_{k=1}^\infty A_k^+
  \subset A^+\setminus A_\ell^+\subset (A\setminus A_\ell)^+$ for each
  $\ell\in\N$, from Proposition \ref{prop:inf-P} we first infer
  $\Cp_1\big(A^+\setminus\bigcup_{k=1}^\infty A_k^+\big)
  \le\lim_{\ell\to\infty}\P(A\setminus A_\ell)=0$, then by Proposition
  \ref{prop:Cp1-null} we deduce
  $\H^{n-1}\big(A^+\setminus\bigcup_{k=1}^\infty A_k^+\big)=0$, and finally via
  Lemma \ref{lem:abs-con-H} we arrive at \eqref{eq:conv-A+-H-ae}. With
  \eqref{eq:conv-A+-H-ae} at hand we can then go to the limit $k\to\infty$ in
  \eqref{eq:mu<P+eps-TC} to establish $\mu(A^+)\le C\P(A)+\eps$ in the
  generality of \eqref{item:ic-m-closed}.

  For the strong conditions instead of the small-volume ones, the reasoning
  works in the same way.
\end{proof}

In the specific cases that the measure $\mu$ is finite or supported at positive
distance from $\partial\Omega$, we have further characterizations of the
small-volume IC for $\mu$ in $\Omega$. Indeed, we can allow test sets $A$
reaching up to $\partial\Omega$, can pass to the \emph{relative} perimeter
$\P(A,\Omega)$, or can even state the condition in a fully localized way. This
is detailed in the next statement, where for notational
convenience\footnote{Indeed, if one considers a Radon measure $\mu$ on $\Omega$
and assumes in analogy to Lemma \ref{lem:bdry-local-ic} either finiteness of
$\mu$ or $\dist(\spt\mu,\Omega^\c)>0$, the extension of $\mu$ from $\Omega$ to
$\R^n$ by zero is still a Radon measure. This goes without saying for finite
$\mu$, but is true also when requiring $\dist(\spt\mu,\Omega^\c)>0$, since this
condition improves local finiteness on $\Omega$ to finiteness on all
\emph{bounded} subsets of $\Omega$ and thus ensures local finiteness of the
extension.} we work with a Radon measure $\mu$ defined on full $\R^n$.

\begin{lem}\label{lem:bdry-local-ic}
  Consider an open set\/ $\Omega\subset\R^n$, a Radon measure $\mu$ on $\R^n$,
  and $C\in{[0,\infty)}$. If either $\mu$ is finite with
  $\mu\ecke\Omega^\c\equiv0$ or $\mu$ satisfies $\dist(\spt\mu,\Omega^\c)>0$,
  then the following assertions are \textbf{equivalent}\textup{:}
  \begin{enumerate}[{\rm(a)}]
  \item The measure $\mu$ restricted to $\Omega$ satisfies the small-volume IC
    in $\Omega$ with constant\/ $C$.\label{item:ic-std}
  \item For every $\eps>0$, there exists $\delta>0$ \st{}
    $\mu(A^+)\le C\P(A){+}\eps$ for all $A\in\M(\R^n)$ with
    $|A\setminus\Omega|=0$, $|A|<\delta$.\label{item:ic-to-bdry}
  \item For every $\eps>0$, there exists $\delta>0$ \st{}
    $\mu(A^+)\le C\P(A,\Omega){+}\eps$ for all $A\in\M(\R^n)$ with
    $|A|<\delta$.\label{item:mod-ic-to-bdry}
  \end{enumerate}
  In the case of finite $\mu$ with $\mu\ecke\Omega^\c\equiv0$ one more
  equivalent assertion is\textup{:}
  \begin{enumerate}[{\rm(a)}]
  \refstepcounter{enumi}\refstepcounter{enumi}\refstepcounter{enumi}
\item For every $x\in\Omega$, there exists $r_x>0$ with
  $\B_{r_x}(x)\subset\Omega$ such that $\mu$ restricted to $\B_{r_x}(x)$
  satisfies the small-volume IC in $\B_{r_x}(x)$ with constant\/
  $C$.\label{item:ic-local}
  \end{enumerate}
\end{lem}

Here the implications \eqref{item:mod-ic-to-bdry}\!$\implies$\!%
\eqref{item:ic-to-bdry}\!$\implies$\!\eqref{item:ic-std}\!$\implies$\!%
\eqref{item:ic-local} are simple generalities, while the reverse implications
are non-trivial and draw crucially on the assumption that $\mu$ is
finite or satisfies $\dist(\spt\mu,\Omega^\c)>0$. Indeed, setting
$h_k\coleq\sum_{i=1}^k\frac1i\in\R$, we record that
\eqref{item:ic-to-bdry}\!$\implies$\!\eqref{item:mod-ic-to-bdry} fails for the
\emph{infinite} Radon measure $\mu=2C\sum_{k=1}^\infty\delta_{h_{3k}}$
on $\R$ with $C>0$ and $\Omega=\bigcup_{k=1}^\infty{(h_{3k-1},h_{3k+1})}$, while
\eqref{item:ic-std}\!$\implies$\!\eqref{item:ic-to-bdry} and
\eqref{item:ic-local}\!$\implies$\!\eqref{item:ic-std} fail for the same measure
together with $\Omega=\bigcup_{k=1}^\infty{(h_{3k-2},h_{3k+1})}$ and $\Omega=\R$,
respectively.

In addition, also the $\eps$-$\delta$-nature of the small-volume IC is crucial
for Lemma \ref{lem:bdry-local-ic} insofar that the simple implications
\eqref{item:mod-ic-to-bdry}\!$\implies$\!\eqref{item:ic-to-bdry}\!$\implies$\!%
\eqref{item:ic-std}\!$\implies$\!\eqref{item:ic-local} carry over by analogy
to a strong-IC case with $\eps$ and $\delta$ removed, while the reverse
implications do not have analogs there. Indeed, the strong-IC analog of
\eqref{item:ic-to-bdry}\!$\implies$\!\eqref{item:mod-ic-to-bdry} fails for the
\emph{finite} Radon measure $\mu=2C(\delta_{-2}{+}\delta_2)$ on $\R$ together
with $\Omega={({-}3,{-}1)}\,\dcup{(1,3)}$, while the analoga of
\eqref{item:ic-std}\!$\implies$\!\eqref{item:ic-to-bdry} and
\eqref{item:ic-local}\!$\implies$\!\eqref{item:ic-std} fail for the same measure
together with $\Omega={({-}3,3)}\setminus\{0\}$ and $\Omega={(-3,3)}$,
respectively.

Furthermore, all counterexamples mentioned here can be easily adapted to work in
$\R^n$ instead of $\R$.

\begin{proof}[Proof of Lemma \ref{lem:bdry-local-ic}]
  As already observed, the implications \eqref{item:mod-ic-to-bdry}\!$\implies$\!%
  \eqref{item:ic-to-bdry}\!$\implies$\!\eqref{item:ic-std}\!$\implies$\!%
  \eqref{item:ic-local} are straightforward.

  Next we prove that \eqref{item:ic-std} implies \eqref{item:mod-ic-to-bdry}. We
  record that $d\colon\R^n\to{(0,\infty)}$, given by
  $d(x)\coleq\dist(x,\Omega^c)$, is Lipschitz with constant $1$ and then
  by Rademacher's theorem satisfies $|\nabla d|\le1$ \ae{} on
  $\Omega$. Moreover, since $\Omega$ is open, we have
  $\Omega=\bigcup_{t>0}\{d>t\}$. Now we consider an arbitrary $\eps>0$. Then, in
  case of finite $\mu$ with $\mu\ecke\Omega^\c\equiv0$ we can fix a corresponding
  $t_0>0$ such that $\mu(\{d<t_0\})<\frac\eps3$ holds, while in
  case $\dist(\spt\mu,\Omega^\c)>0$ we are even in position to ensure
  $\mu(\{d<t_0\})=0$. In addition, we fix $\delta>0$ such that the standard form
  of the small-volume IC in $\Omega$ from \eqref{item:ic-std}
  applies with this $\delta$ and $\frac\eps3$ in place of $\eps$, and we
  consider $A\in\M(\R^n)$ with $|A|<\min\{\delta,\frac{t_0\eps}{3C}\}$. Via the
  coarea formula of Theorem \ref{thm:Lip-coarea} we get
  \[
    \int_0^{t_0}\H^{n-1}(A^+\cap\{d=t\})\,\d t
    =\int_{A^+\cap\{d<t_0\}}|\nabla d|\dx
    \le|A^+|<\frac{t_0\eps}{3C}
  \]
  and can thus choose $t\in{(0,t_0)}$ with
  \begin{equation}\label{eq:cut-off-bdry}
    \H^{n-1}(A^+\cap\{d=t\})<\frac\eps{3C}
  \end{equation}
  (where for $C=0$ an arbitrary $t\in{(0,t_0)}$ suffices). We now cut off portions
  of $A$ close to $\partial\Omega$ by introducing $E\coleq A\cap\{d>t\}$,
  for which clearly $\overline E\subset\{d\ge t\}\subset\Omega$ and $|E|\le|A|<\delta$
  hold. Estimating via the choice of $t_0$, the small-volume IC from
  \eqref{item:ic-std} (with $\frac\eps3$ in place of $\eps$),
  Lemma \ref{lem:P(AcapB),P(A-S)}, and \eqref{eq:cut-off-bdry}, we then arrive at
  \[\begin{aligned}
    \mu(A^+)
    &\le\mu(A^+\cap\{d>t\})+\mu(\{d<t_0\})
    \le\mu(E^+)+\frac\eps3
    \le C\P(E)+\frac{2\eps}3\\
    &\le C\P(A,\Omega)+C\H^{n-1}(A^+\cap\{d=t\})+\frac{2\eps}3
    \le C\P(A,\Omega)+\eps\,.
  \end{aligned}\]
  Thus, we obtain $\mu(A^+\cap\Omega)\le C\P(A,\Omega)+\eps$ in the setting of
  \eqref{item:mod-ic-to-bdry}.

  Finally, in case of finite $\mu$ with $\mu\ecke\Omega^\c\equiv0$ we show that
  \eqref{item:ic-local} implies \eqref{item:mod-ic-to-bdry}. To this end we fix
  once more some $\eps>0$. We then apply Vitali's covering theorem (see
  \cite[Theorem 2.8]{Mattila95}, for instance) to the family of all balls
  $\B_r(x)$ with $x\in\Omega$ and $r\le r_x$ and exploit $\mu(\Omega)<\infty$ to
  obtain finite number $k\in\N$ of disjoint balls $\B_{\r_i}(x_i)$ with
  $x_i\in\Omega$ and $\r_i\le r_{x_i}$ for $i\in\{1,2,\ldots,k\}$ such that it
  holds
  \[
    \mu\bigg(\Omega\setminus\bigcup_{i=1}^k\B_{\r_i}(x_i)\bigg)\le\frac\eps2\,.
  \]
  Now the assumption \eqref{item:ic-local} guarantees the validity of
  \eqref{item:ic-std} on each of the balls
  $\B_{\r_i}(x_i)\subset\B_{r_{x_i}}(x_i)$ with $i\in\{1,2,\ldots,k\}$ in place
  of $\Omega$. Since we have already shown that \eqref{item:ic-std} implies
  \eqref{item:mod-ic-to-bdry}, we also have \eqref{item:mod-ic-to-bdry} on each
  of these balls. Since the number of balls is finite, this in turn yields a
  common $\delta>0$ such that we have
  \[
    \mu(A^+\cap\B_{\r_i}(x_i))\le C\P(A,\B_{\r_i}(x_i))+\frac\eps{2k}
  \]
  for all $A\in\M(\R^n)$ with $|A|<\delta$ and all $i\in\{1,2,\ldots,k\}$. In
  conclusion, for all $A\in\M(\R^n)$ with $|A|<\delta$, we achieve
  \[
    \mu(A^+)
    \le\sum_{i=1}^k\mu(A^+\cap\B_{\r_i}(x_i))+\mu\bigg(\Omega\setminus\bigcup_{i=1}^k\B_{\r_i}(x_i)\bigg)
    \le\sum_{i=1}^k\Big[C\P(A,\B_{\r_i}(x_i))+\frac\eps{2k}\Big]+\frac\eps2
    \le C\P(A,\Omega)+\eps\,,
  \]
  where the disjointness of $\B_{\r_i}(x_i)$ is used in the last step. In this
  way we arrive at \eqref{item:mod-ic-to-bdry}.
\end{proof}

As a rather unexpected consequence of Lemma \ref{lem:bdry-local-ic}, we next
derive that the small-volume IC with a fixed constant actually carries over to
the sum of two (or finitely many) mutually singular measures with still the same
constant. Clearly, for the strong IC, one cannot draw an analogous conclusion in
comparable generality.

\begin{prop}[small-volume IC for a sum of singular measures]\label{prop:ic-sing}
  Consider non-negative Radon measures $\mu_1$, $\mu_2$ on $\R^n$ which are
  singular to each other in the sense that there exists a decomposition
  $\R^n=S_1\dcup S_2$ into $S_1,S_2\in\Bo(\R^n)$ with
  $\mu_1(S_1^\c)=\mu_2(S_2^\c)=0$. Further suppose that either $\mu_1$ is finite
  or $\dist(\spt\mu_1,\spt\mu_2)>0$ holds. Then, if $\mu_1$ and $\mu_2$ both
  satisfy the small-volume IC on $\R^n$ with constant $C\in{[0,\infty)}$, also
  $\mu_1{+}\mu_2$ satisfies the small-volume IC on $\R^n$ with the same constant
  $C$.
\end{prop}

From the example in the later Remark
\ref{rem:infinite-P}\eqref{item:infinite-P-without IC} it becomes clear that the
extra assumptions in the proposition (either one measure finite or supports
at positive distance) cannot be dropped.

\begin{proof}
  We start with the case that $\mu_1$ is finite. Given an arbitrary
  $\eps>0$, the finiteness of $\mu_1$ together with
  $\mu_1(S_1^\c)=\mu_2(S_2^\c)=0$ yields the existence of a compact set
  $K_1\subset S_1$ and a closed set $C_2\subset S_2$ such that
  $\mu_1(K_1^\c){+}\mu_2(C_2^\c)\le\eps$. In view of $\dist(K_1,C_2)>0$
  we can choose \emph{disjoint} open sets $O_1\supset K_1$ and $O_2\supset C_2$
  and can also ensure $\dist(C_2,O_2^\c)>0$. Since the closedness of $C_2$
  yields $\spt(\mu_2\ecke C_2)\subset C_2$, we can then apply
  \eqref{item:ic-std}\!$\implies$\!\eqref{item:mod-ic-to-bdry}
  from Lemma \ref{lem:bdry-local-ic} on one hand for the finite measure
  $\mu_1\ecke K_1$, on the other hand for the possibly infinite measure
  $\mu_2\ecke C_2$ with $\dist(\spt(\mu_2\ecke C_2),O_2^\c)>0$ to obtain some
  $\delta>0$ such that we have $\mu_1(A^+\cap K_1)\le C\P(A,O_1){+}\eps$ and
  $\mu_2(A^+\cap C_2)\le C\P(A,O_2){+}\eps$ for all $A\in\M(\R^n)$ with
  $|A|<\delta$. Consequently, for such sets we also get
  \[
    (\mu_1{+}\mu_2)(A^+)\le\mu_1(A^+\cap K_1)+\mu_2(A^+\cap C_2)+\eps
    \le C\P(A,O_1)+C\P(A,O_2)+3\eps\le\C\P(A)+3\eps\,,
  \]
  which yields the claim.
  
  The case of $\dist(\spt\mu_1,\spt\mu_2)>0$ is a bit simpler, since we can
  directly choose disjoint open sets $O_1\supset\spt\mu_1$ and
  $O_2\supset\spt\mu_2$ with $\dist(\spt\mu_1,O_1^\c)>0$ and
  $\dist(\spt\mu_2,O_2^\c)>0$. Then, we can apply
  \eqref{item:ic-std}\!$\implies$\!\eqref{item:mod-ic-to-bdry} from Lemma
  \ref{lem:bdry-local-ic} to both $\mu_1=\mu_1\ecke O_1$ and $\mu_2=\mu_2\ecke
  O_2$ and conclude the reasoning as before.
\end{proof}

In the sequel we record that ICs can be expressed not only with test sets, but
also with test functions and partially in a distributional way. This is detailed
in the following (almost) twin theorems, where the one for the strong-IC case is
a minor variant of known results from \cite[Theorem 4.7]{MeyZie77},
\cite[Theorem 5.12.4]{Ziemer89}, \cite[Section 2]{DuzSte96},
\cite[Theorem 3.3, Theorem 3.5]{PhuTor08}, \cite[Theorem 4.4]{PhuTor17}, while
the adaptation to the small-volume case does not seem to have direct
predecessors in the literature. As a side benefit it turns out in this context
that the measure-theoretic closure $A^+$ can be replaced with the
measure-theoretic interior $A^1$ in the formulation of both types of ICs.

\begin{thm}[characterizations of the strong IC]\label{thm:char-strong-ic}
  For a Radon measure $\mu$ on an open set\/ $\Omega\subset\R^n$ and a constant
  $C\in{[0,\infty)}$, the following assertions are \textbf{equivalent} with each
  other\textup{:}
  \begin{enumerate}[{\rm(a)}]
  \item The strong IC holds for $\mu$ in $\Omega$ with
    constant\/ $C$.\label{item:mu(A^+)<P}
  \item We have $\mu(A^1)\le C\P(A)$ for all $A\in\M(\R^n)$ with $\overline
    A\subset\Omega$ and\/ $|A|<\infty$.\label{item:mu(A^1)<P}
  \item We have $\int_\Omega\eta\,\d\mu\le C\int_\Omega|\nabla\eta|\dx$
    for all non-negative functions
    $\eta\in\C^\infty_\cpt(\Omega)$.\label{item:est-smooth-grad}
  \item We have $\mu(N)=0$ for all\/ $\H^{n-1}$-negligible $N\in\Bo(\Omega)$ and\/
    $\int_\Omega|v^\ast|\,\d\mu\le C\int_\Omega|\nabla v|\dx$ for all\/
    $v\in\W_0^{1,1}(\Omega)$.\label{item:est-W11-grad}
  \item We have $\mu=\Div\sigma$ in the sense of distributions on $\Omega$ for
    some vector field $\sigma\in\mathrm{L}^\infty(\Omega,\R^n)$ with
    $\|\sigma\|_{\L^\infty(\Omega,\R^n)}\le C$.\label{item:sub-C-div}
  \end{enumerate}
\end{thm}

\begin{thm}[characterizations of the small-volume IC]\label{thm:char-small-vol-ic}
  For a Radon measure on an open set\/ $\Omega\subset\R^n$ and a constant
  $C\in{[0,\infty)}$, the following assertions are \textbf{equivalent} with each other\textup{:}
  \begin{enumerate}[{\rm(a)}]
  \item The small-volume IC holds for $\mu$ in $\Omega$
    with constant $C$.\label{item:mu(A^+)<P+eps}
  \item For every $\eps>0$, there exists some $\delta>0$ such that we have
    $\mu(A^1)\le C\P(A)+\eps$ for all $A\in\M(\R^n)$ with $\overline
    A\subset\Omega$ and\/ $|A|<\delta$.\label{item:mu(A^1)<P+eps}
  \item There exists a modulus $\omega\colon{[0,\infty)}\to{[0,\infty]}$ with
    $\lim_{t\searrow0}\omega(t)=\omega(0)=0$ such that we have
    $\int_\Omega\eta\,\d\mu\le C\int_\Omega|\nabla\eta|\dx+\omega\big(|\spt\eta|\big)$
    for all\/ $\eta\in\C^\infty_\cpt(\Omega)$ with $0\le\eta\le1$ on
    $\Omega$.\label{item:omega-est-smooth-grad}
  \item We have $\mu(N)=0$ for all\/ $\H^{n-1}$-negligible $N\in\Bo(\Omega)$,
    and, for every $\eps>0$, there exists some $\delta>0$ such that we have
    $\int_\Omega|v^\ast|\,\d\mu\le C\int_\Omega|\nabla v|\dx+\eps\sup_\Omega|v|$
    for all $v\in\W_0^{1,1}(\Omega)\cap\L^\infty(\Omega)$ with
    $|\{v\neq0\}|<\delta$.\label{item:eps-delta-est-W11-grad}
  \end{enumerate}
  In addition, the \textbf{subsequent property at least implies} each of the
  preceding ones\textup{:}
  \begin{enumerate}[{\rm(a)}]
  \refstepcounter{enumi}\refstepcounter{enumi}\refstepcounter{enumi}\refstepcounter{enumi}
  \item We have $\mu=H\Ln{+}\Div\sigma$ in the sense of distributions on
    $\Omega$ for some vector field $\sigma\in\mathrm{L}^\infty(\Omega,\R^n)$ with
    $\|\sigma\|_{\L^\infty(\Omega,\R^n)}\le C$ and some function
    $H\in\L^1(\Omega)$.\label{item:sub-C-div-plus-L1}
  \end{enumerate}
\end{thm}

Here, the extra terms distinguishing Theorem \ref{thm:char-small-vol-ic} from Theorem
\ref{thm:char-strong-ic} have been incorporated in slightly different forms, but
indeed the formulations are to some extent interchangable. However, a subtlety
related to Lemma \ref{lem:control-spt} is that in condition
\eqref{item:eps-delta-est-W11-grad} it seems decisive to require smallness for
$|\{v\neq0\}|$ (or alternatively for any $\L^p$ norm of $v$), but \emph{not} in
fact for $|{\spt v}|$.

In the sequel we first detail the proof of Theorem \ref{thm:char-small-vol-ic} and then
comment on the necessary adaptations needed to cover the case of Theorem
\ref{thm:char-strong-ic} as well.

\begin{proof}[Proof of Theorem \ref{thm:char-small-vol-ic}]
  Since we have $A^1\subset A^+$ by definition, it is clear that
  \eqref{item:mu(A^+)<P+eps} implies \eqref{item:mu(A^1)<P+eps}.

  \smallskip

  We start by proving that \eqref{item:mu(A^1)<P+eps} implies
  \eqref{item:omega-est-smooth-grad}. We denote by $\delta_i>0$ the value of
  $\delta$ which corresponds to $\eps=\frac1i$ in \eqref{item:mu(A^1)<P+eps}, we
  assume $\delta_{i+1}<\delta_i$ for $i\in\N$, and we choose the modulus
  $\omega\coleq\sum_{i=1}^\infty\frac1i\1_{[\delta_{i+1},\delta_i)}+\infty\1_{[\delta_1,\infty)}$.
  We now consider $\eta\in\C^\infty_\cpt(\Omega)$ with $0\le\eta\le1$ on $\Omega$.
  If $\eta$ vanishes identically or we have
  $|{\spt\eta}|\ge\delta_1$, the claim is trivially valid. Otherwise
  we henceforth fix $i\in\N$ with $|{\spt\eta}|\in{[\delta_{i+1},\delta_i)}$ and
  thus $\omega(|{\spt\eta}|)=\frac1i$. We observe that $\{\eta>t\}$ is open and
  thus $\{\eta>t\}\subset\{\eta>t\}^1$ holds for all $t\in\R$. Then, via a
  layer-cake type rewriting, the estimate from \eqref{item:mu(A^1)<P+eps} for
  $\{\eta>t\}\Subset\Omega$ with $|\{\eta>t\}|<\delta_i$, and the coarea formula
  of Theorem \ref{thm:BV-coarea} we get
  \[
    \int_\Omega\eta\,\d\mu
    =\int_0^1\mu(\{\eta>t\})\,\d t
    \le\int_0^1\mu(\{\eta>t\}^1)\,\d t
    \le\int_0^1\Big[\P(\{\eta>t\})+{\ts\frac1i}\Big]\,\d t
    =\int_\Omega|\nabla\eta|\dx+\omega(|{\spt\eta}|)\,.
  \]
  This gives the property \eqref{item:omega-est-smooth-grad}.

  \smallskip
  
  Next we verify that \eqref{item:omega-est-smooth-grad} implies
  \eqref{item:eps-delta-est-W11-grad}. In order to show $\mu(N)=0$ for an
  $\H^{n-1}$-negligible $N\in\Bo(\Omega)$, we slightly adapt the proof of Lemma
  \ref{lem:abs-con-H}. Indeed, we can assume $N\Subset\Omega$. Given $\eps>0$,
  Lemma \ref{lem:negligible} yields an open $A$ with $N\subset A\Subset\Omega$,
  $|A|<\eps$, $\P(A)<\eps$, and by mollifying the $\1_A$ we obtain
  $\eta\in\C^\infty_\cpt(\Omega)$ with $\1_N\le\eta\le1$ on $\Omega$,
  $|{\spt\eta}|<\eps$, and $\int_\Omega|\nabla\eta|<\eps$. Exploiting the
  estimate from \eqref{item:omega-est-smooth-grad} for this $\eta$, we find
  $\mu(N)<C\eps+\sup_{[0,\eps)}\omega$. As $\eps>0$ is arbitrary, we end up with
  $\mu(N)=0$. We now derive the main inequality in
  \eqref{item:eps-delta-est-W11-grad}. Given $\eps>0$ we fix $\delta>0$ such
  that $\sup_{[0,\delta)}\omega\le\eps$. We consider
  $v\in\W_0^{1,1}(\Omega)\cap\L^\infty(\Omega)$ with
  $|\{v\neq0\}|<\delta$ and may additionally assume $\sup_\Omega|v|=1$. We
  record $|v|\in\W_0^{1,1}(\Omega)\cap\L^\infty(\Omega)$ with
  $|\nabla|v||=|\nabla v|$ \ae{} and choose
  $\eta_k\in\C^\infty_\cpt(\Omega)$ with $0\le\eta_k\le1$ on $\Omega$ such that
  $\eta_k$ converge to $|v|$ in $\W^{1,1}(\Omega)$. Involving
  $|\{|v|>0\}|=|\{v\neq0\}|<\delta$ and drawing on Lemma \ref{lem:control-spt}
  we can modify the sequence $(\eta_k)_{k\in\N}$ such that we additionally have
  $|{\spt\eta_k}|<\delta$ for all $k\in\N$. Moreover, possibly replacing
  $(\eta_k)_{k\in\N}$ by a subsequence, we infer from Lemma
  \ref{lem:str-implies-Hae} that $\eta_k$ converge to $|v|^\ast=|v^\ast|$
  also $\H^{n-1}$-\ae{} on $\Omega$, and by the preceding this convergence holds
  $\mu$-\ae{} on $\Omega$ as well. Hence, via Fatou's lemma and the estimate in
  \eqref{item:est-smooth-grad} we find
  \[
    \int_\Omega|v^\ast|\,\d\mu
    \le\liminf_{k\to\infty}\int_\Omega\eta_k\,\d\mu
    \le\liminf_{k\to\infty}\bigg[C\int_\Omega|\nabla\eta_k|\dx+\omega(|{\spt\eta_k}|)\bigg]
    \le C\int_\Omega|\nabla v|\dx+\eps\,.
  \]
  This completes the derivation of \eqref{item:eps-delta-est-W11-grad}.

  \smallskip

  We turn to the implication from \eqref{item:eps-delta-est-W11-grad} back to
  \eqref{item:mu(A^+)<P+eps}. We consider $\eps>0$, the corresponding $\delta$
  from \eqref{item:eps-delta-est-W11-grad}, and a set $A\in\BVs(\R^n)$ with
  $A\Subset\Omega$ and $|A|<\delta$. Then, by Lemma
  \ref{lem:strict-approx-above} applied with $u=\1_A$, we can find
  $v_k\in\W_0^{1,1}(\Omega)$ with $\1_A\le v_k\le1$ \ae{} on $\Omega$ for all
  $k\in\N$ such that $v_k$ converge strictly in $\BV(\Omega)$ to $\1_A$. Observing
  $|\{\1_A>0\}|=|A|<\delta$, we next apply Lemma \ref{lem:control-spt} with
  $u=\1_A$ to modify the sequence and achieve additionally $|\{v_k>0\}|<\delta$
  for all $k\in\N$. Taking into account that $\eta_k^\ast\ge(\1_A)^+=\1_{A^+}$
  holds $\H^{n-1}$-\ae{}, we deduce
  \[
    \mu(A^+)
    \le\liminf_{k\to\infty}\int_\Omega\eta_k^\ast\,\d\mu
    \le\lim_{k\to\infty}\bigg[C\int_\Omega|\nabla\eta_k|\dx+\eps\sup_\Omega|\eta_k|\bigg]
    \le C\P(A)+\eps\,.
  \]
  By Lemma \ref{lem:closed-vs-compact} this suffices to ensure the small-volume
  IC in $\Omega$ with constant $C$

  \smallskip
  
  Finally, we prove that \eqref{item:sub-C-div-plus-L1} implies
  \eqref{item:omega-est-smooth-grad}. Given $\sigma$ and $H$ as in
  \eqref{item:sub-C-div-plus-L1}, by absolute continuity of the integral,
  there exists $\omega\colon{[0,\infty]}\to{[0,\infty]}$ with
  $\lim_{t\searrow0}\omega(t)=\omega(0)=0$ such that
  $\int_A|H|\dx\le\omega(|A|)$ holds for all $A\in\Bo(\Omega)$. Using
  this together with the definition of the distributional divergence, we
  estimate
  \[
    \int_\Omega\eta\,\d\mu={-}\int_\Omega\sigma\cdot\nabla\eta\dx+\int_\Omega\eta H\dx
    \le\int_\Omega|\sigma|\,|\nabla\eta|\dx+\int_{\spt\eta}|H|\dx
    \le C\int_\Omega|\nabla\eta|\dx+\omega(|{\spt\eta}|)
  \]
  for every $\eta\in\C^\infty_\cpt(\Omega)$ with $0\le\eta\le1$ on $\Omega$.
\end{proof}

Theorem \ref{thm:char-strong-ic} is in most regards a special case of Theorem
\ref{thm:char-small-vol-ic}, the only true addition being the fact that we can also get
back from \eqref{item:mu(A^+)<P}, \eqref{item:mu(A^1)<P},
\eqref{item:est-smooth-grad}, \eqref{item:est-W11-grad} to
\eqref{item:sub-C-div}. Consequently, we can keep the proof comparably brief:

\begin{proof}[Proof of Theorem \ref{thm:char-strong-ic}]
  The implications \eqref{item:mu(A^+)<P}$\implies$\eqref{item:mu(A^1)<P},
  \eqref{item:mu(A^1)<P}$\implies$\eqref{item:est-smooth-grad},
  \eqref{item:est-smooth-grad}$\implies$\eqref{item:est-W11-grad},
  \eqref{item:est-W11-grad}$\implies$\eqref{item:mu(A^+)<P},
  \eqref{item:sub-C-div}$\implies$\eqref{item:est-smooth-grad}    
  in Theorem \ref{thm:char-strong-ic} can be proved along the lines of
  the corresponding implications in Theorem \ref{thm:char-small-vol-ic}.
  In fact, one can drop from the reasoning all arguments and terms with $\eps$,
  $\omega$, $H$ as well as the requirements $\eta\le1$, $v\in\L^\infty(\Omega)$,
  while at the same time weakening all $\delta$-smallness conditions to merely
  finiteness conditions. This leads to some simplifications, for instance, Lemma
  \ref{lem:control-spt} is no longer needed. However, we refrain from discussing
  any further details in this regard.

  \smallskip

  Rather to conclude the proof we address the implication
  \eqref{item:est-W11-grad}$\implies$\eqref{item:sub-C-div}, which
  follows from (a homogeneous version of) the duality
  $(\W_0^{1,1})^\ast=\W^{-1,\infty}$ and, in more concrete terms, from the
  following reasoning. Consider the closed subspace
  $X\coleq\{\nabla\eta\,:\,\eta\in\W_0^{1,1}(\Omega)\}$ of
  $\L^1(\Omega,\R^n)$ with the $\L^1$-norm. Then the assumption
  \eqref{item:est-W11-grad} gives that the linear functional
  $\nabla\eta\mapsto\int_\Omega\eta^\ast\,\d\mu$ is an element of norm $\le C$
  in the dual $X^\ast$. By the Hahn-Banach theorem, this functional extends
  to an element of norm $\le C$ in $\L^1(\Omega,\R^n)^\ast$, and by the Riesz
  duality $(\L^1)^\ast=\L^\infty$ there exists some
  $\sigma\in\L^\infty(\Omega,\R^n)$ with
  $\|\sigma\|_{\L^\infty(\Omega,\R^n)}\le C$ such that
  \[
    \int_\Omega\eta^\ast\,\d\mu
    =-\int_\Omega\sigma\cdot\nabla\eta\dx
    \qq\qq\text{holds for all }\eta\in\W_0^{1,1}(\Omega)\,.
  \]
  Specifying this conclusion to $\eta\in\C^\infty_\cpt(\Omega)$, we obtain
  $\mu=\Div\sigma$ in the sense of distributions on $\Omega$.
\end{proof}

\section{\boldmath Isoperimetric conditions for perimeter measures and
  rectifiable measures}\label{sec:IC-for-P}

We begin this section by checking the validity of the strong IC in an
already-mentioned basic case, namely for the perimeter measure of a pseudoconvex
set. In view of the preceding results this can be implemented
conveniently by checking the variant of the IC with the representative $A^1$
instead of $A^+$.

\begin{prop}[strong IC for perimeters measures of pseudoconvex sets]
    \label{prop:strong-ic-for-pseudoconvex}
  For every pseudoconvex set $K\in\BVs(\R^n)$, the perimeter measure
  $\H^{n-1}\ecke\partial^\ast\!K$ satisfies the strong IC in $\R^n$ with
  constant $1$ and in case $|K|>0$ does not satisfy the strong IC in $\R^n$ with
  any smaller constant.
\end{prop}

\begin{proof}
  By Theorems \ref{thm:DeGiorgi} and \ref{thm:Federer} together with Lemma
  \ref{lem:intersect-with-pseudoconvex}, we infer
  \[
    (\H^{n-1}\ecke\partial^\ast\!K)(A^1)
    =\H^{n-1}(A^1\cap K^\frac12)
    \le\H^{n-1}((A\cap K)^\frac12)
    =\P(A\cap K)
    \le\P(A)
  \]
  By Theorem \ref{thm:char-strong-ic} this means that $\H^{n-1}\ecke\partial K$
  satisfies the strong IC in $\R^n$ with constant $1$. As moreover the equality
  $(\H^{n-1}\ecke\partial K)(K^+)=\P(K)$ occurs for the test set $K$ itself,
  the constant $1$ is optimal in case $|K|>0$ (in which we have $\P(K)>0$ as
  well).
\end{proof}

We stress that the pseudoconvexity assumption in Proposition
\ref{prop:strong-ic-for-pseudoconvex} cannot be dropped, as already for $n=2$
and a bounded, smooth, open, but non-convex $K\subset\R^2$ one finds with
$(\H^1\ecke\partial K)(\mathrm{C}(K))=\P(K)>\P(\mathrm{C}(K))$ for the
closed convex hull $\mathrm{C}(K)$ of $K$ that the strong IC fails for
$\H^1\ecke\partial K$. In contrast to this, however, we show with the next (and
much more interesting) results that the small-volume IC is independent of
geometric properties such as convexity of an underlying set and indeed admits a
much wider class of admissible measures.

\begin{thm}[small-volume IC for general perimeter measures]\label{thm:2P-admis}
  For every $E\in\M(\R^n)$ with $\P(E)<\infty$, the double perimeter measure
  \[
    \mu\coleq2\P(E,\,\cdot\,)=2|\D\1_E|=2\H^{n-1}\ecke\partial^\ast\!E
  \]
  can be expressed in the form $\mu=H\Ln+\Div\sigma$ in $\mathscr{D}'(\R^n)$
  with a sub-unit $\L^\infty$ vector field $\sigma$ on $\R^n$ and a function
  $H\in\L^1(\R^n)$. Consequently, $\mu$ satisfies all properties in
  Theorem \ref{thm:char-small-vol-ic} on $\Omega=\R^n$ and in particular satisfies
  the small-volume IC in $\R^n$ with constant $1$, that is, for every $\eps>0$,
  there is some $\delta>0$ such that
  \begin{equation}\label{eq:2P-admis}
    2\H^{n-1}(A^+\cap\partial^\ast\!E)\le\P(A)+\eps
    \qq\text{for all }A\in\M(\R^n)\text{ with }|A|<\delta\,.
  \end{equation}
\end{thm}

We would like to highlight that the small-volume IC reached in the theorem
trivially carries over to $\mu=2\H^{n-1}\ecke S$ with any subset
$S\in\Bo(\partial^\ast\!E)$ and even more generally to
$\mu=\alpha\H^{n-1}\ecke\partial^\ast\!E$ with any ${[0,2]}$-valued Borel
density $\alpha\colon\partial^\ast\!E\to[0,2]$ on $\partial^\ast\!E$. Thus, we
have identified a reasonably broad class of ($n{-}1$)-dimensional measures for
which the central assumption of our semicontinuity and existence results
holds. Beyond that a further broadening of the class will be achieved in
Corollary \ref{cor:rect-admis}, and the optimality of the upper bound
$2$ for the density $\alpha$ will be established in Proposition
\ref{prop:bound-2-optimal}.

\begin{proof}
  In the case $n=1$, the boundary $\partial^\ast\!E$ consists of
  finitely many points. Then, for $\mu=2\H^0\ecke\partial^\ast\!E$, the claim
  $\mu=H\mathcal{L}^1+\sigma'$ follows trivially by taking any
  sub-unit $\sigma\in\BV(\R)$ which is smooth on $(\partial^\ast\!E)^\c$
  and jumps from ${-}1$ to $1$ at each point of $\partial^\ast\!E$ so that
  $\sigma'={-}H\mathcal{L}^1+2\H^0\ecke\partial^\ast\!E$ with
  $H\in\L^1(\R)$. (In fact, if $\partial^\ast\!E\subset{(a,b)}$ for a bounded
  interval ${(a,b)}$, one may take $\sigma$ linear on each component of
  ${(a,b)}\setminus\partial^\ast\!E$ and $\sigma\equiv0$ on
  ${(a,b)}^\c$.)

  In the case $n\ge2$, from Theorem \ref{thm:isoperi} we get
  $E\in\BVs(\R^n)$ or $E^\c\in\BVs(\R^n)$, where in view of
  $\P(E^c,\,\cdot\,)=\P(E,\,\cdot\,)$ and
  $\partial^\ast\!E^\c=\partial^\ast\!E$ it suffices to treat the case
  $E\in\BVs(\R^n)$. By results of Barozzi \& Gonzalez \& Tamanini
  \cite{BarGonTam87} and Barozzi \cite{Barozzi94} (see specifically
  \cite[Remark 2.1, Theorem 2.1]{Barozzi94} or alternatively
  \cite[Section 2]{GonMas94}), there exists an optimal $\L^1$ variational
  mean curvature $H_E$ of $E$, that is, a function $H_E\in\L^1(\R^n)$ with
  $\int_EH_E\dx=\P(E)={-}\int_{E^\c}H_E\dx$ and thus $\int_{\R^n}H_E\dx=0$ such
  that
  \[
    \P(E)-\int_EH_E\dx\le\P(F)-\int_FH_E\dx
    \qq\text{for all }F\in\M(\R^n)\text{ with }\P(F)<\infty\,.
  \]
  We apply this to $F$ and $F^\c$ and exploit $\P(F^\c)=\P(F)$ and
  $\int_{F^\c}H_E\dx={-}\int_FH_E\dx$ to deduce
  \[
    \bigg|\int_FH_E\dx\bigg|\le\P(F)
    \qq\text{for all }F\in\M(\R^n)\text{ with }\P(F)<\infty\,.
  \]
  This estimate can be read as a strong IC for $H_E\Ln$, but at this point is not
  perfectly in line with the previous considerations in this paper, which
  would rather require separate conditions on $(H_E)_+\Ln$ and $(H_E)_-\Ln$.
  Nonetheless, most of the arguments used for Theorems \ref{thm:char-strong-ic}
  and \ref{thm:char-small-vol-ic} still apply, and we now give a brief rereading
  in the present situation in order to eventually reach a divergence structure
  $H_E=\Div\sigma_E$. Indeed, for $\eta\in\C^\infty_\cpt(\R^n)$, with the help
  of a layer-cake formula and the coarea formula of Theorem \ref{thm:BV-coarea}
  we find $\P(\{\eta>t\})<\infty$ for \ae{} $t\in\R$ and
  \[
    \bigg|\int_{\R^n}\eta H_E\dx\bigg|
    =\bigg|\int_\R\int_{\{\eta>t\}}H_E\dx\,\d t\bigg|
    \le\int_\R\bigg|\int_{\{\eta>t\}}H_E\dx\bigg|\,\d t
    \le\int_\R\P(\{\eta>t\})\,\d t
    =\int_{\R^n}|\nabla\eta|\dx\,.
  \]
  Consequently, if we consider the subspace
  $X\coleq\{\nabla\eta\,:\,\eta\in\C^\infty_\cpt(\R^n)\}$ of $\L^1(\R^n,\R^n)$
  with the $\L^1$-norm, the functional $\nabla\eta\mapsto\int_{\R^n}\eta H_E\dx$
  is a sub-unit element in $X^\ast$ and extends to a sub-unit element in
  $\L^1(\R^n,\R^n)^\ast$ by virtue of the Hahn-Banach theorem. The duality
  $(\L^1)^\ast=\L^\infty$ then yields some $\sigma_E\in\L^\infty(\R^n,\R^n)$ with
  $\|\sigma_E\|_{\L^\infty(\R^n,\R^n)}\le1$ such that
  $\int_{\R^n}\eta H_E\dx=-\int_{\R^n}\sigma_E\cdot\nabla\eta\dx$ holds for all
  $\eta\in\C^\infty_\cpt(\R^n)$, in other words, it gives a sub-unit $\L^\infty$
  vector field $\sigma_E$ on $\R^n$ with
  \[
    \Div\sigma_E=H_E
    \qq\qq\text{in the sense of distributions on }\R^n\,.
  \]
  Exploiting $E\in\BVs(\R^n)$ and the Gauss-Green formula
  \eqref{eq:int-Gauss-Rn} we then infer
  \[
    \H^{n-1}(\partial^\ast\!E)
    =\P(E)
    =\int_EH_E\dx
    =\int_E\Div\sigma_E\dx
    =\int_{\partial^\ast\!E}\sigma_E\cd\nu_E\,\d\H^{n-1}
  \]
  for the generalized normal trace $\sigma_E\cd\nu_E$ introduced in Definition
  \ref{def:gen-traces}. This improves the $\H^{n-1}$-\ae{} \emph{in}equality
  $|\sigma_E\cd\nu_E|\le1$ on $\partial^\ast\!E$ to the $\H^{n-1}$-\ae{} equality
  \[
    \sigma_E\cd\nu_E=1
    \qq\qq\text{on }\partial^\ast\!E\,.
  \]
  \begin{minipage}{8.5cm}
    We next introduce the modifications
    \[
      \sigma\coleq\begin{cases}{-}\sigma_E&\text{on }E\\\sigma_E&\text{on }E^\c\end{cases}
    \]
      and
    \[
      H\coleq\begin{cases}H_E&\text{on }E\\{-}H_E&\text{on }E^\c\end{cases}
    \]
    of $\sigma_E$ and $H_E$ and record that $\sigma$ and $H$ are still a
    sub-unit $\L^\infty$ vector field and an $\L^1$ function on $\R^n$. Then,
    for\hfill arbitrary\hfill $\p\in\C^\infty_\cpt(\R^n)$,\hfill the\hfill
    Gauss-Green\hfill formulas
  \end{minipage}
  \hfill
  \begin{minipage}{6.8cm}\vspace{-5ex}
  \begin{figure}[H]\centering
    \includegraphics{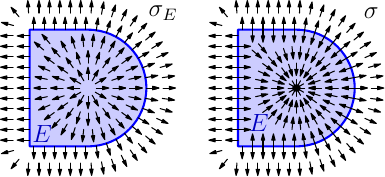}
    \caption{An illustration of $\sigma_E$ and $\sigma$, which differ by
      reversing the arrows inside $E$.}
  \end{figure}
  \end{minipage}\\
  \eqref{eq:int-Gauss}, \eqref{eq:ext-Gauss} (here used for $\sigma_E$ with
  $\Div\sigma_E=H_E\in\L^1(\R^n)$ on $\Omega=\R^n$) yield
  \[\begin{aligned}
    \int_{\R^n}\sigma\cd\nabla\p\dx
    &={-}\int_E\sigma_E\cd\nabla\p\dx+\int_{E^\mathrm{c}}\sigma_E\cd\nabla\p\dx\\
    &=\int_E\p(\Div\sigma_E)\dx-\int_{E^\mathrm{c}}\p(\Div\sigma_E)\dx-2\int_{\partial^\ast\!E}\p\,\nortrace\,\d\H^{n-1}\\
    &=\int_E\p H_E\dx-\int_{E^\mathrm{c}}\p H_E\dx-2\int_{\partial^\ast\!E}\p\,\d\H^{n-1}\\
    &=\int_{\R^n}\p\,\d(H\Ln-2\H^{n-1}\ecke\partial^\ast\!E)\,.
  \end{aligned}\]
  In conclusion we have  
  \[
    {-}\Div\sigma=H\Ln-2\H^{n-1}\ecke\partial^\ast\!E
    \qq\qq\text{in the sense of distributions on }\R^n
  \]
  or in other words $\mu=H\Ln+\Div\sigma$ in the sense of distributions on
  $\R^n$. Thus, all the claims follow directly from Theorem
  \ref{thm:char-small-vol-ic}.
\end{proof}

\begin{rem}[on infinite perimeter measures]\label{rem:infinite-P}
  If $E\in\BVs_\loc(\R^n)\setminus\BVs(\R^n)$ has only locally finite, but not
  finite perimeter, the following examples show that $2\P(E,\,\cdot\,)$
  \emph{may or may not} satisfy the small-volume IC with constant $1$.
  \begin{enumerate}[{\rm(i)}]
  \item On one hand, if $E$ is a half-space or the infinite strip between
    two parallel hyperplanes, for instance, then $2\P(E,\,\cdot\,)$ satisfies
    the small-volume IC with constant $1$\textup{;} see Proposition
    \ref{prop:inf-model-meas}.
  \item On the other hand, if we consider $n=1$ and the union of intervals
    $E_\ell\coleq\bigcup_{k=2\ell}^\infty\bigcup_{i=1}^\ell{(k{+}\frac{2i-1}k,k{+}\frac{2i}k)}$,
    with arbitrary fixed $\ell\in\N$, then $\P(E_\ell,\,\cdot\,)$
    consists of groups of $2\ell$ Dirac measures concentrated on shorter and
    shorter intervals, and thus $2\P(E_\ell,\,\cdot\,)$ satisfies the
    small-volume IC with constant $\frac2\ell$ at most \ka
    but no larger constant\kz. This example can be adapted to higher
    dimensions either simply by taking
    $E_\ell{\times}{(0,1)^{n-1}}\subset\R^n$ or by considering
    $\bigcup_{i=1}^\ell\{(x',x_n)\in\R^{n-1}{\times}\R\,:\,f_{2i-1}(x')<x_n<f_{2i}(x')\}$,
    where $f_1<f_2<\ldots<f_{2\ell}$ are smooth functions
    $\R^{n-1}\to\R$ with $\lim_{|x'|\to\infty}f_j(x')=0$.\label{item:infinite-P-without IC}
  \end{enumerate}
\end{rem}

Next, as announced, we address a further extension of Theorem
\ref{thm:2P-admis}:

\begin{cor}[small-volume IC for rectifiable $\H^{n-1}$-measures]\label{cor:rect-admis}
  If $S\in\Bo(\R^n)$ is $\H^{n-1}$-finite and countably $\H^{n-1}$-rectifiable
  \ka in the sense that\/ $\H^{n-1}(S)<\infty$ and\/
  $\H^{n-1}(S\setminus\bigcup_{j=1}^\infty f_j(\R^{n-1}))=0$ for Lipschitz
  mappings $f_j\colon\R^{n-1}\to\R^n$\kz, then the measure $2\H^{n-1}\ecke S$
  satisfies the small-volume IC in $\R^n$ with constant\/ $1$.
\end{cor}

\begin{proof}
  It follows from \cite[Proposition 2.76]{AmbFusPal00} that we have
  $\H^{n-1}(S\setminus\bigcup_{j=1}^\infty K_j)=0$ for countably many
  compact subsets $K_j\subset\Gamma_j$ of Lipschitz-($n{-}1$)-graphs $\Gamma_j$
  in the sense of \cite[Example 2.58]{AmbFusPal00}. Clearly, we have
  $K_j\subset\partial^\ast\!E_j$ for some $E_j\in\BVs(\R^n)$ (which can
  be obtained by suitably cutting off the subgraphs of the Lipschitz functions,
  for instance). From Theorem \ref{thm:2P-admis} we have that
  $2\H^{n-1}\ecke K_j'$ with $K_j'\coleq K_j\setminus\bigcup_{i=1}^{j-1}K_i$ for
  $j\in\N$ satisfies the small-volume IC in $\R^n$ with constant $1$. In a next
  step we use Proposition \ref{prop:ic-sing} and the finiteness of these
  measures to conclude that
  $2\H^{n-1}\ecke\bigcup_{j=1}^kK_j=\sum_{j=1}^k2\H^{n-1}\ecke K_j'$ with
  $k\in\N$ satisfies this condition as well. Given an
  arbitrary $\eps>0$, in view of $\H^{n-1}(S)<\infty$ we can fix first
  $k\in\N$ with $\H^{n-1}(S\setminus\bigcup_{j=1}^kK_j)\le\frac\eps2$
  and then $\delta>0$ such that 
  $2\H^{n-1}(A^+\cap\bigcup_{j=1}^kK_j)\le\P(A){+}\frac\eps2$ holds for all
  $A\in\M(\R^n)$ with $|A|<\delta$. By combination of these properties we obtain
  in fact $2\H^{n-1}(A^+\cap S)\le\P(A){+}\eps$, that is, the small-volume IC
  holds for $2\H^{n-1}\ecke S$ in $\R^n$ with constant $1$.
\end{proof}

Finally, we establish a converse to Theorem \ref{thm:2P-admis} and
Corollary \ref{cor:rect-admis}.

\begin{prop}[necessity of the upper density bound $2$ for the small-volume IC]
    \label{prop:bound-2-optimal}
  If $S\in\Bo(\R^n)$ is countably $\H^{n-1}$-rectifiable and
  $\alpha\H^{n-1}\ecke S$ with $\alpha\in\L^1_\loc(\R^n;\H^{n-1}\ecke S)$
  satisfies the small-volume IC with constant\/ $1$, then necessarily
  $\alpha\le2$ holds $\H^{n-1}$-\ae{} on $S$.
\end{prop}

\begin{proof}
  We assume, for a proof by contradiction, that $\alpha>2$ holds on a
  non-$\H^{n-1}$-negligible subset of $S$, and similar to the preceding proof
  we infer from \cite[Proposition 2.76]{AmbFusPal00} that
  $\H^{n-1}(S\setminus\bigcup_{j=1}^\infty\Gamma_j)=0$ holds for countably many
  Lipschitz-($n{-}1$)-graphs $\Gamma_j$ over hyperplanes $\pi_j$ in $\R^n$.
  Then, we can also find a compact subset $G$ of $S\cap\Gamma_{j_0}$, for some
  fixed $j_0\in\N$, with $\H^{n-1}(G)>0$ such that $\alpha\ge2{+}4\eps/\H^{n-1}(G)$
  holds $\H^{n-1}$-\ae{} on $G$ for some $\eps>0$. Since $G$ is compact, there
  exists an open neighborhood $U$ of $G$ in $\Gamma_{j_0}$ such that $U$ is a
  Lipschitz-($n{-}1$)-graph over an open $\BVs$ set in the hyperplane
  $\pi_{j_0}$ with $\H^{n-1}(U)<\H^{n-1}(G){+}\eps$. Next, for the $\eps>0$
  already fixed, we consider the corresponding $\delta>0$ from the
  IC, and we choose $\ell>0$ small enough that the ``width-$2\ell$ thickening''
  $A\coleq\bigcup_{t\in{({-}\ell,\ell)}}(U{+}t\nu_{j_0})\in\BVs(\R^n)$ of $U$ in
  the normal direction $\nu_{j_0}$ of $\pi_{j_0}$ satisfies $|A|<\delta$ and
  $\P(A)<2\H^{n-1}(U){+}\eps$. Then the previous estimates combine to
  $\P(A)<2\H^{n-1}(G){+}3\eps$, and in view of $G\subset S$ and
  $G\subset U\subset A^+$ we arrive at
  \[
    (\alpha\H^{n-1}\ecke S)(A^+)
    \ge(\alpha\H^{n-1})(G)
    \ge2\H^{n-1}(G){+}4\eps>\P(A){+}\eps\,.
  \]
  This, however, contradicts the assumed small-volume IC for
  $\alpha\H^{n-1}\ecke S$.
\end{proof}

\section{Lower semicontinuity on general domains}\label{sec:dom}

Once more we consider non-negative Radon measures $\mu_+$ and $\mu_-$ on $\R^n$
and define a functional of the previously considered type over arbitrary
$D\in\Bo(\R^n)$ by setting
\begin{equation}\label{eq:P-dom}
  \Pmu[A;D]\coleq\P(A,D)+\mu_+(A^1)-\mu_-(A^+)
\end{equation}
whenever for $A\in\M(\R^n)$ at least one of $\P(A,D){+}\mu_+(A^1)$ and
$\mu_-(A^+)$ is finite. Our aim in this section is to complement the
semicontinuity results of Section \ref{sec:Rn} for the full-space functional
$\Pmu=\Pmu[\,\cdot\,;\R^n]$ and the ones of Section \ref{sec:Dir} for
(generalized) Dirichlet classes with local semicontinuity results, which do not
involve boundary conditions and apply for $\Pmu[\,\cdot\,;D]$ with
$\mu_\pm\ecke D^\c\equiv0$ over arbitrary (measure-theoretically) open sets $D$.

In order to single out basic lines of our approach we point out directly that in
spite of requiring $\mu_\pm\ecke D^\c\equiv0$ we keep working with Radon measures
$\mu_\pm$ on all of $\R^n$ and impose ICs on these measures in all of $\R^n$
rather than using ICs in the sense of Definition \ref{def:ICs} on open domains
$D=\Omega$. In particular, our measures $\mu_\pm$ are
necessarily finite in cases with bounded $D$ (by definition of a Radon measure
on $\R^n$) and more generally whenever $\Cp_1(D)<\infty$ (by Proposition
\ref{prop:inf-P} and Lemma \ref{lem:finite}). One reason for proceeding in this
way is that the full-space viewpoint is convenient in order to apply the
previously achieved results and at least in case of \emph{finite}
measures $\mu_\pm$ on open $\Omega=D$ is not truly restrictive, as in fact the
small-volume ICs in $\Omega$ and in $\R^n$ are even equivalent by Lemma
\ref{lem:bdry-local-ic}. Moreover, for cases with \emph{infinite} measures
$\mu_-$ concentrated on unbounded domains $D$ with $\Cp_1(D)=\infty$ the
following example suggests that working with ICs in all of $\R^n$ is even more
appropriate for semicontinuity. Indeed, we consider for $n=1$ the
infinite union $\Omega\coleq\bigcup_{m=1}^\infty I_m$ of the intervals
$I_m\coleq{(m{-}2^{-m},m{+}2^{-m})}$ and, for arbitrarily small
$\theta\in{(0,\infty)}$, the infinite Radon measure
$\mu_-=\theta\sum_{m=1}^\infty\delta_m=\theta\H^0\ecke\N$ supported in
$\Omega$. Then $\mu_-$ satisfies the strong IC even with (small) constant
$\theta/2$ in $\Omega$, but lower semicontinuity of
$\mathscr{P}_{0,\mu_-}[\,\cdot\,;\Omega]$ fails, since $\bigcup_{m=k}^\infty
I_m$ converge globally in measure to $\emptyset$ with
$\mathscr{P}_{0,\mu_-}\big[\bigcup_{m=k}^\infty I_m;\Omega\big]={-}\infty$ for
all $k\in\N$ and $\mathscr{P}_{0,\mu_-}[\emptyset;\Omega]=0$. In fact, in the
light of Theorem \ref{thm:lsc-BV-dom}\eqref{item:lsc-BV-dom-c} below this
failure of semicontinuity is possible only since
$\mu_-{+}\H^0\ecke\partial\Omega$ satisfies the small-volume IC \emph{in all
of\/ $\R$} at best with constant $1{+}\theta/2$, but not with the required
constant $1$. We remark that similar configurations can be arranged with
absolutely continuous measures (by ``spreading out'' the Dirac measures a bit)
and in arbitrary dimension $n\in\N$ (e.\@g.\@ by placing measures in thin annuli
instead of short intervals). Thus, as foreshadowed above, semicontinuity does
not follow from an IC in open $D=\Omega$ in the sense of Definition
\ref{def:ICs}, but rather from certain $D$-dependent ICs in full $\R^n$.
In fact, these ICs can be read, if not as ICs in $D$, then still as ICs
\emph{relative to $D$} with the relative perimeter occurring in essentially the
same way as in the condition of Lemma
\ref{lem:bdry-local-ic}\eqref{item:mod-ic-to-bdry}.

\smallskip

Before reaching semicontinuity on arbitrary open sets $D=\Omega$ in the later
Theorem \ref{thm:lsc-dom}, we first provide a semicontinuity statement, which
applies on the measure-theoretic interior $D=\Omega^1$ of a set $\Omega$ of
locally finite perimeter and in fact seems illustrative and interesting in its
own right. We remark that at this point we apply the notions of local and global
convergence in measure from \eqref{eq:conv-in-meas} and
\eqref{eq:local-conv-in-meas} on the possibly non-open set $\Omega^1$.

\begin{thm}[lower semicontinuity on a domain of locally finite perimeter]
    \label{thm:lsc-BV-dom}
  Consider a set\/ $\Omega\in\M(\R^n)$, a set $A_\infty\in\M(\R^n)$, a sequence
  $(A_k)_{k\in\N}$ in $\M(\R^n)$, and non-negative Radon measures $\mu_+$ and\/
  $\mu_-$ on $\R^n$ with $\mu_\pm\ecke(\Omega^1)^\c\equiv0$ such that one of the
  following sets of assumptions is valid\textup{:}  
  \begin{enumerate}[{\rm(a)}]
  \item We have $\Omega\in\BVs_\loc(\R^n)$, the measure $\mu_-$ is
    \emph{finite}, the measures $\mu_+$ and $\mu_-$ both satisfy the
    small-volume IC in $\R^n$ with constant\/ $1$, and $A_k$ converge to $A_\infty$
    \emph{locally} in measure on $\Omega$.\label{item:lsc-BV-dom-a}
  \item We have $\Omega\in\BVs_\loc(\R^n)$, the measures $\mu_+$ and
    $\mu_-{+}\P(\Omega,\,\cdot\,)$ both satisfy the small-volume IC in $\R^n$
    with constant\/ $1$, the measure $\mu_-{+}\P(\Omega,\,\cdot\,)$ additionally
    satisfies the \emph{almost-strong} IC from \eqref{eq:strong-ic-at-infty}
    with constant $1$ near $\infty$, and $A_k$ converge to $A_\infty$ \emph{locally} in
    measure on $\Omega$ with
    $|(A_k\Delta A_\infty)\cap\Omega|{+}\P(A_k\cap\Omega){+}\P(A_\infty\cap\Omega)<\infty$ for
    all $k\in\N$.\label{item:lsc-BV-dom-b}
  \item We have $\Omega\in\BVs_\loc(\R^n)$, the measures $\mu_+$ and
    $\mu_-{+}\P(\Omega,\,\cdot\,)$ both satisfy the small-volume IC in $\R^n$ with
    constant\/ $1$, and $A_k$ converge to $A_\infty$ \emph{globally} in measure on
    $\Omega$ with $\P(A_k\cap\Omega){+}\P(A_\infty\cap\Omega)<\infty$ for all
    $k\in\N$.\label{item:lsc-BV-dom-c}
  \end{enumerate}
  If furthermore $\min\{\mu_+(A_k^1),\mu_-(A_k^+)\}<\infty$ holds for all
  $k\in\N$, then we have $\min\{\mu_+(A_\infty^1),\mu_-(A_\infty^+)\}<\infty$ and
  \begin{equation}\label{eq:lsc-BV-dom}
    \liminf_{k\to\infty}\Pmu[A_k;\Omega^1]\ge\Pmu[A_\infty;\Omega^1]\,.
  \end{equation}
\end{thm}

Since (all representatives of) a set $\Omega\in\BVs_\loc(\R^n)$ with $|\Omega|>0$
may have empty interior, the previous statement differs from the more usual
semicontinuity on open sets, and indeed semicontinuity on $D=\Omega^1$ does not
to seem to be well known even in case $\mu_\pm\equiv0$, that is, for the
perimeter itself. Therefore, we explicitly record as a subcase of Theorem
\ref{thm:lsc-BV-dom}:

\begin{cor}[lower semicontinuity of the perimeter on a measure-theoretic
  interior]\label{cor:lsc-per-BV-dom}
  Consider a set $\Omega\in\BVs_\loc(\R^n)$. If a sequence $(A_k)_{k\in\N}$ in
  $\M(\R^n)$ converges to $A_\infty\in\M(\R^n)$ locally in measure on $\Omega$,
  then we have
  \[
    \liminf_{k\to\infty}\P(A_k,\Omega^1)\ge\P(A_\infty,\Omega^1)\,.
  \]
\end{cor}

Interestingly, when specializing the subsequent proof of Theorem
\ref{thm:lsc-BV-dom}\eqref{item:lsc-BV-dom-a} to the case $\mu_\pm\equiv0$ of
the corollary, it turns out that even in this case the approach does rely on the
theory of the previous sections with $\mu_\pm\not\equiv0$ and indeed plugs in
the perimeter measure $\P(\Omega,\,\cdot\,)$ in place of either $\mu_+$ or
$\mu_-$. Alternatively, however, Corollary \ref{cor:lsc-per-BV-dom} can be
derived as a special case of a recent result of Lahti \cite{Lahti20}. Indeed,
\cite[Theorem 4.5]{Lahti20} guarantees lower semicontinuity of the
perimeter even on every $\Cp_1$-quasi-open set in a general metric-space
setting, while it follows from \cite[Theorem 2.5]{CarDalLeaPas88} that
$\Omega^1$ is $\Cp_1$-quasi-open for every $\Omega\in\BVs_\loc(\R^n)$.

\medskip

Next, we provide a refined discussion of the different settings in Theorem
\ref{thm:lsc-BV-dom}, where once more the differences concern the handling of
the $\mu_-$-term only.

First of all we emphasize that the statement under assumptions
\eqref{item:lsc-BV-dom-a} with \emph{finite} $\mu_-$ should be considered as
the most basic, but also central point of the theorem and will be sufficient in
order to eventually move on to semicontinuity on arbitrary open sets. Exemplary
cases covered by \eqref{item:lsc-BV-dom-a} are finite perimeter
measures $\mu_-=2\H^{n-1}\ecke\partial^\ast\!E$ of $E\in\BVs(\R^n)$
considered on any open $\Omega\in\BVs_\loc(\R^n)$ with
$\partial^\ast\!E\subset\Omega$, since for these Theorem \ref{thm:2P-admis}
gives the small-volume IC with constant $1$.

The settings \eqref{item:lsc-BV-dom-b} and \eqref{item:lsc-BV-dom-c} of Theorem
\ref{thm:lsc-BV-dom} improve on \eqref{item:lsc-BV-dom-a} in case of infinite
measures $\mu_-$, as seen similarly in Theorems \ref{thm:lsc-Rn} and
\ref{thm:lsc-Dir}. An exemplary case covered by \eqref{item:lsc-BV-dom-b}, but
not by \eqref{item:lsc-BV-dom-a} is
$\mu_-=2\H^{n-1}\ecke({(0,\infty)}{\times}\R^{n-2}{\times}\{0\})$ on
$\Omega={(0,\infty)}{\times}\R^{n-1}$ with $n\ge2$, for which
$\P(\Omega)=\infty$ holds, but still $\mu_-{+}\P(\Omega,\,\cdot\,)$ satisfies
even the strong IC on full $\R^n$ with constant $1$. While the exemplary cases
mentioned so far are covered also by the setting \eqref{item:lsc-BV-dom-c}, from
\eqref{item:lsc-BV-dom-c} we get the semicontinuity conclusion only along
sequences with \emph{global} convergence. Additional exemplary cases which are
covered by \eqref{item:lsc-BV-dom-c} only and come merely with
global-convergence semicontinuity are given by the infinite measures
$\mu_-=2\H^{n-1}\ecke(\R^{n-1}{\times}\{0,1\})$ on
$\Omega=\R^n$ and $\mu_-=2\H^{n-1}\ecke(\R^{n-1}{\times}\{1\})$ on
$\Omega=\R^{n-1}{\times}{(0,\infty)}$. In both these cases, Proposition
\ref{prop:inf-model-meas} implies the small-volume IC with constant $1$ for
$\mu_-{+}\P(\Omega,\,\cdot\,)$, but this measure does not satisfy the
almost-strong IC required in \eqref{item:lsc-BV-dom-b}.

\smallskip

We add one specific remark on the assumptions of the theorem:

\begin{rem}[on the finite-perimeter assumptions in Theorem \ref{thm:lsc-BV-dom}]
  The assumption $\P(A_k\cap\Omega)<\infty$, which occurs in parts
  \eqref{item:lsc-BV-dom-b} and \eqref{item:lsc-BV-dom-c} of Theorem
  \ref{thm:lsc-BV-dom}, follows from the more local and thus slightly more
  natural assumption $\P(A_k,\Omega^1)<\infty$ together with\/
  $\P(\Omega)<\infty$. Clearly, $\P(A_\infty\cap\Omega)<\infty$ follows from
  $\P(A_\infty,\Omega^1)<\infty$ together with $\P(\Omega)<\infty$ in the same way.
    
  \begin{proof}
    By distinguishing between points inside $\Omega^1$ and outside $\Omega^1$ it
    is not difficult to verify the inclusion
    $\partial^\mathrm{e}(A_k\cap\Omega)
    \subset(\partial^\mathrm{e}\!A_k\cap\Omega^1)\cup\partial^\mathrm{e}\Omega$.
    By Theorems \ref{thm:DeGiorgi} and \ref{thm:Federer} we infer
    $\H^{n-1}(\partial^\mathrm{e}\;\!\!(A_k\cap\Omega))\le\P(A_k,\Omega^1)+\P(\Omega)<\infty$,
    and then Federer's criterion (see \cite[Theorem 5.23]{EvaGar15}, for
    instance) yields $\P(A_k\cap\Omega)<\infty$.
  \end{proof}
\end{rem}

Now we turn to the proof of the theorem, where the essential strategy is to
apply the full-space or Dirichlet results and to include in $\mu_-$ a boundary
term $\P(\Omega,\,\cdot\,)$, which eventually cancels out with the boundary
contribution $\P(\,\cdot\,,\partial^\ast\Omega)$ of the perimeter.

\begin{proof}[Proof of Theorem \ref{thm:lsc-BV-dom}]
  In a first step we establish the result for the setting
  \eqref{item:lsc-BV-dom-a} with additional requirement $\P(\Omega)<\infty$
  and for the settings \eqref{item:lsc-BV-dom-b},
  \eqref{item:lsc-BV-dom-c}. We introduce
  \[
    S_k\coleq A_k\cap\Omega\,,\qq\qq
    S_\infty\coleq A_\infty\cap\Omega\,,\qq\qq
    \mu_-^\Omega\coleq\mu_-+\P(\Omega,\,\cdot\,)\,,
  \]
  and observe that the present assumptions imply the ones of the corresponding
  setting in Theorem \ref{thm:lsc-Rn} or its extension due to Remark
  \ref{rem:lsc-Rn-infinite-vol} with $S_k$, $S_\infty$, $\mu_+$, $\mu_-^\Omega$ in
  place of $A_k$, $A_\infty$, $\mu_+$, $\mu_-$. (As an alternative, we could also take
  into account $S_k\setminus\Omega=\emptyset=S_\infty\setminus\Omega$ and use Theorem
  \ref{thm:lsc-Dir} as our reference here.) However, while in assumptions
  \eqref{item:lsc-BV-dom-b} and \eqref{item:lsc-BV-dom-c} the relevant IC on
  $\mu_-^\Omega$ is explicitly included, under \eqref{item:lsc-BV-dom-a}
  with additionally $\P(\Omega)<\infty$ it remains to justify that
  $\mu_-^\Omega$ satisfies the small-volume IC on $\R^n$ with constant $1$. To
  this end we first argue that in view of the requirement $\P(\Omega)<\infty$ in
  \eqref{item:lsc-BV-dom-a} the small-volume IC with constant $1$ holds for
  $\P(\Omega,\,\cdot\,)$ by Theorem \ref{thm:2P-admis} (where we have even discarded a
  factor $2$). Moreover, in view of $\mu_-\ecke(\Omega^1)^\c\equiv0$ and
  specifically $\mu_-\ecke\partial^\ast\Omega\equiv0$ the measures $\mu_-$ and
  $\P(\Omega,\,\cdot\,)=\H^{n-1}\ecke\partial^\ast\Omega$ are singular to each other
  and under the present assumptions are both finite. Thus, by Proposition
  \ref{prop:ic-sing} the small-volume IC with constant $1$ carries over from
  these two measures to their sum $\mu_-^\Omega$. After this justification we
  are in position to apply Theorem \ref{thm:lsc-Rn}, which yields
  \begin{equation}\label{eq:lsc-Rn-adapt}
    \liminf_{k\to\infty}\mathscr{P}_{\mu_+,\mu_-^\Omega}[S_k]
    \ge\mathscr{P}_{\mu_+,\mu_-^\Omega}[S_\infty]
  \end{equation}
  for the full-space functional defined in \eqref{eq:P}, but now with
  $\mu_-^\Omega$ in place of $\mu_-$. In order to rewrite the perimeter term in
  this functional we next deduce from the equality case of \eqref{eq:P(AcapB)}
  in Lemma \ref{lem:P(AcapB),P(A-S)} that we have
  \[
    \P(S_k)=\P(S_k,\Omega^1)+\P(\Omega,S_k^+)\,.
  \]
  We use this equality in conjunction with the definition of
  $\mu_-^\Omega$ and the observations
  $\P(A\cap\Omega,\Omega^1)=\P(A,\Omega^1)$ and
  $\mu_\pm\ecke(\Omega^1)^\c\equiv0$. Arguing in this way we end
  up with
  \[
    \mathscr{P}_{\mu_+,\mu_-^\Omega}[S_k]
    =\P(S_k)+\mu_+(S_k^1)-\mu_-^\Omega(S_k^+)
    =\P(S_k,\Omega^1)+\mu_+(S_k^1)-\mu_-(S_k^+)
    =\Pmu[S_k;\Omega^1]
    =\Pmu[A_k;\Omega^1]\,.
  \]
  Since we can analogously rewrite
  $\mathscr{P}_{\mu_+,\mu_-^\Omega}[S_\infty]=\Pmu[A_\infty;\Omega^1]$, the
  semicontinuity property obtained in \eqref{eq:lsc-Rn-adapt} directly
  transforms into the one claimed in \eqref{eq:lsc-BV-dom}.

  In a second step, it remains to remove in case of the setting
  \eqref{item:lsc-BV-dom-a} the additional assumption $\P(\Omega)<\infty$ which we
  have imposed so far. To this end we consider the general case of
  \eqref{item:lsc-BV-dom-a} with merely $\Omega\in\BVs_\loc(\R^n)$ and apply the
  result achieved on the cut-offs
  $\Omega_R\coleq\Omega\cap\B_R\in\BVs(\R^n)$ with $\mu_\pm\ecke\Omega_R^1$
  in place of $\mu_\pm$ to establish
  \[
    \liminf_{k\to\infty}\big[\P(A_k,\Omega_R^1)+\mu_+(A_k^1\cap\Omega_R^1)-\mu_-(A_k^+\cap\Omega_R^1)\big]
    \ge\P(A_\infty,\Omega_R^1)+\mu_+(A_\infty^1\cap\Omega_R^1)-\mu_-(A_\infty^+\cap\Omega_R^1)
  \]
  for every $R\in{(0,\infty)}$. Using $\Omega_R^1\subset\Omega^1$ and elementary
  estimations we deduce
  \[
    \liminf_{k\to\infty}\Pmu[A_k;\Omega^1]+\mu_-((\Omega_R^1)^\c)
    \ge\P(A_\infty,\Omega_R^1)+\mu_+(A_\infty^1\cap\Omega_R^1)-\mu_-(A_\infty^+)\,,
  \]
  from which we obtain the claim \eqref{eq:lsc-BV-dom} also in the general case
  of \eqref{item:lsc-BV-dom-a} by sending $R\to\infty$, by taking into account
  pointwise monotone convergence of $\Omega_R^1$ to $\Omega^1$ and the assumption
  $\mu_\pm\ecke(\Omega^1)^c\equiv0$, and finally by crucially exploiting the
  finiteness of $\mu_-$.
\end{proof}

Next, even though these are side issues, we add remarks on a
modified strategy for proving Theorem \ref{thm:lsc-BV-dom} and on a refined
version of the theorem, which gives the semicontinuity conclusion
\eqref{eq:lsc-BV-dom} for $\Pmu[\,\cdot\,;\Omega^1]$ even for some
measures $\mu_\pm$ which merely satisfy $\mu\ecke(\Omega^+)^\c\equiv0$ and thus
include boundary terms on $\partial^\ast\!\Omega$.

\begin{rem}[on a modified proof of Theorem \ref{thm:lsc-BV-dom} and a variant
    with boundary measures]
  \text{}
  \begin{enumerate}[{\rm(i)}]
  \item Imposing $\P(\Omega)<\infty$ as a decisive additional assumption, the
    conclusion of Theorem \ref{thm:lsc-BV-dom} can also be established by
    modified strategy. In case of the setting \eqref{item:lsc-BV-dom-a} this
    strategy bypasses Proposition \ref{prop:ic-sing}, and in case of the
    settings \eqref{item:lsc-BV-dom-b} and \eqref{item:lsc-BV-dom-c} it requires
    the ICs imposed on $\mu_-{+}\P(\Omega,\,\cdot\,)$ now merely for $\mu_-$
    itself. One may wonder whether the latter point partially
    improves on the statement of the theorem, but actually it does not, since in
    case $\P(\Omega)<\infty$ the relevant ICs for $\mu_-$ imply the ones for
    $\mu_-{+}\P(\Omega,\,\cdot\,)$ \ka possibly with increased $R_\eps$ and
    decreased $\delta$\kz\textup{;} compare with points
    \eqref{item:lsc-BV-dom-i} and \eqref{item:lsc-BV-dom-ii} of Remark
    \ref{rem:lsc-BV-dom} below. Nonetheless, we believe that the modified
    strategy is of some intrinsic interest, and thus we explicate it
    here.\label{item:lsc-BV-dom-alt-strat}

    \begin{proof}[Modified proof of Theorem \ref{thm:lsc-BV-dom} in case
        $\P(\Omega)<\infty$]
      We first record that $\P(\Omega)<\infty$ implies, by Theorem
      \ref{thm:2P-admis}, the small-volume IC with constant $1$ for the finite
      measure $\pi^\Omega\coleq\P(\Omega,\,\cdot\,)$. Arguing as in the
      preceding proof, but with application of Theorem \ref{thm:lsc-Rn} to
      $\mathscr{P}_{\mu_+,\pi^\Omega}$ (and thus no need for having or checking
      ICs for $\mu_-{+}\pi^\Omega$), we end up with
      \begin{equation}\label{eq:lsc-dom-plus}
        \liminf_{k\to\infty}\mathscr{P}_{\mu_+,0}[A_k;\Omega^1]
        \ge\mathscr{P}_{\mu_+,0}[A_\infty;\Omega^1]\,.
      \end{equation}
      We can now complement this with a similar, but \glqq dual\grqq{}
      reasoning. To this end we work with
      \[
        U_k\coleq A_k\cup\Omega^\c
        \qq\qq\text{and}\qq\qq
        U_\infty\coleq A_\infty\cup\Omega^\c
      \]
      (which under \eqref{item:lsc-BV-dom-b} or \eqref{item:lsc-BV-dom-c} with
      $\P(\Omega)<\infty$ are finite-perimeter sets) and deduce by an application
      of Theorem \ref{thm:lsc-Rn} to $\mathscr{P}_{\pi^\Omega,\mu_-}$ (still
      with $\pi^\Omega=\P(\Omega,\,\cdot\,)$) the semicontinuity property
      \begin{equation}\label{eq:lsc-Rn-adapt'}
        \liminf_{k\to\infty}\mathscr{P}_{\pi^\Omega,\mu_-}[U_k]
        \ge\mathscr{P}_{\pi^\Omega,\mu_-}[U_\infty]\,.
      \end{equation}
      Crucially exploiting $\P(\Omega)<\infty$ once more, we can rewrite
      $\P(U_k)=\P(U_k,\Omega^1)+\P(\Omega,(U_k^1)^\c)=\P(A_k,\Omega^1)+\P(\Omega)-\P(\Omega,U_k^1)$
      and consequently $\mathscr{P}_{\pi^\Omega,\mu_-}[U_k]=\mathscr{P}_{0,\mu_-}[A_k;\Omega^1]+\P(\Omega)$.
      With this and the analogous formula for $U_\infty$ and $A_\infty$ we go into
      \eqref{eq:lsc-Rn-adapt'} and, after canceling the $\P(\Omega)$-terms,
      then find
      \begin{equation}\label{eq:lsc-dom-minus}
        \liminf_{k\to\infty}\mathscr{P}_{0,\mu_-}[A_k;\Omega^1]\ge\mathscr{P}_{0,\mu_-}[A_\infty;\Omega^1]\,.    
      \end{equation}
      Since \eqref{eq:lsc-dom-plus} and \eqref{eq:lsc-dom-minus} apply also with
      $A_k\cap A_\infty$ and $A_k\cup A_\infty$, respectively, in place of $A_k$, we can
      combine these two semicontinuity properties by the strategy from the proof
      of Theorem \ref{thm:lsc-Rn}\eqref{item:lsc-Rn-c}. Thus, we indeed arrive
      at the full claim \eqref{eq:lsc-BV-dom} which includes both the $\mu_+$-
      and $\mu_-$-terms.
    \end{proof}
  \item\label{rem:lsc-BV-dom-bdry}
        If we add again $\P(\Omega)<\infty$ to the assumptions of Theorem
    \ref{thm:lsc-BV-dom} and require also those ICs imposed in the original
    statement on $\mu_\pm$ now even for
    $\mu_\pm{+}\P(\Omega,\,\cdot\,)$, then we can weaken the
    requirement $\mu_\pm\ecke(\Omega^1)^\c\equiv0$ from the original statement to merely
    $\mu_\pm\ecke(\Omega^+)^\c\equiv0$ and still obtain the semicontinuity
    conclusion for the up-to-the-boundary functional
    \[
      A\mapsto\P(A,\Omega^1)+\mu_+((A\cup\Omega^\c)^1)-\mu_-((A\cap\Omega)^+)\,.
    \]
    Here, in order to better classify the terms we record that
    \begin{align*}
      \mu_+((A\cup\Omega^\c)^1)&=\mu_+(A^1\cap\Omega^1)+\mu_+((A\cup\Omega^\c)^1\cap\partial^\ast\Omega)\,,\\
      \mu_-((A\cap\Omega)^+)&=\mu_-(A^+\cap\Omega^1)+\mu_-((A\cap\Omega)^+\cap\partial^\ast\Omega)
    \end{align*}
    split into an interior portion on $\Omega^1$ and a boundary portion on
    $\partial^\ast\Omega$, where the latter is evaluated via the interior traces
    $(A\cap\Omega)^+\cap\partial^\ast\Omega$ and
    $(A\cup\Omega^\c)^1\cap\partial^\ast\Omega$ of\/ $A$ on $\partial^\ast\Omega$
    and where these two traces actually coincide up to $\H^{n-1}$-negligible sets at
    least in case $\P(A,\partial^\ast\Omega)<\infty$ of finite perimeter up to
    $\partial^\ast\Omega$.

    The proof of the semicontinuity just claimed is still a variant of the
    preceding ones. Indeed, setting again $\pi^\Omega\coleq\P(\Omega,\,\cdot\,)$,
    we recall that in the original proof we applied Theorem \ref{thm:lsc-Rn}
    directly for $\mathscr{P}_{\mu_+,\mu_-+\pi^\Omega}[S_k]$, while in the
    variant of the preceding point \eqref{item:lsc-BV-dom-alt-strat} we applied
    it for $\mathscr{P}_{\mu_+,\pi^\Omega}[S_k]$ and\/ $\mathscr{P}_{\pi^\Omega,\mu_-}[U_k]$.
    We now follow closely the latter strategy, where the only essential
    modification is that in order to come out with non-trivial boundary terms
    we cannot anymore \glqq decouple\grqq{} $\mu_\pm$ and\/
    $\pi^\Omega=\P(\Omega,\,\cdot\,)$, but rather now apply Theorem \ref{thm:lsc-Rn}
    for $\mathscr{P}_{0,\mu_-{+}\pi^\Omega}[S_k]$ and
    $\mathscr{P}_{\mu_+{+}\pi^\Omega,0}[U_k]$.
  
    Among the assumptions mentioned above, we single out and discuss the
    case of the basic setting \eqref{item:lsc-BV-dom-a} with
    $\mu_\pm\ecke(\Omega^+)^\c\equiv0$ and the small-volume IC with constant\/ $1$
    for the finite measures
    $\mu_\pm{+}\P(\Omega,\,\cdot\,)=\mu_\pm{+}\H^{n-1}{\ecke}\partial^\ast\Omega$. In
    this case, once more by Proposition \ref{prop:ic-sing}, the IC splits into
    separate ICs for $\mu_\pm\ecke\Omega^1$ and
    $(\mu_\pm{+}\H^{n-1})\ecke\partial^\ast\Omega$, and then Theorem
    \ref{thm:2P-admis} identifies a wide class of admissible measures. Indeed,
    $\mu_\pm$ will be admissible if the interior portion $\mu_\pm\ecke\Omega^1$
    has the form $\alpha\H^{n-1}\ecke(\Omega^1\cap\partial^\ast\!E)$ with
    $E\in\M(\R^n)$, $\P(E)<\infty$ and weight function $\alpha$ bounded by $2$
    and if the boundary portion $\mu_\pm\ecke\partial^\ast\Omega$ has the form
    $\beta\H^{n-1}\ecke\partial^\ast\Omega$ with boundary weight function
    $\beta$ bounded by $1$ \ka so that the resulting weight for
    $(\mu_\pm{+}\H^{n-1})\ecke\partial^\ast\Omega$ is again bounded by $2$\kz.
    We actually consider this part of the outcome with the bound $2$ on $\Omega^1$
    and the bound\/ $1$ on $\partial^\ast\Omega$ as a natural and very plausible
    manifestation of the \glqq one-sided accessibility\grqq{} of
    $\partial^\ast\Omega$ only from inside $\Omega$.
  \end{enumerate}
\end{rem}

The next remark uncovers that the ICs for $\mu_-{+}\P(\Omega,\,\cdot\,)$ in
Theorem \ref{thm:lsc-BV-dom} may in fact be understood as a kind of
domain-adapted ICs. This also motivates the usage of very similar ICs in the
subsequent semicontinuity statement of Theorem \ref{thm:lsc-dom} on general
open sets.

\begin{rem}[on the interpretation of the ICs for $\mu_-{+}\P(\Omega,\,\cdot\,)$
    in Theorem \ref{thm:lsc-BV-dom}]\label{rem:lsc-BV-dom}
  Consider $\Omega\in\Bo(\R^n)$ and a Radon measure $\mu_-$ on $\R^n$.
  \begin{enumerate}[{\rm(i)}]
  \item If we assume $\Omega\in\BVs_\loc(\R^n)$ and
    $\mu_-\ecke(\Omega^1)^\c\equiv0$, then the almost-strong IC near $\infty$
    with constant\/ $1$ for $\mu_-{+}\P(\Omega,\,\cdot\,)$, as it occurs in
    \eqref{item:lsc-BV-dom-b}, implies that, for every $\eps>0$ with its
    corresponding $R_\eps$, we have
    \begin{equation}\label{eq:strong-ic-at-infty-rel-pre}
      \mu_-(E^+)\le\P(E,\Omega^1)+\eps
      \qq\text{for all }E\in\M(\R^n)\text{ with }|E\cap\B_{R_\eps}|=0\text{ and }|E|<\infty\,.
    \end{equation}
    This can be understood as version of the same type of\/ IC only for $\mu_-$
    but \emph{relative to the domain $\Omega^1$}.
    
    \begin{proof}
      It suffices to verify \eqref{eq:strong-ic-at-infty-rel-pre} for
      $E\in\M(\R^n)$ with $|E\cap\B_{R_\eps}|=0$ and
      $|E|{+}\P(E,\Omega^1)<\infty$. To this end, we consider
      $R\in{(R_\eps,\infty)}$, abbreviate $\Omega_R\coleq\Omega\cap\B_R$, use
      $\mu_-\ecke(\Omega^1)^\c\equiv0$, and test the IC with $E\cap\Omega_R$. In this way we find
      $\mu_-(E^+\cap\B_R)+\P(\Omega,(E\cap\Omega_R)^+)
      \le\mu_-((E\cap\Omega_R)^+)+\P(\Omega,(E\cap\Omega_R)^+)\le\P(E\cap\Omega_R)+\eps$.
      Next we derive a slightly sharpened variant of the equality case in
      \eqref{eq:P(AcapB)}. By distinguishing between points in $\Omega_R^1$ and
      $\partial^\mathrm{e}\Omega_R$ we verify      
      $\partial^\mathrm{e}(E\cap\Omega_R)=(\Omega_R^1\cap\partial^\mathrm{e}\!E)\dcup((E\cap\Omega_R)^+\cap\partial^\mathrm{e}\Omega_R)$,
      and then via Theorems \ref{thm:DeGiorgi}, \ref{thm:Federer}, and
      \eqref{eq:P(AcapB)} we arrive at
      $\P(E\cap\Omega_R)=\P(E,\Omega_R^1)+\P(\Omega_R,(E\cap\Omega_R)^+)
      \le\P(E,\Omega^1)+\P(\Omega,(E\cap\Omega_R)^+)+\H^{n-1}(E^+\cap\partial\B_R)$
      for $R\in{(0,\infty)}$. When combining this with the previous estimate,
      the terms $\P(\Omega,(E\cap\Omega_R)^+)$ cancel out, and we obtain
      $\mu_-(E^+\cap\B_R)\le\P(E,\Omega^1)+\H^{n-1}(E^+\cap\partial\B_R)+\eps$.
      Exploiting
      once more $|E^+|=|E|<\infty$ in a coarea argument, we have
      $\liminf_{R\to\infty}\H^{n-1}(E^+\cap\partial\B_R)=0$, and in the limit
      $R\to\infty$ we arrive at \eqref{eq:strong-ic-at-infty-rel-pre}. (In case of
      $\P(\Omega,(E\cap\Omega)^+)<\infty$ this argument also works
      more directly with $E\cap\Omega$ in place of $E\cap\Omega_R$. However,
      we cannot exclude $\P(\Omega,(E\cap\Omega)^+)=\infty$ in general.)
    \end{proof}

    Moreover, in case of $\P(\Omega)<\infty$ and with a possible increase of the
    radii $R_\eps$, we can also get back from
    \eqref{eq:strong-ic-at-infty-rel-pre} to the original almost-strong IC near
    $\infty$ for $\mu_-{+}\P(\Omega,\,\cdot\,)$. This simply works by trivially
    enlarging the right-hand side in \eqref{eq:strong-ic-at-infty-rel-pre} to
    $\P(E)+\eps$ and using $\lim_{R\to\infty}\P(\Omega,(\B_R)^\c)=0$ to estimate
    $\P(\Omega,\,\cdot\,)$ outside large balls by $\eps$. In case
    $\P(\Omega)=\infty$, however, this backwards implication is false even if,
    in addition to \eqref{eq:strong-ic-at-infty-rel-pre} for $\mu_-$, both
    $\mu_-$ and $\P(\Omega,\,\cdot\,)$ separately satisfy the strong IC with
    constant\/ $1$. This is demonstrated, for $n\ge2$, by
    $\mu_-=\H^{n-1}\ecke(\R^{n-1}{\times}\{-2,2\})$ on
    $\Omega=\R^{n-1}{\times}{[-1,1]}^\c$, which has the announced
    properties.\label{item:lsc-BV-dom-i}
  \item If we assume once more $\Omega\in\BVs_\loc(\R^n)$ and
    $\mu_-\ecke(\Omega^1)^\c\equiv0$, then the small-volume IC with constant\/
    $1$ for $\mu_-{+}\P(\Omega,\,\cdot\,)$, as it occurs in
    \textup{\eqref{item:lsc-BV-dom-c}}, implies, for every $\eps>0$, the
    existence of $\delta>0$ such that
    \begin{equation}\label{eq:small-vol-ic-rel-pre}
      \mu_-(E^+)\le\P(E,\Omega^1)+\eps
      \qq\text{for all }E\in\M(\R^n)\text{ with }|E|<\delta\,.
    \end{equation}
    This can be seen as a small-volume IC for $\mu_-$ \emph{relative to the
    domain $\Omega^1$}, and the implication can be proved by straightforward
    adaptation of the reasoning in the preceding point
    \eqref{item:lsc-BV-dom-i}. Moreover, if we assume
    $\Omega\in\BVs_\loc(\R^n)$, $\mu_-\ecke(\partial^\ast\Omega)^\c\equiv0$, the
    small-volume IC with constant\/ $1$ on $\R^n$ for $\P(\Omega,\,\cdot\,)$
    \ka as it is generally satisfied in case $\P(\Omega)<\infty$ by Theorem
    \ref{thm:2P-admis}\kz, and that either $\mu_-$ is finite or the supports of\/ $\mu_-$ and
    $\P(\Omega,\,\cdot\,)$ have positive distance, then we can also get back
    from \eqref{eq:small-vol-ic-rel-pre} to the small-volume IC for
    $\mu{+}\P(\Omega,\,\cdot\,)$ with constant\/ $1$ by using Proposition
    \ref{prop:ic-sing}. In connection with this last claim, it is easy to see
    that the assumptions
    $\Omega\in\BVs_\loc(\R^n)$, $\mu_-\ecke(\partial^\ast\Omega)^\c\equiv0$, and
    the small-volume IC for $\P(\Omega,\,\cdot\,)$ cannot be dropped. Moreover,
    the example given, for $n=2$, by $\mu_-=\H^1\ecke(\R{\times}\{0\})$ on
    $\Omega=\R^2\setminus\bigcup_{i=1}^\infty\big({[2i{-}1,2i]}{\times}{\big[\frac1{2i},\frac1i\big]}\big)$
    demonstrates that also the requirement of finiteness of $\mu_-$ or supports
    at positive distance is indeed necessary for the backwards implication \ka
    even if, as it happens here, both $\mu_-$ and $\P(\Omega,\,\cdot\,)$
    separately satisfy the strong IC with constant\/ $1$\kz.\label{item:lsc-BV-dom-ii}
  \end{enumerate}
\end{rem}

At this point we finally turn to the second main statement of this section,
which complements the previous result on the measure-theoretic interior of
$\BVs_{(\loc)}$ sets with a version on arbitrary open sets $D=\Omega$ in
$\R^n$. So, in comparison with Theorem \ref{thm:lsc-BV-dom} we drop any
regularity of the domain, but require openness in the stronger topological
sense.

\begin{thm}[lower semicontinuity on a general open set]
    \label{thm:lsc-dom}
  Consider an open set\/ $\Omega$ in $\R^n$, a set $A_\infty\in\M(\R^n)$, a sequence
  $(A_k)_{k\in\N}$ in $\M(\R^n)$. For non-negative Radon measures $\mu_+$ and\/
  $\mu_-$ on $\R^n$ with $\mu_\pm\ecke\Omega^\c\equiv0$, assume that both
  $\mu_+$ and $\mu_-$ satisfy the small-volume IC in $\R^n$ with constant\/ $1$
  and that \emph{one of the following sets of additional assumptions} is
  valid\textup{:}
  \begin{enumerate}[{\rm(a)}]
  \item The measure $\mu_-$ is \emph{finite}, and $A_k$ converge to $A_\infty$
    \emph{locally} in measure on $\Omega$.\label{item:lsc-dom-a}
  \item The measure $\mu_-$ satisfies an \emph{almost-strong} IC near $\infty$
    \emph{relative} to $\Omega$ with constant\/ $1$ in the sense that, for
    every $\eps>0$, there exists some $R_\eps\in{(0,\infty)}$ such that
    \begin{equation}\label{eq:strong-ic-at-infty-rel}
      \mu_-(A^+)\le\P(A,\Omega)+\eps
      \qq\text{for all }A\in\M(\R^n)\text{ with }|A\cap\B_{R_\eps}|=0\text{ and }|A|<\infty\,,
    \end{equation}
    and $A_k$ converge to $A_\infty$ \emph{locally} in measure on $\Omega$ with
    $|(A_k\Delta A_\infty)\cap\Omega|{+}\P(A_k,\Omega){+}\P(A_\infty,\Omega)<\infty$ for all
    $k\in\N$.\label{item:lsc-dom-b}
  \item The measure $\mu_-$ satisfies the small-volume IC \emph{relative} to
    $\Omega$ with constant\/ $1$ in the sense that, for every $\eps>0$, there is
    some $\delta>0$ such that
    \begin{equation}\label{eq:small-vol-ic-rel}
      \mu_-(A^+)\le\P(A,\Omega)+\eps
      \qq\text{for all }A\in\M(\R^n)\text{ with }|A|<\delta\,,
    \end{equation}
    and $A_k$ converge to $A_\infty$ \emph{globally} in measure on $\Omega$ with
    $\P(A_k,\Omega){+}\P(A_\infty,\Omega)<\infty$ for all $k\in\N$.\label{item:lsc-dom-c}
  \end{enumerate}
  If furthermore $\min\{\mu_+(A_k^1),\mu_-(A_k^+)\}<\infty$ holds for all
  $k\in\N$, then we have $\min\{\mu_+(A_\infty^1),\mu_-(A_\infty^+)\}<\infty$ and
  \begin{equation}\label{eq:lsc-dom}
    \liminf_{k\to\infty}\Pmu[A_k;\Omega]\ge\Pmu[A_\infty;\Omega]\,.
  \end{equation}
\end{thm}

Since the different cases in Theorem \ref{thm:lsc-dom} are still illustrated
well by the examples given in connection with Theorem \ref{thm:lsc-BV-dom}, we
now keep the discussion brief. Once more, the setting \eqref{item:lsc-dom-a}
concerns finite measures $\mu_-$, and this part of
Theorem \ref{thm:lsc-dom} will be deduced from the corresponding assertion for
finite-perimeter domains by a simple exhaustion argument, which closely
resembles the last step in the proof of Theorem \ref{thm:lsc-BV-dom} and
crucially draws on the finiteness of $\mu_-$. The improvements for infinite
measures provided by \eqref{item:lsc-dom-b} and \eqref{item:lsc-dom-c} involve
essentially the same relative ICs found in \eqref{eq:strong-ic-at-infty-rel-pre}
and \eqref{eq:small-vol-ic-rel-pre}. Despite this similarity, under
\eqref{item:lsc-dom-b} or \eqref{item:lsc-dom-c} with possibly infinite $\mu_-$
we cannot derive the result directly from the counterparts in Theorem
\ref{thm:lsc-BV-dom} by exhaustion, but rather will implement a deduction from
the result in the setting \eqref{item:lsc-dom-a} by cut-off arguments widely
analogous to the proof of Theorem \ref{thm:lsc-Dir}.

The difference between \eqref{item:lsc-dom-b} and \eqref{item:lsc-dom-c} can
again be underpinned with concrete examples: On one hand, the cases $n\ge2$,
$\mu=2\H^{n-1}\ecke(\R^{n-1}{\times}\{0\})$, $\Omega=\R^{n-1}{\times}{({-}1,1)}$
and $n=1$, $\mu=2\H^0\ecke\mathds{Z}$, $\Omega=\R$ are included in
\eqref{item:lsc-dom-c}, but not in \eqref{item:lsc-dom-b}. On the other hand,
both \eqref{item:lsc-dom-b} and \eqref{item:lsc-dom-c} apply in the cases
$n=2$, $\mu=2\H^1\ecke(\R{\times}\{0\})$, $\Omega=\{(x,y)\in\R^2\,:\,|y|<|x|\}$
and $n=1$, $\mu=2\H^0\ecke(2\mathds{Z}{-}1)$, $\Omega=\R{\setminus}2\mathds{Z}$,
where, however, only \eqref{item:lsc-dom-b} gives semicontinuity even with
respect to \emph{local} convergence in measure.

We also record in connection with both Theorem \ref{thm:lsc-BV-dom} and
Theorem \ref{thm:lsc-dom} and the corresponding examples:

\begin{rem}[on the settings of Theorems \ref{thm:lsc-BV-dom} and
    \ref{thm:lsc-dom}]\label{rem:rel-ic-enforces-finite}
  In Theorem \ref{thm:lsc-dom}, the settings \eqref{item:lsc-dom-b} and
  \eqref{item:lsc-dom-c} improve on \eqref{item:lsc-dom-a} only in the
  infinite-volume case $|\Omega|=\infty$, since indeed the IC from
  \eqref{eq:strong-ic-at-infty-rel} or \eqref{eq:small-vol-ic-rel} for a Radon
  measure $\mu_-$ on $\R^n$ and open $\Omega\subset\R^n$ together with
  $|\Omega|<\infty$ and $\mu_-\ecke(\Omega^+)^\c\equiv0$ already enforces
  finiteness of $\mu_-$.

  \begin{proof}
    In case $|\Omega|<\infty$ we may test \eqref{eq:strong-ic-at-infty-rel}
    with $\Omega\setminus\B_{R_1}$ to infer
    $\mu_-(\Omega^+\setminus\overline{\B_{R_1}})\le\mu_-((\Omega\setminus\B_{R_1})^+)
    \le\P(\Omega\setminus\B_{R_1},\Omega)+1\le\P(\B_{R_1})+1<\infty$. Similarly,
    if we fix $\delta>0$ such that \eqref{eq:small-vol-ic-rel} applies with
    $\eps=1$, then in view of $|\Omega|<\infty$ we have
    $|\Omega\setminus\B_{R_1}|<\delta$ for some suitably large
    $R_1\in{(0,\infty)}$, and by testing \eqref{eq:small-vol-ic-rel} with
    $\Omega\setminus\B_{R_1}$ we deduce exactly the same estimate. Clearly,
    taking into account local finiteness of $\mu_-$ and
    $\mu_-\ecke(\Omega^+)^\c\equiv0$, the estimate yields finiteness of $\mu_-$
    in both cases.
  \end{proof}

  Also in the earlier Theorem \ref{thm:lsc-BV-dom}, the settings
  \eqref{item:lsc-BV-dom-b} and \eqref{item:lsc-BV-dom-c} improve on
  \eqref{item:lsc-BV-dom-a} only in case $|\Omega|=\infty$. This follows by the
  same reasoning, which also works with \eqref{eq:strong-ic-at-infty-rel-pre}
  and \eqref{eq:small-vol-ic-rel-pre} in place of
  \eqref{eq:strong-ic-at-infty-rel} and \eqref{eq:small-vol-ic-rel}.
\end{rem}

Finally, let us point out that the additional \emph{relative} IC
\eqref{eq:small-vol-ic-rel} of Theorem \ref{thm:lsc-dom}\eqref{item:lsc-dom-c}
could in fact be required only near $\infty$ by adding a condition
$|A\cap\B_{R_\eps}|=0$, as it was included in all our settings of type
\eqref{item:lsc-dom-b}. However, while for the strong-type setting
\eqref{item:lsc-dom-b} the near-$\infty$ feature does win some generality, in
the small-volume setting \eqref{item:lsc-dom-c} an adaptation of Proposition
\ref{prop:ic-sing} shows that it does not, and therefore we have in fact decided
to stick to the formulation of Theorem \ref{thm:lsc-dom}\eqref{item:lsc-dom-c}
given above.

\smallskip

Now we proceed to the final semicontinuity proof of this paper.

\begin{proof}[Proof of Theorem \ref{thm:lsc-dom}]
  Throughout the proof we assume that $\lim_{k\to\infty}\Pmu[A_k;\Omega]$ exists
  and is finite. In addition, in view of $\mu_\pm\ecke\Omega^\c\equiv0$
  the values $\Pmu[A_k;\Omega]$, $\Pmu[A_\infty;\Omega]$ and all assumptions depend
  only on the portions $A_k\cap\Omega$ and $A_\infty\cap\Omega$ of $A_k$ and $A_\infty$. Hence
  we may and do assume
  \[
    A_k\subset\Omega\qq\text{and}\qq A_\infty\subset\Omega\,,
  \]
  which allows to rewrite the assumption $|(A_k\Delta A_\infty)\cap\Omega|<\infty$ of
  \eqref{item:lsc-dom-b} as $|A_k\Delta A_\infty|<\infty$ and to consider the global
  convergence on $\Omega$ in \eqref{item:lsc-dom-c} as global convergence on
  $\R^n$.

  In order to treat the situation \eqref{item:lsc-dom-a} we observe that the
  open set $\Omega$ can be exhausted by smooth open sets
  $\Omega_\ell\Subset\Omega$ with $\ell\in\N$ in the sense that
  $\Omega_\ell\subset\Omega_{\ell+1}$ for all $\ell\in\N$ and
  $\bigcup_{\ell=1}^\infty\Omega_\ell=\Omega$. Applying Theorem
  \ref{thm:lsc-BV-dom}\eqref{item:lsc-BV-dom-a} on $\Omega_\ell$ (which in
  particular satisfies $\Omega_\ell\in\BVs(\R^n)$ and $\Omega_\ell^1=\Omega_\ell$)
  with the measures $\mu_\pm\ecke\Omega_\ell$ we find
  \[
    \liminf_{k\to\infty}\big[\P(A_k,\Omega_\ell)+\mu_+(A_k^1\cap\Omega_\ell)-\mu_-(A_k^+\cap\Omega_\ell)\big]
    \ge\P(A_\infty,\Omega_\ell)+\mu_+(A_\infty^1\cap\Omega_\ell)-\mu_-(A_\infty^+\cap\Omega_\ell)\,.
  \]
  Using $\Omega_\ell\subset\Omega$ and elementary estimations we deduce
  \[
    \liminf_{k\to\infty}\Pmu[A_k;\Omega]+\mu_-(\Omega_\ell^\c)
    \ge\P(A_\infty,\Omega_\ell)+\mu_+(A_\infty^1\cap\Omega_\ell)-\mu_-(A_\infty^+)\,,
  \]
  from which we obtain the claim \eqref{eq:lsc-dom} in the generality of the
  situation \eqref{item:lsc-dom-a} by sending $\ell\to\infty$, by taking into
  account the pointwise monotone convergence of $\Omega_\ell$ to $\Omega$ and
  the assumption $\mu_\pm\ecke\Omega^\c\equiv0$, and finally by crucially
  exploiting the finiteness of $\mu_-$.

  In view of the analogy to the proof of Theorem
  \ref{thm:lsc-Dir}\eqref{item:lsc-Dir-b} we only sketch the arguments relevant
  for the present setting \eqref{item:lsc-dom-b}. As in the earlier proof, given
  an arbitrary $\eps>0$, we first choose a sequence of radii
  $R_i\in{(R_\eps,\infty)}$ with $\lim_{i\to\infty}R_i=\infty$ and pass to a
  subsequence of $(A_k)_{k\in\N}$ in order to ensure $\mu_-(\partial\B_{R_i})=0$ and
  $\lim_{k\to\infty}\H^{n-1}((A_k\Delta A_\infty)^+\cap\partial\B_{R_i})=0$. We then
  apply the already established part \eqref{item:lsc-dom-a} of the present
  theorem on $\Omega\cap\B_{R_i}$ with the \emph{finite} measures
  $\mu_\pm\ecke(\Omega\cap\B_{R_i})$, which inherit the small-volume IC from
  $\mu_\pm$, to infer
  \[
    \liminf_{k\to\infty}\big[\P(A_k,\Omega\cap\B_{R_i})+\mu_+(A_k^1\cap\B_{R_i})-\mu_-(A_k^+\cap\B_{R_i})\big]
    \ge\P(A_\infty,\Omega\cap\B_{R_i})+\mu_+(A_\infty^1\cap\B_{R_i})-\mu_-(A_\infty^+\cap\B_{R_i})\,.
  \]
  In order to estimate the terms cut off we follow closely the derivation
  around \eqref{eq:mu-at-infty} and \eqref{eq:Pmu-at-infty}, where now we take
  perimeters in the open domain $\Omega$ and rely on the relative version
  \eqref{eq:strong-ic-at-infty-rel} of the almost-strong IC in the form
  $\mu_-(((A_k\Delta A_\infty)\setminus\B_{R_i})^+)\le\P((A_k\Delta A_\infty)\setminus\B_{R_i},\Omega)+\eps$
  (which does apply, since $R_i\ge R_\eps$). Arguing as described we find that
  either the claim \eqref{eq:lsc-dom} holds trivially or we have
  $\mu_-(A_k^+){+}\mu_-(A_\infty^+)<\infty$ for all $k\in\N$ together with
  \begin{equation}\label{eq:P-dom-at-infty}
    \liminf_{k\to\infty}\big[\P(A_k,\Omega\setminus\B_{R_i})-\mu_-(A_k^+\setminus\B_{R_i})\big]
    \ge{-}\P(A_\infty,\Omega\setminus\B_{R_i})-\mu_-(A_\infty^+\setminus\B_{R_i})-\eps\,.
  \end{equation}
  By addition of the last two displayed equations and elementary estimation we
  arrive at
  \[
    \liminf_{k\to\infty}\Pmu[A_k;\Omega]
    \ge\P(A_\infty,\Omega\cap\B_{R_i})-\P(A_\infty,\Omega\setminus\B_{R_i})+\mu_+(A_\infty^1\cap\B_{R_i})-\mu_-(A_\infty^+)-\eps\,.
  \]
  Going to the limit $i\to\infty$ and using the arbitrariness of $\eps$, we
  obtain the claim \eqref{eq:lsc-dom} in the generality of the situation
  \eqref{item:lsc-dom-b}.

  The proof in the setting \eqref{item:lsc-dom-c} is an adaptation of the one in
  the setting \eqref{item:lsc-dom-b}, precisely as Theorem
  \ref{thm:lsc-Dir}\eqref{item:lsc-Dir-c} was obtained by adapting the argument
  given for Theorem \ref{thm:lsc-Dir}\eqref{item:lsc-Dir-b}. Indeed, for an
  arbitrary $\eps>0$, we can exploit $\lim_{k\to\infty}|A_k\Delta A_\infty|=0$ in order
  to apply the relative version \eqref{eq:small-vol-ic-rel} of the small-volume
  IC in the form
  $\mu_-(((A_k\Delta A_\infty)\setminus\B_{R_i})^+)\le\P((A_k\Delta A_\infty)\setminus\B_{R_i},\Omega)+\eps$
  at least for $k\gg1$. In the limit $k\to\infty$ we still arrive at the
  estimate \eqref{eq:P-dom-at-infty} and in conclusion can deduce the claim
  \eqref{eq:lsc-dom} also in the generality of the situation \eqref{item:lsc-dom-c}.
\end{proof}

We conclude this section by pointing out that, as it was on $\Omega=\R^n$, also
on arbitrary $\Omega$ the relative small-volume IC \eqref{eq:small-vol-ic-rel}
on $\mu_-$ is in fact optimal. This will go hand in hand with recording
further connections between the standard small-volume IC, its variant in
\eqref{eq:small-vol-ic-rel}, and semicontinuity properties of the functional,
and will now be explicated for the case $\mu_+\equiv0$, $\mu_-=\mu$:

\begin{rem}[on optimality of the relative IC \eqref{eq:small-vol-ic-rel} and 
    more connections between ICs and semicontinuity]
  {We here consider an open set\/ $\Omega\subset\R^n$ and a non-negative Radon
  measure $\mu$ on $\R^n$ with $\mu\ecke\Omega^\c\equiv0$.}
  \begin{enumerate}[{\rm(i)}]
  \item If\/ $\mathscr{P}_{0,\mu}[\,\cdot\,;\Omega]$ is lower semicontinuous on
    $\BVs(\Omega)$ with respect to global convergence in measure on $\Omega$, then
    for every $\eps>0$, there is some $\delta>0$ such that
    \eqref{eq:small-vol-ic-rel} holds for $\mu$, that is,
    $\mu(A^+)\le\P(A,\Omega)+\eps$ for all $A\in\M(\R^n)$ with $|A|<\delta$.

    \begin{proof}
      If \eqref{eq:small-vol-ic-rel} fails for some $\eps>0$ and all $\delta>0$, in
      particular, for each $k\in\N$, there exists $A_k\in\M(\R^n)$ with
      $|A_k|<\frac1k$ and $\mu(A_k^+)>\P(A_k,\Omega)+\eps$. However, then
      $A_k\in\BVs(\Omega)$ converge to $\emptyset$ in measure on $\Omega$ with
      $\mathscr{P}_{0,\mu}[A_k;\Omega]<{-}\eps$, and
      $\mathscr{P}_{0,\mu}[\,\cdot\,;\Omega]$ is not lower semicontinuous.
    \end{proof}

    Thus, at least in case $\mu_+\equiv0$ the assumption
    \eqref{eq:small-vol-ic-rel} on $\mu_-$ in Theorem
    \ref{thm:lsc-dom}\eqref{item:lsc-dom-c} is also necessary for lower
    semicontinuity of\/ $\mathscr{P}_{0,\mu_-}[\,\cdot\,;\Omega]$ and thus
    optimal.\label{item:optimality-doms}
  \item Consider the following assertions\footnote{Here, for the local-convergence
    semicontinuity \eqref{item:equiv-local-lsc}, we need to restrict to subclasses of
    $\BVs(\Omega)$ which exclude convergence of $A_k$ to $A$ with
    $|(A_k\Delta A)\cap\Omega|=\infty$ for all $k\in\N$. In contrast, the
    global-convergence statement \eqref{item:equiv-global-lsc} could
    equivalently be stated on all of $\{A\in\M(\R^n):\P(A,\Omega)<\infty\}$,
    since global convergence of $A_k$ to $A$ anyway yields
    $|(A_k\Delta A)\cap\Omega|<\infty$ for $k\gg1$.}\textup{:}
    \begin{enumerate}[{\rm(1)}]\renewcommand\theenumi{}
    \item  The measure $\mu$ is finite and satisfies the small-volume IC in $\R^n$
      with constant\/ $1$.\label{item:equiv-finite+ic}
    \item For every $A_0\in\M(\R^n)$ with $\P(A_0,\Omega)<\infty$, the functional
      $\mathscr{P}_{0,\mu}[\,\cdot\,;\Omega]$ is lower semicontinuous on
      $\{A\in\M(\R^n)\,:\,A\Delta A_0\in\BVs(\Omega)\}$ with respect
      to \emph{local} convergence in measure on $\Omega$.\label{item:equiv-local-lsc}
    \item For every $A_0\in\M(\R^n)$ with $\P(A_0,\Omega)<\infty$, the functional
      $\mathscr{P}_{0,\mu}[\,\cdot\,;\Omega]$ is lower semicontinuous on
      $\{A\in\M(\R^n)\,:\,A\Delta A_0\in\BVs(\Omega)\}$ with respect
      to \emph{global} convergence in measure on $\Omega$.\label{item:equiv-global-lsc}
    \item The functional $\mathscr{P}_{0,\mu}[\,\cdot\,;\Omega]$ is lower
      semicontinuous on $\BVs(\Omega)$ with respect
      to \emph{global} convergence in measure on $\Omega$.\label{item:equiv-global-lsc-0}
    \item For every $\eps>0$, there is some $\delta>0$ such that $\mu$ satisfies
      small-volume IC \eqref{eq:small-vol-ic-rel} \emph{relative} to
      $\Omega$.\label{item:equiv-rel-ic}
    \end{enumerate}

    Then, we claim that the implications \eqref{item:equiv-finite+ic}$\implies$%
    \eqref{item:equiv-local-lsc}$\implies$\eqref{item:equiv-global-lsc}$\iff$%
    \eqref{item:equiv-global-lsc-0}$\iff$\eqref{item:equiv-rel-ic} are generally
    valid. Indeed, \eqref{item:equiv-finite+ic}$\implies$\eqref{item:equiv-local-lsc}
    holds by Theorem \ref{thm:lsc-dom}\eqref{item:lsc-dom-a}, the implications
    \eqref{item:equiv-local-lsc}$\implies$\eqref{item:equiv-global-lsc}$\implies$%
    \eqref{item:equiv-global-lsc-0} are trivial,
    \eqref{item:equiv-global-lsc-0}$\implies$\eqref{item:equiv-rel-ic} has been
    established in the preceding point \eqref{item:optimality-doms}, and
    \eqref{item:equiv-rel-ic}$\implies$\eqref{item:equiv-global-lsc} holds by 
    Theorem \ref{thm:lsc-dom}\eqref{item:lsc-dom-c}.\label{item:five-prop-doms}

    We could in fact formulate even more equivalent statements, for instance, one
    such statement is given by the localized IC variant of Lemma
    \ref{lem:bdry-local-ic}\eqref{item:ic-local} together with finiteness of $\mu$.
  \item In general, the implication
    \eqref{item:equiv-finite+ic}$\implies$\eqref{item:equiv-local-lsc} from point
    \eqref{item:five-prop-doms} cannot be reversed. To see this, for $n\ge2$,
    we consider $\mu=2\H^{n-1}\ecke(\R^{n-1}{\times}\{0\})$ on $\Omega=\R^n$ or
    alternatively $\mu=\H^{n-1}\ecke(\R^{n-1}{\times}\{0\})$ on any open
    $\Omega\subset\R^n$ with $\R^{n-1}{\times}{[0,\infty)}\subset\Omega$. Then,
    it can be checked that $\mu$ satisfies
    \eqref{eq:strong-ic-at-infty-rel}. Thus, Theorem
    \ref{thm:lsc-dom}\eqref{item:lsc-dom-b} gives the validity of
    \eqref{item:equiv-local-lsc}, while \eqref{item:equiv-finite+ic} fails in view
    of the infiniteness of $\mu$. \ka The specific case $n=1$ is different, and
    for this case one can in fact show that the validity of
    \eqref{item:equiv-local-lsc} requires finiteness of\/ $\mu$ and that
    \eqref{item:equiv-finite+ic}$\iff$\eqref{item:equiv-local-lsc} holds.\kz

    Also the implication
    \eqref{item:equiv-local-lsc}$\implies$\eqref{item:equiv-global-lsc} cannot be
    reversed in general. Here, for $n\ge2$ we consider the infinite measure
    $\mu=2\H^{n-1}\ecke(\R^{n-1}{\times}\{0,1\})$ on any open $\Omega\subset\R^n$
    with $\dist(\R^{n-1}{\times}\{0,1\},\Omega^\c)>0$. Then, by adapting
    the proof of Proposition \ref{prop:inf-model-meas} one checks that $\mu$
    satisfies \eqref{eq:small-vol-ic-rel} for all these $\Omega$. Hence,
    Theorem \ref{thm:lsc-dom}\eqref{item:lsc-dom-c} gives the validity of
    \eqref{item:equiv-global-lsc}, while
    $A_k\coleq{[k,k{+}n]^{n-1}}{\times}{[0,1]}\in\BVs(\R^n)$ converge locally in
    measure on $\Omega$ to $\emptyset$ with $\mathscr{P}_{0,\mu}[A_k;\Omega]
    \le\mathscr{P}_{0,\mu}[A_k;\R^n]={-}2n^{n-2}<0$ and thus demonstrate that
    \eqref{item:equiv-local-lsc} fails in this case. For $n=1$, the same
    phenomenon occurs for $\mu=2\H^0\ecke\mathds{Z}$ on any open
    $\Omega\subset\R$ with
    $\dist(\mathds{Z},\Omega^\c)>0$.\label{item:equiv-expls}
  \item However, if we impose as an additional assumption
    \[
      \text{either }|\Omega|<\infty\text{ or }\mu(\Omega)<\infty\,,
    \]
    it turns out that the \textbf{five assertions} of point
    \eqref{item:five-prop-doms} are in fact \textbf{all equivalent}. In order to
    justify this claim we recall from Remark \ref{rem:rel-ic-enforces-finite}
    that \eqref{eq:small-vol-ic-rel} and $|\Omega|<\infty$ together enforce
    finiteness of $\mu$. Since moreover
    \eqref{eq:small-vol-ic-rel} is stronger than the usual small-volume IC, this
    means that under the additional assumption we also have the backwards
    implication \eqref{item:equiv-rel-ic}$\implies$\eqref{item:equiv-finite+ic}.

    In particular, we record that for the \ka counter\kz""examples of point
    \eqref{item:equiv-expls} it was inevitable to have both $|\Omega|=\infty$
    and $\mu(\Omega)=\infty$.
  \end{enumerate}  
\end{rem}

\begin{appendix}
  
\section{Isoperimetric conditions for infinite model measures}

In this appendix we justify the validity of ICs for basic infinite model
measures concentrated on hyperplanes by suitable capacity computations. We start
with an auxiliary lemma, which determines the $1$-capacity of sets in a
hyperplane and is not at all surprising. Still, since we are not aware of a
custom-fit reference for this statement, we also include a proof.

\begin{lem}[$1$-capacity on hyperplanes]\label{lem:Cp1-hyperplane}
  For $n\ge2$, every $S\in\Bo(\R^{n-1})$, and $t\in\R$, we have
  \[
    \Cp_1(S{\times}\{t\})=2|S|\,.
  \]
  In different words, this means $\Cp_1(A)=2\H^{n-1}(A)$ for every
  $A\in\Bo(\R^{n-1}{\times}\{t\})$ with\/ $t\in\R$.
\end{lem}

\begin{proof}
  We prove the inequalities \glqq$\le$\grqq{} and \glqq$\ge$\grqq{} separately.
  
  We consider an open $U\in\BVs(\R^{n-1})$ and the open cylinder
  $U^\delta\coleq U{\times}{(t{-}\delta,t{+}\delta)}$ with $\delta>0$. One
  verifies $|U^\delta|=2\delta|U|<\infty$,
  $U{\times}\{t\}\subset U^\delta\subset(U^\delta)^+$,
  and $\P(U^\delta)=2|U|{+}2\delta\P(U)$. Therefore, Proposition
  \ref{prop:inf-P} gives
  $\Cp_1(U{\times}\{t\})\le\Cp_1(U^\delta)\le2|U|{+}2\delta\P(U)$ for
  arbitrary $\delta>0$, and we get $\Cp_1(U{\times}\{t\})\le2|U|$. Now, an
  arbitrary open set in $\R^{n-1}$ is the union of an increasing sequence of
  bounded open sets with smooth boundaries, thus in particular of open sets from
  $\BVs(\R^{n-1})$. (This claim can be proved essentially by mollifying $\1_K$
  with compact $K\subset U$ and then choosing good superlevel sets of the
  mollifications via Sard's theorem.) By \cite[Theorem 4.15(viii)]{EvaGar15}
  one can pass to the limit along such a sequence to deduce that
  $\Cp_1(U{\times}\{t\})\le2|U|$ stays valid for
  arbitrary open $U\subset\R^{n-1}$. For arbitrary $S\in\Bo(\R^{n-1})$, one
  then concludes
  \[
    \Cp_1(S{\times}\{t\})
    \le\inf\{\Cp_1(U{\times}\{t\}):U\text{ open in }\R^{n-1}\,,\,S\subset U\}
    \le\inf\{2|U|:U\text{open in }\R^{n-1}\,,\,S\subset U\}
    =2|S|\,.
  \]
  
  From Definition \ref{def:Cp1} one obtains in a standard way (essentially by
  mollification and multiplication with cut-off functions) the equality
  \begin{equation}\label{eq:Cp1-on-cpct}
    \Cp_1(K)=\inf\bigg\{\int_{\R^n}|\nabla\eta|\dx\,:\,
    \eta\in\C^\infty_\cpt(\R^n)\,,\,\eta\ge1\text{ on }K\bigg\}
    \qq\text{for compact }K\subset\R^n\,.
  \end{equation}
  Now, if $H$ is compact in $\R^{n-1}$, for every $\eta\in\C^\infty_\cpt(\R^n)$
  with $\eta\ge1$ on $H{\times}\{t\}$, one has
  \[
    \int_{\R^n}|\nabla\eta|\dx
    \ge\int_H\bigg[\,\bigg|\int_{-\infty}^t\partial_n\eta(x',x_n)\dx_n\bigg|
      +\bigg|\int_t^\infty\partial_n\eta(x',x_n)\dx_n\bigg|\,\bigg]\dx'
    =\int_H2|\eta(x',t)|\dx'\ge2|H|\,,
  \]
  and by \eqref{eq:Cp1-on-cpct} this implies $\Cp_1(H{\times}\{t\})\ge2|H|$. For
  arbitrary $S\in\Bo(\R^{n-1})$, one then concludes
  \[
    \Cp_1(S{\times}\{t\})
    \ge\sup\{\Cp_1(H{\times}\{t\}):H\text{ compact}\,,\,H\subset S\}
    \ge\sup\{2|H|:H\text{ compact}\,,\,H\subset S\}=2|S|\,,
  \]
  which completes the proof.
\end{proof}

The following results now identify two infinite measures, which satisfy the
strong IC with constant $1$ and the small-volume IC with constant $1$,
respectively.

\begin{prop}[strong IC for $\H^{n-1}$ on a single hyperplane]\label{prop:inf-model-meas-pre}
  For $n\ge2$, the non-negative Radon measure
  \[
    \mu\coleq2\H^{n-1}\ecke(\R^{n-1}{\times}\{0\})
    \qq\text{on }\R^n
  \]
  satisfies the strong IC in $\R^n$ with constant\/ $1$.
\end{prop}

\begin{proof}
  For $A\in\BVs(\R^n)$, from Lemma \ref{lem:Cp1-hyperplane} and Proposition
  \ref{prop:inf-P} we obtain
  \[
    \mu(A^+)=2\H^{n-1}(A^+\cap(\R^{n-1}{\times}\{0\}))
    =\Cp_1(A^+\cap(\R^{n-1}{\times}\{0\}))\le\P(A)\,.
  \]
  Since the resulting estimate trivially holds in case $\P(A)=\infty$ as well,
  we have verified the claimed IC.
\end{proof}

\begin{prop}[small-volume IC for $\H^{n-1}$ on two parallel hyperplanes]
    \label{prop:inf-model-meas}
  For $n\ge2$, the non-negative Radon measure
  \[
    \mu\coleq2\H^{n-1}\ecke(\R^{n-1}{\times}\{0,1\})
    \qq\text{on }\R^n
  \]
  satisfies the small-volume IC in $\R^n$ with constant\/ $1$, and more precisely
  we have in fact
  \[
    \mu(A^+)\le\P(A)+2|A|
    \qq\text{for all }A\in\M(\R^n)\,.
  \]
\end{prop}

\begin{proof}
  The validity of the IC follows by combining Proposition
  \ref{prop:inf-model-meas-pre} and Proposition \ref{prop:ic-sing}. However, we
  now carry out an alternative and self-contained proof, which also yields the
  explicit estimate claimed. Clearly we can assume $A\in\BVs(\R^n)$. In view of
  $\int_0^1\H^{n-1}(A^+\cap(\R^{n-1}{\times}\{t\}))\,\d t\le|A^+|=|A|$ we can
  find and fix some $t\in{(0,1)}$ with
  \[
    \H^{n-1}(A^+\cap(\R^{n-1}{\times}\{t\}))\le|A|\,.
  \]
  Introducing $A_0\coleq A\cap(\R^{n-1}{\times}{({-}\infty,t)})$ with 
  $|A_0|\le|A|<\infty$, by an application\footnote{If we stick to the precise
  statement of Lemma \ref{lem:P(AcapB),P(A-S)}, then in view of
  $\P(\R^{n-1}{\times}{({-}\infty,t)})=\infty$ we cannot use \eqref{eq:P(AcapB)}
  directly for $\R^{n-1}{\times}{({-}\infty,t)}$ and $G=\R^n$, but clearly we can
  circumvent this by applying \eqref{eq:P(AcapB)} with $G=\B_R$ first and then
  passing $R\to\infty$.} of \eqref{eq:P(AcapB)} we get
  \[
    \P(A_0)
    \le\P(A,\R^{n-1}{\times}{({-}\infty,t)})+\H^{n-1}(A^+\cap(\R^{n-1}{\times}\{t\}))
    \le\P(A,\R^{n-1}{\times}{({-}\infty,t)})+|A|\,.
  \]
  Via Lemma \ref{lem:Cp1-hyperplane} and Proposition \ref{prop:inf-P} (the
  latter applied in view of $A^+\cap(\R^{n-1}{\times}\{0\})\subset A_0^+$) we
  infer
  \[
    2\H^{n-1}(A^+\cap(\R^{n-1}{\times}\{0\}))
    =\Cp_1(A^+\cap(\R^{n-1}{\times}\{0\}))
    \le\P(A_0)
    \le\P(A,\R^{n-1}{\times}{({-}\infty,t)})+|A|\,.
  \]
  With the help of $A_1\coleq A\cap(\R^{n-1}{\times}{(t,\infty)})$, we
  analogously obtain the estimate
  \[
    2\H^{n-1}(A^+\cap(\R^{n-1}{\times}\{1\}))
    =\Cp_1(A^+\cap(\R^{n-1}{\times}\{1\}))
    \le\P(A_1)
    \le\P(A,\R^{n-1}{\times}{(t,\infty)})+|A|\,.
  \]
  Adding up the two estimates gives the claim $\mu(A^+)\le\P(A)+2|A|$, from which
  the IC is immediate.
\end{proof}

We remark that the preceding propositions formally extend to the case $n=1$,
where they correspond to the much simpler estimates $2\delta_0(A^+)\le\P(A)$ for
$A\in\Bo(\R)$ with $|A|<\infty$ and $2(\delta_0{+}\delta_1)(A^+)\le\P(A){+}2|A|$
for arbitrary $A\in\Bo(\R)$, with the Dirac measures $\delta_0$ and $\delta_1$
at $0$ and $1$. However, the measures $\delta_0$ and $\delta_0{+}\delta_1$ are
clearly finite, and indeed, for $n=1$, measures with strong IC are necessarily
finite, while the small-volume IC with constant $1$ still admits infinite
examples such as the measure
$2\H^0\ecke\mathds{Z}=2\sum_{z\in\mathds{Z}}\delta_z$, for instance.

\end{appendix}

\bibliographystyle{amsplain}
\bibliography{ic_lsc}

\end{document}